\documentclass[10pt,a4paper]{article}

\usepackage{amsmath}
\usepackage{amsfonts}

\usepackage{amssymb}
\usepackage{amsrefs}
\usepackage{latexsym}
\usepackage{amscd}
\usepackage{amsthm}
\usepackage[normalem]{ulem}
\usepackage{amsfonts}
\usepackage{mathrsfs}
\usepackage{ifthen}
\usepackage{mleftright}
 \usepackage[all]{xy}
 \usepackage{color}

\usepackage{braket}
\usepackage{comment} 
\usepackage{shadethm}
\usepackage[capitalize]{cleveref} 

\newtheorem{thm}{\bf Theorem}[section]
\newtheorem{lem}[thm]{\bf Lemma}
\newtheorem{cor}[thm]{\bf Corollary}
\newtheorem{prop}[thm]{\bf Proposition}
\theoremstyle{definition}
\newtheorem{defn}[thm]{\bf Definition}
\theoremstyle{remark}
\newtheorem{rem}[thm]{\bf Remark}
\newtheorem{ass}[thm]{\bf Assumption}

\newtheorem{prob}[thm]{\bf Problem}

\newtheorem{ex}[thm]{\bf Example}

\newtheorem{ques}[thm]{\bf Question}

\newcommand{\A}{\mathcal A}
\newcommand{\B}{\mathcal B}
\newcommand{\G}{\mathcal G}
\newcommand{\su}{\mathfrak{su}(2)}

\newcommand{\bb}{\mathbb}
\newcommand{\cs}{\mathit{cs}}

\def\ri{\rightarrow}
\newcommand{\pr}{\operatorname{pr}}

\newcommand{\al}{\alpha}

\def\Om{\Omega}

\def\cal{\mathcal}
\def\mdeg{\operatorname{mdeg}}

\newcommand{\wt}{\widetilde}
\newcommand{\R}{\mathbb R}
\newcommand{\Z}{\mathbb Z}
\def\q{{\mathbb Q}}

\newcommand{\co}{\mathbb C}

\def\bb{\mathbb}

\def\ker{\operatorname{Ker}}
\def\im{\operatorname{Im}}
\def\ad{\operatorname{ad}}
\def\ind{\operatorname{ind}}
\def\rank{\operatorname{Rank}}
\def\grad{\operatorname{grad}}
\def\End{\operatorname{End}}

\def\hol{\operatorname{Hol}}
\def\aut{\operatorname{Aut}}

\def\id{\text{Id}}

\def\inner<#1>{\langle #1 \rangle}
\def\:{\colon}
\def\trace{\operatorname{Tr}}

\def\hom{\operatorname{Hom}}

\def\dis{\displaystyle}

\begin{document}

\title{Seifert hypersurfaces of 2-knots and Chern-Simons functional}

\author{Masaki Taniguchi}
\date{}
\setcounter{tocdepth}{2}

\maketitle

\begin{abstract}
For a given smooth $2$-knot in $S^4$, we relate the existence of a smooth Seifert hypersurface of a certain class to the existence of irreducible $ SU(2)$-representations of its knot group. For example, we see that any smooth $2$-knot having the Poincar\'e homology $3$-sphere as a Seifert hypersurface has at least four irreducible $SU(2)$-representations of its knot group. This result is false in the topological category. The proof uses a quantitative formulation of instanton Floer homology. Using similar techniques, we also obtain similar results about codimension-$1$ embeddings of homology $3$-spheres into closed definite $4$-manifolds and a fixed point type theorem for instanton Floer homology.

\end{abstract}

\tableofcontents

\section{Introduction}\label{Introduction}
\subsection{Chern-Simons functional for oriented  2-knots}\label{Seifert}
A $2$-knot is a smooth embedding from $S^2$ into $S^4$. The classification problem of isotopy classes of $2$-knots has been studied since the introduction of the subject by Artin in 1925(\cite{Ar25}). There are several diagrammatic approaches to the study of 2-knots including motion pictures (\cite{Fo62}), surface knot diagrams (\cite{Ro98}), and chart diagrams (\cite{K17}), as well as invariants of 2-knots, including (twisted) Alexander polynomials (\cite{Al28}) and quandle cocycle invariants (\cite{K17}).   We focus, instead, on {\it Seifert hypersurfaces} of $2$-knots. 
\begin{defn}Let $K$ be an oriented $2$-knot in $S^4$. We call a closed oriented connected $3$-manifold $Y$ a (smooth) {\it Seifert hypersurface of} $K$ if there exists a smooth embedding $f: Y \setminus B^3 \to S^4$ such that $f|_{\partial ( Y \setminus B^3)} =K$ as oriented manifolds, where $B^3$ is a small open $3$-ball.  
\end{defn}
We consider the following problem: what are the Seifert hypersurfaces for a given $2$-knot? For a given $1$-dimensional knot $k$ in $S^3$, topological types of Seifert surfaces are determined by the Seifert genus $g(k)$ of $k$. In \cite{OS04}, Ozsv\'{a}th and Szab\'{o} proved that the Seifert genus of a 1-knot can be computed from its knot Floer homology. However, in the case of $2$-knots, the detection of topological types of Seifert hypersurfaces remains an open problem even for the unknot. One difficulty comes from the difference between smooth and topological Seifert hypersurfaces. For example, the Poincar\'e homology $3$-sphere is a Seifert hypersurface of the unknot in the topological category but not in the smooth category. Our main result relates the existence of smooth Seifert hypersurfaces of a certain class to the existence of irreducible $SU(2)$-representations of its knot group. The main result is proved by using a {\it quantitative formulation} of instanton Floer homology. 
In order to state our main result, we introduce maps \footnote{The invariants $\{ \cs_{K,j} \}$ can be defined for every oriented null-homologous $2$-knot $K$ embedded into a fixed closed oriented $4$-manifold $X$. See \cref{general cs_K}.} 
 \[
 \cs_{K,j} : R(K,j):= \hom (G_j (K), SU(2)) / SU(2)  \to (0,1], 
 \]
 where the group $G_j (K)$ is the kernel of the composite homomorphism
\begin{align}\label{psij}
\psi_j:  G(K):= \pi_1(S^4 \setminus K)  \xrightarrow{\text{Ab}} H_1(S^4 \setminus K; \Z) \cong  \Z \xrightarrow{\mod j } \Z/ j\Z
\end{align}
and the action of $SU(2)$ on $\hom (G_j (K), SU(2))$ is given by the conjugation. The map $\cs_{K,j}$ is an analog of the Chern-Simons functional; see \cref{cs for k}. 
The set $\im \cs_{K,j } \subset (0,1]$ is finite and is an invariant of the pair $(K,j)$.
 First, we provide several fundamental properties of $\{\cs_{K,j} \colon R(K,j) \to (0,1]\}_{j \in \Z_{>0}}$ containing a relationship to the Chern-Simons functionals of Seifert hypersurfaces:  
 \begin{prop} \label{2-knot1} The functionals $\{\cs_{K,j} \colon R(K,j) \to (0,1]\}_{j \in \Z_{>0}}$ satisfy the following conditions: 
 \begin{enumerate}
 \item Let $Y$ be a Seifert hypersurface of a given oriented $2$-knot $K$. Then,
\[
\im\cs_{K,j} \subset \im\cs_Y
\]
holds for any $j \in \Z_{>0}$, where $\im\cs_Y$ is given by
 \[
\im\cs_Y := \set  {\cs_Y(\rho)  | \rho: \text{a flat $SU(2)$-connection on $Y$} }\cap (0,1] 
 \]
and $\cs_Y$ is the $SU(2)$-Chern-Simons functional for $Y$. Moreover, if $Y$ is a Seifert $3$-manifold, then 
\[
\im\cs_{K,j} \subset \q \cap (0,1].
\]
\item For any $j \in \Z_{>0}$, 
\[
 \im\cs_{K_1,j} \cup \im\cs_{K_2,j}  \subset \im\cs_{K_1\# K_2,j}.
\]
\item For all positive integers $ m$ and $j$,
\[
\im\cs_{K,j}\subset  \im\cs_{K, mj}. 
\] 
\item The relation between $\im \cs_{K, j}$ and $\im \cs_{-K, j}$ is given by 
\[
\im \cs_{ K, j } = \set { 1-r  | r \in   \im\cs_{-K,j} \cap (0,1)  } \cup \{1\}. 
\]
If $ K$ is reversible ( i.e. $K$ is isotopic to $-K$), then $1-r \in  \im\cs_{K,j}$  for any $r \in \im\cs_{K,j} \cap (0,1)$.
\item Suppose $\im\cs_{K,j} \cap (0,1)$ is non-empty for $j \in \Z_{>0}$. Then there exist $ 2\#(\im\cs_{K,j} \cap (0,1))$ $SU(2)$-irreducible representations of $G_j(K)$.
 \end{enumerate}
 \end{prop}
 We first calculate $\im \cs_{ K, j }$ for {\it ribbon 2-knots}. Here a ribbon $2$-knot means a $2$-knot obtained as the boundary of the union of disjoint embedded 3-disks in $\R^4$ with some number of disjoint $3$-dimensional $1$-handles attached. (For more details, see \cite{Y69} and \cite[Subsection 5.6]{K17}. ) 
Property (1) in \cref{2-knot1} implies the following.  
 \begin{cor}\label{ribbon}If $K$ is a ribbon 2-knot, then $\im\cs_{K,j}= \im\cs_{U,j} =\{1\}$ for any $j\in \Z_{>0}$, where $U$ is the 2-unknot.
 \end{cor}

Next, we give calculations of $\{\cs_{K, j}\}_{j\in \Z_{>0}}$ for a certain class of twisted spun knots. 
 \begin{prop}\label{cal2}
 Let $T(p,q)$ be the $(p,q)$-torus knot, $M(p,q,r)$ the Montesinos knot of type $(p,q,r)$ for a pairwise relative prime tuple $(p,q,r)$ of positive integers, and $k(p/q)$ any $2$-bridge knot such that $\Sigma^2(k (p/q) ) = L(p,q)$, where $\Sigma^2(k)$ is the double branched cover of $k \subset S^3$.
\begin{enumerate}
\item For any $m \in \Z_{>0}$ and $j \in  \Z_{>0}$, 
 \[
\im\cs_{K(T(p,q),m), j} =\im\cs_{\Sigma(p,q,m)}  , 
 \]
where $K(k, m)$ is the twisted $m$-spun knot of the knot $k$. For the definition of the twisted $m$-spun knot, see \cite[Subsection 6.1]{K17}. 
  \item For any $j \in  \Z_{>0}$, 
 \[
\im\cs_{K(M(p,q,r),2), j} = \im\cs_{\Sigma(p,q,r)}  . 
 \]
 \item If $p$ is odd and satisfies the condition
 \[
 \left\{ s \in \set{ 2, \cdots , p-2} \middle | \frac{s^2-1} {p} \in \Z \right\} = \emptyset, 
 \]
  then, for any $j \in  \Z_{>0}$, 
 \[
 \im\cs_{K(k (p/q) ,2), j} = \left\{ -\frac{n^2r}{p}  \mod 1  \middle| 0 \leq n \leq \left\lceil  \frac{p}{2} \right\rceil  \right\},
 \] 
 where $r$ is any integer satisfying $qr\equiv -1 \mod p$ and $\lceil - \rceil$ is the ceiling function. 
  \end{enumerate}
 \end{prop}
 In \cite{FS90}, Fintushel-Stern gave an algorithm to compute $\im\cs_{\Sigma(p,q,r)} $ when $\Sigma(p,q,r)$ is a homology $3$-sphere. We give explicit calculations of $\im\cs_{K,j}$ for several $2$-knots in Yoshikawa's table in \cite{Y94} and twisted spun $2$-knots of $3_1$. 
 \begin{ex}We calculate $\im \cs_{K,j}$ for $8_1$, $10_1$ and $10_2$ in Yoshikawa's table (\cite{Y94}) and $k$-twisted spun $2$-knots of $3_1$.
 \begin{itemize}
 \item The $2$-knots $8_1$ and $10_1$ are spun $2$-knots $K(3_1, 0)$ and $K(4_1,0)$. It is known that spun knots $K(k,0)$ are ribbon. \footnote{It is known that $K(k,0)$ admits a surface diagram without triple intersections. For further details, see \cite{Sat02}. Such a 2-knot is known to be ribbon (\cite[Subsection 4.5]{K17}). } Therefore, $\im\cs_{8_1,j} = \im\cs_{10_1, j}= \{ 1\}$ for any $j\in \Z_{>0}$.
 \item The $2$-knot $10_2$ is the $2$-twisted spun knot $K(3_1, 2)$ of $3_1$. The $2$-knot $K(3_1, 2)$ has $\Sigma(2,2,3) = L(3,1)$ as a Seifert hypersurface so we have $\im\cs_{10_2, j} = \left\{ \frac{2}{3}, 1\right\}$ for any $j \in \Z_{>0}$ (\cite{KiKl90}).
 \item Since $3_1 =T(2,3)$, $-K(3_1, 6k-1)$ has $-\Sigma(2,3,6k-1)$ as a Seifert hypersurface for each $k\in \Z_{>0}$. Thus,
 \[
 \im \cs_{-K(3_1, 6k-1)}  = \left\{  \frac{12( 3k^2 - k + 3 l^2) +1
 }{24 (6k-1)}  \mod 1  \middle |\  l \in \{k, \cdots 5k-1\} \cap {2 \Z} \right\} \cup \{1\}.
 \]
 For example, $\im cs_{-K(3_1,5)} = \{ 1/120, 49/120, 1\}$.
 \end{itemize}
 \end{ex}
 
  In order to state the main theorem, we also need two kinds of Floer theoretic invariants: 
 \begin{itemize}
\item[(i)] In \cite{D18}, Daemi introduced a sequence of invariants $\Gamma_Y(k) \in [0, \infty]$ of an oriented homology $3$-sphere $Y$ parametrized by $k \in \Z$. In \cite{NST19}, Nozaki, Sato and the author introduced similar invariants $r_s(Y)\in (0, \infty]$ of oriented homology 3-spheres parametrized by $s \in [-\infty, 0]$. These invariants are defined by using a quantitative formulation of instanton Floer homology. 

\item[(ii)]  For an oriented homology $3$-sphere $Y$, $s \in [-\infty, 0]$ and $k \in \Z_{>0}$, we introduce invariants \[
l^s_{Y} ,  l^k_Y \in \Z_{>0} \cup \{\infty\}\text{ and } l_Y \in \Z_{>0} \cup \{\infty\}
\]
which satisfy the inequality 
 \[
 \max_{s \in  [-\infty, 0], k \in \Z_{>0}} \{l^s_Y ,  l^k_Y\} \leq l_Y.
 \]
If $Y$ is a Seifert homology $3$-sphere, $l_Y$ coincides with $2|\lambda (Y) |$, where $\lambda(Y)$ is the Casson invariant of $Y$.
 \end{itemize}
In terms of $r_s(Y)$, $\Gamma_Y(k)$, $l^s_Y$ and $l^k_Y$, our main result\footnote{\cref{2-knot2} can be generalized to any oriented $2$-knot $K$ embedded into a fixed closed negative definite 4-manifold $X$ with $0=[K] \in H_2(X; \Z)$.} is as follows. 
 \begin{thm}\label{2-knot2}
  Let $Y$ be an oriented homology $3$-sphere and $K$ an oriented $2$-knot.   \begin{enumerate}
\item 
   Suppose that $l^s_{Y}<\infty$ and $r_s(Y)<\infty$ for some $s \in [-\infty, 0]$.
If $Y$ is a Seifert hypersurface of $K$, then 
\[
 r_s(Y)-  \lfloor r_s(Y) \rfloor  \in \bigcup_{1 \leq j \leq l^s_{Y}} \im\cs_{K,j},  
\]
where  $\lfloor x  \rfloor := \begin{cases} \max \{ n \in \Z | n  \leq x\} \text{ if } x\notin \Z \\ 
x-1 \text{ if } x \in \Z \\
\end{cases}$. \footnote{Note that our convention of $\lfloor -  \rfloor $ is different from the usual floor function. }

\item Suppose that $l^k_{Y}<\infty$ and $\Gamma_{-Y}(k)<\infty$ for some $k \in \Z_{>0}$.
If $Y$ is a Seifert hypersurface of $K$, then 
\[
 \Gamma_{-Y} (k)- \lfloor  \Gamma_{-Y} (k)\rfloor   \in \bigcup_{1 \leq j \leq l^k_{Y}} \im\cs_{K,j} . 
\]
\end{enumerate}
 \end{thm}
 We also prove a sufficient condition for $l_Y< \infty$.
\begin{thm} \label{finiteness1}If the Chern-Simons functional of $Y$ is {\it Morse-Bott}\footnote{For the definition of the Morse-Bott property, see \cref{per}.}, then $l_Y<\infty$. 
\end{thm}
For example, the Chern-Simons functionals of finite connected sums of Seifert homology $3$-spheres are Morse-Bott. A sufficient condition for $r_s(Y)<\infty$ and $\Gamma_{-Y}(1)<\infty$ is given by $h(Y) <0$ (See \cite{D18}, \cite{NST19}), where $h(Y)$ is the Fr\o yshov invariant of $Y$ (\cite{Fr02}). \cref{2-knot2} gives a relation between Seifert hypersurfaces and $SU(2)$-representations of $G_j(K)$ in the following sense: 
   \begin{thm}\label{rep1}
 Let $Y$ be an oriented homology $3$-sphere and $K$ an oriented $2$-knot. Suppose $Y$ is a Seifert hypersurface of $K$.
 \begin{enumerate}
 \item  If $r_s(Y)<\infty$ and $l^s_Y<\infty$ for some $s \in [-\infty, 0]$ and $Y$ is a Seifert hypersurface of $K$, then there exists a positive integer $l$ with $l \leq l^s_Y$ such that there exists an irreducible representation 
 \[
 \rho\colon G_l(K)  \to SU(2).
 \]
 In particular, If  $l^s_Y=1$ and $r_s(Y)<\infty$ for some $s \in [-\infty, 0]$, then there exists an irreducible representation $\rho : G(K) \to SU(2)$. 
 \item  If $\Gamma_{-Y}(k)<\infty$ and $l^k_Y<\infty$ for some $ k \in \Z_{>0}$ and $Y$ is a Seifert hypersurface of $K$, then there exists a positive integer $l$ with $l \leq l^k_Y$ such that there exists an irreducible representation 
 \[
 \rho\colon G_l(K)  \to SU(2).
 \]
 In particular, if  { $\Gamma_{-Y}(k)<\infty$} and $l^k_{Y}=1$ for some $k \in \Z_{>0}$, then there exists an irreducible representation $\rho : G(K) \to SU(2)$. 
 \end{enumerate}
\end{thm}

 As a corollary, we have the following result: 
\begin{cor}\label{rep3}Let $n$ be a positive integer. Then the knot group of any $2$-knot having $\Sigma(2,3,6n-1)$ as a Seifert hypersurface has at least two irreducible $SU(2)$-representations. Moreover, when $n=1$, the knot group admits at least four irreducible $SU(2)$-representations. 
\end{cor} 
Freedman \cite{F82} proved that for any homology $3$-sphere $Y$, there is a locally flat topological embedding from $Y$ into $S^4$. 
 This means that $Y$ can be realized as a Seifert hypersurface of the unknot when we admit locally flat topological embeddings in the definition of Seifert hypersurfaces. 
  Thus, \cref{rep3} is false for topological Seifert hypersurfaces.

\subsection{Embeddings of $3$-manifolds into negative definite $4$-manifolds}\label{Exi}
Existence of embeddings is a fundamental problem in differential topology. It is well-known that every orientable closed $3$-manifold can be embedded in $S^5$. However, the following problem is quite difficult in general. 
\begin{prob}
For a given $4$-manifold $X$ and $c \in H_3(X; \Z)$, which $3$-manifold $Y$ can be embedded in $X$ with $[Y]=c \in H_3(X; \Z)$ ? 
\end{prob}
This problem has been studied in several situations (\cite{H96}, \cite{GL83}, \cite{K88}, \cite{H09}, \cite{D15}). For example, if $p$ is an integer with $|p | >1$, then $L(p,q)$ cannot be embedded into $S^4$ (\cite{Ha38}).  
As another example, by Donaldson's theorem A (\cite{D83}), the Poincar\'e homology $3$-sphere cannot be smoothly embedded into $S^4$.  However, by Freedman's result (\cite{F82}), it does admit locally flat embedding into $S^4$. 
Thus we see that the smooth and locally flat topological embedding problems are different.  
Our main result relates the existence of embeddings of homology $3$-spheres $Y$ of a certain type into a definite $4$-manifold $X$ to the existence of irreducible representations $\pi_1(X) \to  SU(2)$. 
In order to state our main result, we recall from \cite{Ma18} the maps
\[
\cs_{X,c}^j : R(X_{j,c}):= \hom (\pi_1(X_{j,c}), SU(2)) /SU(2) \to (0,1]
\]
defined for an oriented closed connected $4$-manifold $X$ and a class $c\in H_3(X; \Z)$ having certain properties (see \cref{Preliminaries}) with parameter $j \in \Z_{>0}$, where $\{X_{j,c}\}$ is the $j$-fold cyclic covering space of $X$ corresponding to $c$. The functional $\cs_{X,c}^j$ is an analog of the Chern-Simons functional. For the precise definition, see \eqref{cs_YjX}. 
\begin{thm}\label{emb1}Let $Y$ be an oriented homology $3$-sphere and $X$ be a closed connected oriented negative definite $4$-manifold. 
Suppose that there exists a smooth embedding from $Y$ to $X$ with $0\neq [Y] \in H_3(X;\Z)$.  
\begin{enumerate}
\item If $r_s(Y)<\infty$ and $l_Y^s<\infty$ for some $s \in [-\infty, 0]$, then
\[
 r_s(Y)-  \lfloor r_s(Y) \rfloor  \in  \bigcup_{1\leq j \leq l^s_Y} \im\cs_{X, [Y]}^j.
\] 
\item If $\Gamma_{-Y}(k)<\infty$ and $l_Y^k<\infty$ for some $k \in \Z_{>0}$, then
\[
\Gamma_{-Y} (k)  - \lfloor \Gamma_{-Y}(k) \rfloor  \in  \bigcup_{1\leq j \leq l^k_Y} \im\cs_{X, [Y]}^j.
 \]
\end{enumerate}
If $[Y]=0$, then 
\[
\infty = r_s(Y)= r_s(-Y) =  \Gamma_{-Y}(k) =\Gamma_{Y}(k)
\]
for any $s \in [-\infty, 0]$ and $k \in \Z_{>0}$. 
\end{thm}
The proposition below provides fundamental properties of $\{cs_{X,c}^j\}_{j\in \Z_{>0}}$.
\begin{prop}\label{emb2}Let $Y$ be an oriented closed connected $3$-manifold and $X$ a closed connected oriented $4$-manifold. 
\begin{enumerate}
\item 
If there exists an embedding from $Y$ to $X$ with $0\neq [Y] \in H_3(X;\Z)$ then  
\[
 \im \cs^j_{X,[Y]}  \subset \im\cs_Y
\]
for any $j$. 
\item Suppose that $\im\cs_{X,c}^j \cap (0,1)\neq \emptyset$ for $j \in \Z_{>0}$. Then there exist $2\# (\im\cs_{X,c}^j \cap (0,1))$ irreducible $SU(2)$-representations of $\pi_1(X_{j,c})$.
\item If $R(X_{j,c})$ is connected, then $\im\cs_{X,c}^j = \{1\}$.
\end{enumerate}
\end{prop}
Since $\Sigma(2,3,5)$ satisfies a nice Floer theoretic condition (see \cref{Image1}), we can detect $\im\cs_{X,c}$ from the critical values of the Chern-Simons functional of $\Sigma(2,3,5)$ when $X$ contains $\Sigma(2,3,5)$ as a smooth submanifold. 
\begin{thm}\label{emb3}Suppose $X$ is a negative definite $4$-manifold containing $\Sigma(2,3,5)$ as a smooth submanifold. 
Then 
\[
\im \cs^j_{X,[-\Sigma(2,3,5)]} = \left\{ \frac{1}{120} , \frac{49}{120} ,1\right\} \subset (0,1]
\]
for any $j \in \Z_{>0}$.
In particular, $\pi_1(X)$ admits at least four irreducible $SU(2)$-representations. 
\end{thm}
Note that $\Sigma(2,3,5 ) \times S^1$ satisfies the assumption of \cref{emb3} and there are exactly four irreducible $SU(2)$-representations on $\Sigma(2,3,5 ) \times S^1$. Moreover, we prove the following existence result for $SU(2)$-representations:
\begin{thm}\label{emb5}
There is an $S^2$-component $C$ in the $SU(2)$-representation space $R(\dis \Sigma(2,3,5,7))$ of $\Sigma(2,3,5,7)$ satisfying the following property: for any a closed definite $4$-manifold $X$ containing $\dis \Sigma(2,3,5,7)$ as a smooth submanifold, all elements in $C$ extend as $SU(2)$-representations of $X$. Therefore, $\pi_1(X)$ admits an uncountable family of irreducible $SU(2)$-representations. In particular, this implies the knot group of any $2$-knot having $ \Sigma(2,3,5,7)$ as a Seifert hypersurface has an uncountable family of irreducible $SU(2)$-representations. 
\end{thm} 
Lastly, we state a non-existence result for embeddings of Seifert homology $3$-spheres. 
\begin{thm}\label{emb4}Let $Y$ be a Seifert homology 3-sphere of a type $\Sigma(a_1, \cdots , a_n)$. Suppose the Fr\o yshov invariant $h(Y)$ of $Y$ is non-zero. 
Then $Y$ cannot be smoothly embedded in any negative definite $4$-manifold $X$ such that the $SU(2)$-representation space $R(X_{j,c})$ of $X_{j,c}$ is connected for all $j$. 
\end{thm}
If $\pi_1(X)$ is a free group or isomorphic to $\Z^l$ for some $l \in \Z_{>0}$, then $R(X)$ is connected. 
Moreover, if 
\[
R(a_1, \cdots, a_n ) =\frac{2}{a} -3 +n +\sum_{i=1}^n \frac{2}{a_i}\sum_{k=1}^{a_i-1} \cot \left( \frac{a \pi k}{a_i} \right) \cot \left(\frac{\pi k}{a_i}\right) \sin^2 \left(\frac{\pi k}{a_i}\right)>0,
\]
then $h(\Sigma (a_1, \cdots, a_n ) )>0$. (See \cite{D18}) 
\subsection{Fixed point theorems for $SU(2)$-representation spaces}
Since instanton Floer homology is modeled on the infinite dimensional Morse homology of the Chern-Simons functional of $3$-manifolds, it is interesting to ask whether or not there is a Lefschetz type fixed point theorem for instanton Floer homology. Ruberman and Saveliev showed the following theorem. 
\begin{thm}[Ruberman-Saveliev, \cite{RS04}]\label{RS}Let $h$ be an orientation preserving self-diffeomorphism on $Y$ with some non-degenerate condition described in \cite[(3.7)]{RS04}.  Then 
\[
\lambda_{FO} (X_h(Y)) = \frac{1}{2} L ( h_* , I(Y)),  
\]
where $X_h(Y)$ is the mapping torus of $h:Y \to Y$, $\lambda_{FO}(X)$ is the Furuta-Ohta invariant introduced in \cite{FO93} and $L ( h_* , I_*(Y))$ is the Lefschetz number of $h_* : I_*(Y) \to I_*(Y)$ introduced in \cite{RS04}. 
\end{thm}
\cref{RS} implies the following fixed point theorem for instanton Floer homology. 
\begin{cor}[Ruberman-Saveliev, \cite{RS04}]
Under the same assumption of \cref{RS}, if $L ( h_* , I(Y)) \neq 0$, then $h^* : R^*(Y) \to  R^*(Y)$ has a fixed point, where $R^*(Y)$ is the set of conjugacy classes of irreducible $SU(2)$-representations of $\pi_1(Y)$. \footnote{In general, $\lambda_{FO} (X_h(Y))$ can be defined for any orientation preserving diffeomorphism. Moreover, if $\lambda_{FO} (X_h(Y))\neq 0$, then $h^* : R^*(Y) \to R^*(Y)$ has a fixed point. }
\end{cor}
We prove a similar fixed point theorem by applying \cref{emb1} to mapping tori of diffeomorphisms. 
\begin{thm}\label{emb6}
Let $Y$ be an oriented homology $3$-sphere and $h$ an orientation preserving self-diffeomorphism of $Y$.
\begin{enumerate}
\item 
 If $r_s(Y)<\infty$ and $l^s_Y<\infty$ for some $s \in [-\infty, 0]$, then there exist a positive number $l \leq l^s_Y$ such that 
\[
(h^* )^l \colon R^*(Y) \to R^*(Y)
\]
 has a fixed point. 
 \item  If {$\Gamma_{-Y}(k)<\infty$} and $l^k_Y<\infty$ for some $k \in \Z_{>0}$, then there exist a positive number $l \leq l^k_Y$ such that 
\[
(h^* )^l \colon R^*(Y) \to R^*(Y)
\]
 has a fixed point. 
 \end{enumerate}
\end{thm}
Combining this with \cref{emb5}, we obtain the following result. 
\begin{thm} \label{emb7} There exists an $S^2$-component $C$ of $R^*( \Sigma (2,3,5,7))$ satisfying the following condition: for any orientation preserving diffeomorphism $h$ on $ \Sigma (2,3,5,7)$, the fixed point set of 
\[
h^* : R^*( \Sigma (2,3,5,7)) \to R^*(  \Sigma (2,3,5,7))
\]
contains $C$.
\end{thm}

This paper is organized as follows: In \cref{Preliminaries}, we review the invariants $\{\cs_{X,c}^j\}$, $\{r_s(Y)\}$ and $\{\Gamma_Y(k)\}$ appearing in the theorems in \cref{Introduction}. In \cref{Invariants lY}, we define $\{l^s_Y\}$, $\{l^k_Y\}$ and $l_Y$ and show several properties of these invariants. In \cref{Some remarks for QX}, we establish formal properties of the invariants $\{\cs_{X, c}^j\}$ including a connected sum formula, the behavior of $\{\cs_{X, c}^j\}$ for $j$, and a surgery formula. 
In \cref{Invariants of 2-knots}, we introduce a family of invariants $\{\cs_{K, j}\}$ of $2$-knots and using the results of \cref{Some remarks for QX}, we show \cref{2-knot1}. In \cref{Morse type perturbation}, we give sufficient conditions (\cref{finiteness1}) for finiteness of $l^s_Y$, $l^k_Y$ and $l_Y$. 
In \cref{Convergence}, using a technique of instanton Floer theory, we show the existence of flat connections on 4-manifolds under assumptions of the existence of embeddings and prove Theorems~\ref{2-knot2} and \ref{emb1}. In \cref{Extendability}, we prove \cref{emb5}. In \cref{Examples}, we prove Theorems~\ref{emb3} and \ref{rep1}. In this section, we also compare $\{\cs_{K,j}\}$ with other $2$-knot invariants which can be used to obstruct a certain class of Seifert hypersurfaces.

\section*{Acknowledgements}
The author would like to express his deep gratitude to Mikio Furuta for the helpful suggestions. The author also appreciates Aliakbar Daemi for discussing an alternative proof of \cref{Ma18} using his invariants $\{\Gamma_Y(k)\}$.
The author would also like to express his appreciation to Mizuki Fukuda for asking him about examples of $2$-knots. The author also wishes to thank Seiichi Kamada for informing him of several results for ribbon 2-knots and giving him several interesting questions. The author would  like to express his appreciation to Nobuo Iida for suggesting to consider $2$-knots in general $4$-manifolds. 
The author also appreciates Kouki Sato for giving me several comments of the first draft.  The author wishes to thank Junpei Yasuda for informing me a relation between the triple intersection number of a 2-knot and ribbon 2-knots. 
The author was supported by JSPS KAKENHI Grant Number 17J04364, 20K22319, the Program for Leading Graduate Schools, MEXT, Japan and JPSJ Grant-in-Aid Scientific Research on
Innovative Areas Discrete Geometric Analysis for Materials Design (Grant No.\ 17H06461), MEXT, Japan and RIKEN iTHEMS
Program.

\section{Preliminaries}\label{Preliminaries}
\subsection{Chern-Simons functional $\{\cs^j_{X,c}\}_{j \in \Z_{>0}}$(\cite{Ma18})} \label{Inv QX}
Let $X$ be a closed connected oriented $4$-manifold. We fix a class $c \in H_3(X; \Z ) \cong  H^1(X; \Z) \cong [X, B\Z]$.  The class $c$ determines a covering space
\[
p_c\colon \wt{X}^c \to X
\]
up to isomorphism. 
If $c$ is not equal to $0$, then $\wt{X}^c$ is connected. We impose the following condition on $c$:
\begin{ass}\label{assc}\footnote{For example, $2[\{1\}\times S^3]$ in $H_3(S^1\times S^3; \Z)$ cannot be represented by a connected oriented $3$-manifold. }
The class $c$ can be represented by a connected oriented $3$-manifold $Y$.
\end{ass}
Suppose that $c$ is not equal to $0$. We fix a smooth classifying map $\tau\colon X \to S^1$ of $p_c$ and a lift $\wt{\tau}: \wt{X}^c  \to \R$. 
 Let $\wt{R}(X)$ be the set of $SU(2)$-connections on $X \times SU(2)$ modulo null-homotopic $SU(2)$-gauge transformations. For $a \in \wt{R}(X)$, define
\[
cs_{X,c} (a) := -\frac{1}{8\pi^2}\int_{\wt{X}^c} \trace(F_{A_a} \wedge F_{A_a} ) , 
\]
where $A_a$ is a smooth $SU(2)$-connection on $\wt{X}^c \times SU(2)$ such that 
\[
A_a |_{\wt{\tau}^{-1} (-\infty, -1] } = p_c^* a|_{\wt{\tau}^{-1} (-\infty, -1] } \text{ and }A_a |_{\wt{\tau}^{-1} [1, \infty) } = 0.
\]
 The function $\cs_{X,c} :  \wt{R}(X) \to \R$ does not depend on the choices of additional data $\tau$, $\widetilde{\tau}$, $ A_a$, representative of $a$ and isomorphism class of $\wt{X}^c$. If $c =0$, then we define $\cs_{X,c}$ to be the zero map. Note that 
\[
\pi_0(\operatorname{Map}(X,SU(2))) \cong [X, SU(2) ],
\]
where $[X, SU(2) ]$ is the set of homotopy classes of maps from $X$ to $SU(2)$. Fix an oriented closed $3$-manifold $Y$ embedded in $X$ such that $[Y]=c\in H_3(X;\Z)$. If $[Y]$ is not zero in $H_3(X;\Z)$, then $X \setminus Y$ is connected. Here, we suppose that $Y$ is connected. 
In this case, every $0$-dimensional framed submanifold in $Y \cup_{\id} -Y$ bounds some $1$-dimensional framed submanifold in $X$.
By the Pontryagin construction, we see that every continuous map $Y \to SU(2)$ can be extended to a continuous map $X \to SU(2)$. Since $c$ is not zero, then $X \setminus Y$ is connected. Let $W_0$ be the connected compact cobordism from $Y$ to itself obtained by cutting $X$ open along $Y$. We will use the following notations. 
\begin{itemize}
\item[(i)] The manifold $W_i$ is a copy of $W_0$ for $i \in \Z$.
\item[(ii)] We denote $\partial(W_i)$ by $Y^i_+\cup Y^i_-$ where $Y^i_+$(resp. $Y^i_-$) is equal to $Y$(resp. $-Y$) as oriented manifolds.
\item[(iii)] For  $(m,n)\in (\Z\cup \{ -\infty\} ) \times (\Z\cup \{\infty\})$ with $m<n$, we set
\begin{align} \label{def of W}
\displaystyle W[m,n]:=\coprod_{m\leq i \leq n} W_i  / \{Y^j_- \sim Y^{j+1}_+\ j \in \{m,\cdots ,n\}\}.
\end{align}
\end{itemize}
Note that $W[-\infty, \infty] \to X$ is isomorphic to $\wt{X}^c  \to X$ as $\Z$-covering spaces. Using the identification, we have 
\begin{align}\label{cross sect}
cs_{X,c}^j (a) &= -\frac{1}{8\pi^2} \int_{W[-\infty, \infty] } \trace (F_{A_a} \wedge F_{A_a}) \\
 &=  \cs(i_Y^* a ) - \cs(i_Y^* 0 ) =  \cs(i_Y^* a )
\end{align}
for every element $a \in \wt{R}(X_{j,c})$. The formula \eqref{cross sect} gives the following formula: 
\begin{prop} For any element $g \in \pi_0(\operatorname{Map}(X,SU(2)))$, 
\begin{align}\label{degree formula}
cs_{X,c} (g^* a) = \deg (g|_Y) +\cs_{X,c} ( a).
\end{align}
\end{prop}
By \eqref{degree formula}, we obtain a map
\[
cs_{X,c} : R( X) := \wt{R}(X) /\pi_0(\operatorname{Map}(X,SU(2))) \to \R/\Z\cong (0,1].
\]
We call $\cs_{X,c}$ the {\it Chern-Simons functional for }$(X,c)$. If we consider a $C^\infty$ topology on $R( X)$, $\cs_{X,c}$ is a continuous map. Note that $R( X)$ is compact, $\im\cs_{X,c}$ has a minimal value. 
\begin{lem}\label{func}
Suppose we have another pair $(X',c')$ of a closed oriented 4-manifold $X'$ with $c' \in H_3(X';\Z)$ and an orientation preserving diffeomorphism  $f: X \to X' $ satisfying $f^*c=c'$. Then 
\[
cs_{X',c'} (a) =\cs_{X, c} ( f^* a) 
\] 
holds. 
\end{lem}
\begin{proof}This follows from functoriality of integration and $\Z$-covering spaces.  
\end{proof}

If $c$ is not zero,  the class $c$ determines a homomorphism $\rho_c\colon H_1(X;\Z ) \to \Z$. This gives us a surjective homomorphism 
\begin{align}\label{phij}
\phi_j:= \rho_c \circ \text{Ab} \colon \pi_1(X) \to  \im \rho_c \circ \text{Ab} =  i_c \Z \cong \Z \to \Z/j\Z
\end{align}
for some $i_c \in \Z_{> 0}$. 
For each $j \in \Z_{>0}$, we denote by 
\[
p_j \colon X_{j,c}\to X
\]
the covering space corresponding to $\ker \phi_j$. Note that the closed $4$-manifold obtained by identifying the boundary components of $W[0,j-1]$ is diffeomorphic to $X_{j,c}$. 
We also have a $\Z$-covering space 
\[
p^c_j \colon  p_j^* \wt{X}^c  \to X_{j,c}
\]
for each $ j \in \Z_{>0}$. This corresponds to the class $p_j^* c \in H^1( X_{j,c};  \Z) \cong [X_{j,c} ,S^1] $. 
\begin{defn}Fix a pair $(X,c)$ consisting of an oriented closed $4$-manifold X and a class $c \in H_3(X; \Z)$ satisfying \cref{assc}. Suppose $c\neq 0$.
Corresponding to the $(X_{j,c}, p_j^*c)$, we have a family of maps 
\begin{align}\label{cs_YjX}
\set {cs_{X,c}^j :=\cs_{X_{j,c},p_j^*c}  \colon R( X_{j,c}) \to (0,1]}_{j\in \Z_{>0}}.
\end{align}
When $c=0$, we set $\cs_{X,c}^j  = 1$ for all $j$.
\end{defn}
To compare $ \{cs_{X,c}^j \}_{j \in \Z_{ >0}}$ with the critical values of the Chern-Simons functional of oriented $3$-manifolds, we will use the following definition. 
\begin{defn} For an oriented closed $3$-manifold $Y$, we define
\begin{align}\label{Lambda}
\Lambda_Y = \set {\cs_Y(a) \in \R| a \text{ is an $SU(2)$-flat connection on }Y } \text{ and }
\end{align}
\begin{align} \label{Lambda*}
\Lambda^*_Y =\set {\cs_Y(a) \in \Lambda | a \text{ is an irreducible $SU(2)$-flat connection on }Y}, 
\end{align}
where $\cs_Y$ is the Chern-Simons functional of $Y$.
\end{defn}
\begin{prop}[\cref{emb2}, 1]\label{value of QX}
For any oriented connected $4$-manifold, $0\neq c \in H_3(X;\Z)$ with \cref{assc} and $j \in \Z_{>0}$, 
\[
\im\cs_{X,c}^j \subset (\Lambda_Y \cap (0,1]).
\]
 If $Y$ is a Seifert $3$-manifold, then $\im\cs_{X,c}^j \subset \q \cap (0,1]$.
\end{prop}
\def\proofname{Proof of \cref{value of QX}}
\begin{proof}
Suppose that $Y$ is an oriented connected codimension $1$ submanifold of $X$ with $[Y]=c$. If $c$ is not zero, then $X \setminus Y$ is connected. Let $W_0$ be the compact cobordism from $Y$ to itself given by $\overline{X \setminus Y}$. Then, identifying boundaries of $W[0,j]$, gives us $X_{j,c}$ as in \eqref{def of W}. Fix an element $a \in \wt{R}(X_{j,c})$ such that 
\[
\cs_{X,c}^j (a)   \leq 1. 
\]
Note that $W[-\infty, \infty] \to X_{j,c}$ is isomorphic to $p^c_j \colon  p_j^* \wt{X}^c  \to X_{j,c}$. Then 
\begin{align}
\cs_{X,c}^j (a) &= -\frac{1}{8\pi^2} \int_{W[-\infty, \infty] } \trace (F_{A_a} \wedge F_{A_a}) \\
 \label{formula} &=  \cs_Y(i_Y^* a ) - \cs_Y(i_Y^* 0 ) =  \cs_Y(i_Y^* a ).
\end{align}
This gives the conclusion. If $Y$ is a Seifert $3$-manifold, it is shown in \cite{A94} that $\Lambda_Y \subset \q$. Therefore, we have $\im\cs_{X,c}^j \subset \q \cap (0,1]$. 
\qed \end{proof}
\begin{prop}\label{loc const}The maps $\cs_{ X,c}^j$ are locally constant with respect to the $C^\infty$-topology.
In particular, for any $j \in \Z_{>0}$ and a pair $(X,c)$ satisfying \cref{assc}, $\im\cs_{X,c}^j$ is a finite set. 
\end{prop}
\begin{proof}Fix a closed oriented connected $3$-manifold $Y$ such that $[Y]=c \in H_3(X;\Z)$. Using \eqref{formula}, we have $\cs_{X,c}^j (a) =\cs_Y ( a|_Y)$ for any $j \in \Z_{>0}$. If $\rho_t$ is a path of $SU(2)$-flat connections, then $\cs_{X,c}^j (\rho_t) =\cs_Y ( \rho_t|_{Y})=\cs_Y (\rho_0|_Y)$ since $\cs_Y$ is locally constant.  By compactness of $R(X_{j,c})$, $\im\cs_{X,c}^j$ is a finite set. 
\qed
\end{proof}
The following gives us a sufficient condition for triviality of $\{cs_{X,c}^j \}_{j \in \Z_{>0}}$.
\begin{prop}If $X$ is connected and $\pi_1 (X_{j,c})$ is isomorphic to $\Z^l$ or free for $j \in \Z_{>0}$, then 
\[
\im \cs_{X,c}^j = \{1\}.
\]
In particular, $\im\cs_{S^1\times S^3, c} = \im\cs_{T^4, c} = \im\cs_{T^2 \times S^2,c}= \im\cs_{ \#_m S^1\times S^3, c}  = \{1\}$ for any class $c$ and $m \in \Z_{>0}$.
\end{prop}
\begin{proof}
Since $R(X_{j,c}) =\hom (\pi_1(X_{j,c}) , SU(2) ) / SU(2) $, if  $\pi_1 (X_{j,c})$ is isomorphic to $\Z^l$ or free, then $R(X_{j,c}) $ is connected. By \cref{loc const}, $\cs_{X,c}^j(a) =\cs_{X,c}^j (0) =0$.
\qed
\end{proof}
The $4$-manifolds below give non-trivial examples of $\{\cs_{X,c}^j\}$. 
\begin{ex}Let $Y$ be an oriented closed connected $3$-manifold. Then 
\[
\im \cs_{Y \times S^1, [Y] }^j = \Lambda_Y \cap (0,1]
\]
for any $j \in \Z_{>0}$. This is a consequence of \eqref{formula}. 
\end{ex}
\begin{lem}\label{fund prop} For any positive integer $ m$,
\[
\im\cs_{ X,  c}^{j} \subset  \im\cs_{X,c}^{mj}
\]
for any $j, m  \in \Z_{>0}$. 
\end{lem}
 \begin{proof}Note that $X_{mj,c}$ is the total space of a $\Z/m\Z$ covering space $p_{m,j}: X_{mj,c} \to X_{j,c}$. Choose $\rho \in \wt{R}^* ( X_{j,c})$ such that $\cs_{X, c}^j (\rho)<1$. We put $\rho':= p_{m,j}^* \rho$. Then, one can check that 
 \[
\cs_{X_{j,c}, c} (\rho) =\cs_{(mj)_cX, p^*_j c} (\rho') .
 \]
 This completes the proof. 
 \qed \end{proof}
 Although there are many non-trivial examples of $\cs_{ X,c}^j$, we could not find the example of a pair $(X, c)$ whose $\cs_{X,c}^j$ is not constant with respect to $j$. 
 \begin{ques}
 Is there a $4$-manifold $X$ with a class $c\in H_3(X; \Z)$ such that $\{\im \cs_{X,c}^j\}$ is not constant with respect to $j$?
 \end{ques}
 The following tell us that non-triviality of $\im \cs^j_{X,c}$ implies the existence of irreducible $SU(2)$-representations of $\pi_1(X_{j,c})$.
 \begin{lem}\label{existence of irred}
There exist $2 \#( \im \cs^j_{X,c} \cap (0,1))$ irreducible $SU(2)$-representations of $\pi_1(X_{j,c})$.
 \end{lem}
 \begin{proof} The subspace of reducible $SU(2)$-representations is connected. 
Thus, for a reducible representation $\rho$, one has $\cs^j_{X,c}(\rho)=1$ by \cref{loc const}. Therefore, for every element $r \in  \im \cs^j_{X,c} \cap (0,1)$, we have an irreducible representation $\rho : \pi_1(X_{j,c}) \to SU(2)$ such that $\cs_{X,c}^j (\rho)  =r$. Fix an oriented connected $3$-manifold $Y$ representing the class $c \in H_3(X; \Z)$. As in the construction above, we can reconstruct $X_{j,c}$ by gluing \[
W[0,j-1]:= W_0 \cup_Y W_1 \cup_Y \cdots \cup_Y W_{j-1}
\]
 along the boundaries. We regard $\rho$ as an irreducible connection $\wt{\rho}$ on $W[0,j-1]$. The restrictions of $\wt{\rho}$ on the boundaries $Y^+_0$ and $Y^-_j $ of $W[0,j]$ are isomorphic each other. Moreover, by the definition of $\cs_{X,c}^j$, we have equalities 
 \[
 \cs_{X,c}^j (\rho) =\cs_Y \left( \wt{\rho}|_{Y^+_0} \right) = -\cs_Y \left( \wt{\rho}|_{Y^-_j} \right) \in (0,1).
  \]
  This implies that $\wt{\rho}|_{Y^+_0}$ is an irreducible connection on $Y^+_0$. 
 There is a sign ambiguity when gluing $\wt{\rho}$ along $Y^+_0 \cup Y^+_j$. We denote the two resulting glued connections by $\rho_+$ and $\rho_-$. One of $\rho_+$ and $\rho_-$ is isomorphic to $\rho$. We see that $\rho_+$ and $\rho_-$ are not gauge equivalent: suppose there is a gauge transformation $g$ on $X_{j,c}$ such that $g^* \rho_+ = \rho_-$. Then by restricting $g^* \rho_+$ and $\rho_-$ to $W[0,j-1]$, we obtain 
 \[
 g^*\wt{\rho} = g^* \rho_+|_{W[0,j-1]} = \rho_-|_{W[0,j-1]}= \wt{\rho}.
 \]
 Since $\wt{\rho}$ is irreducible, we conclude that $g= \pm1$. We take a based loop $l \subset X_{j,c}$ such that $l \cdot [Y]=1$. Then the holonomies of the connections $\rho_+|_{l}$ and $\rho_-|_{l}$ obtained by the pull-back satisfy
 \[
\operatorname{Hol}_{\rho_+|_{l}} = -  \operatorname{Hol}_{\rho_-|_{l}}  \in SU(2)
 \]
 by construction of the $\rho_\pm$. This contradicts the assumption that $\rho_+= g^* \rho_+ = \rho_-$. 
 This completes the proof.

 \qed
 \end{proof}
 \subsection{Holonomy perturbations}
 Our main tool of this paper is instanton Floer theory. We refer the reader to \cite{Do02} and \cite{Fl88} for the construction of instanton Floer homology. 
 In instanton Floer theory, we consider parturbations of the Chern-Simons functional. First, we review the set $\cal{P}(Y)$ of perturbations used in ordinary Floer theory. 
 Let $\cal{F}_d$ be the set of $d$ embeddings of $S^1 \times D^2$ into $Y$ for $d \in \Z_{>0}$. Fix $m \gg 2$, we denote by $C^m_{\text{ad}}(SU(2), \R)$ the set of adjoint invariant real valued $C^m$-functions. 
The set of perturbations is 
 \[
\cal{P}(Y):= \bigcup_{d \in \Z_{>0}} \cal{F}_d \times C^m_{\text{ad}}(SU(2), \R)^d.
\]
 In this paper, we treat a slightly larger class $\cal{P}^* (Y)$ of perturbations than $\cal{P}(Y)$. The class $\cal{P}^* (Y)$ was used in \cite{Sav02} to calculate the instanton homologies of Seifert homology $3$-spheres. 
 \begin{defn}
 We define the set of perturbations \footnote{We regard $\cal{P}(Y)$ as a subset of $\cal{P}^*(Y)$ via 
\[
\displaystyle (f, h) \mapsto \left(f,h, (x_i)_{1 \leq i \leq d} \mapsto \sum_{1 \leq i \leq d} x_i \right).
\]} by 
\[
\cal{P}^* (Y):= \bigcup_{d \in \Z_{>0}} \cal{F}_d \times C^m_{\text{ad}}(SU(2), \R)^d \times C^m (\R^d, \R).
\]
\end{defn}

We fix a volume form $ d\mu$ on $D^2$ such that $\text{supp } d\mu \subset \text{int} D^2$ and $\int_{D^2} d\mu =1$. For a triple $\pi=(f,h, q) \in \cal{P}^* (Y)$, we define the {\it perturbed Chern-Simons functional} by
\[
cs_{Y, \pi}=\cs_Y+ h_{\pi}: \wt{\B}^*(Y) \to \R, 
\]
where 
\begin{itemize}
\item $\wt{\B}^*(Y)$ is the quotient set 
\[
\left(\A^*(Y):= \set{ \text{irreducible }SU(2)\text{-connections on } Y\times SU(2) } \right)/ \operatorname{Map}^0 (Y, SU(2)),
\]
where $\operatorname{Map}^0 (Y, SU(2))$ is the set of smooth maps whose mapping degrees are zero, 
\item  $cs_Y$ is { the Chern-Simons functional} given by 
\[
cs_Y(a) := - \frac{1}{8\pi^2} \int_Y \trace \left(a \wedge da + \frac{2}{3} a\wedge a \wedge a \right), \text{ and}
\]
\item $\dis h_\pi(a) :=  q \left( \left( \int_{x\in D^2} h_i \hol_{f_i(s,x)}(a) \right)_{1\leq i \leq d}\right)$.
\end{itemize}
We often use an $L^2_k$-completion and a Banach manifold structure on $\wt{\B}^*(Y)$ for a fixed $k>2$.
Note that the map $\cs_{Y, \pi}$ descends to a map 
\[
\cs_{Y, \pi} : \B^*(Y):= \A^*(Y) /  \operatorname{Map} (Y, SU(2))  \to \R /\Z.
\]
We now write down the formal gradient vector field of $\cs_{Y, \pi}$. Fix a Riemann metric $g_Y$ on $Y$. 
Then, identifying $\su$ with its dual by the Killing form, 
we can regard the derivative $h'_i$ as a map $h'_i \colon SU(2) \to \su$. The holonomy of the loops $\{f_i(s,x) \mid s\in S^1\}$ gives us a section $\hol_{f_i(s,x)}(a)$ of the bundle $\aut P_Y$ over $\im f_i$. The bundle map induced by $h_i' \colon \aut P_Y \to \ad P_Y$, then gives us a section $h_i'(\hol_{f_i(s,x)}(a))$ of $\ad P_Y$ over $\im f_i$. We now describe the gradient-flow equation of $\cs_{Y,\pi}$ with respect to the $L^2$-metric:
\begin{align}\label{grad}
 \frac{\partial}{\partial t} a_t=- \grad_{a_t} \cs_{Y,\pi} = *_{g_Y}\left( F(a_t)+ \sum 
_{i=1}^m \partial_i q_{ (h_i(\hol(a_t)_{f_i(s,x)}))_{1 \leq i \leq m}   }  h'_i(\hol(a_t)_{f_i(s,x)})(f_i)_*{\pr}_2^*d\mu  \right),
\end{align}
where $\pr_2$ is the projection $\pr_2 \colon S^1\times D^2 \ri D^2$ and $*_{g_Y}$ is the Hodge star operator with respect to $g_Y$. 
(For the calculation of the gradient, see \cite{BD95}.)
We denote $\pr_2^*d\mu $ by $\eta$.
We set
\[
\widetilde{R}(Y)_\pi:= \left\{a \in \widetilde{\B}(Y) \Biggm | \grad_{a} \cs_{Y,\pi}  =0 \right\},
 \]
 and 
 \[
 \widetilde{R}^*(Y)_\pi:= \widetilde{R}(Y)_\pi \cap \widetilde{\B}^*(Y).
 \]
 When we consider a smooth manifold structure on $\widetilde{R}^*(Y)_\pi$, 
we use an $L^2_k$-topology\footnote{These topologies on $R^*(Y)_\pi$ do not depend on the choice of $k>2$ if $\pi$ is a smooth perturbation.} for some $k>2$.
The solutions of \eqref{grad} correspond to connections $A$ over $Y\times \R$ which satisfy the equation:
\begin{align}\label{pASD}
F^+(A)+ \pi(A)^+=0,
\end{align}
where
\begin{itemize}
\item the 2-form $\pi(A)$ is given by 
\[
\sum 
_{i=1}^m \partial_i q_{ (h_i(\hol(a_t)_{f_i(t, x, s)}))_{1 \leq i \leq m}   }   h'_i(\hol(A)_{\tilde{f}_i(t,x,s)}) \otimes {(\tilde{f}_i)}_* (\pr_1^* \eta),
\]
\item the map $\pr_1$ is the projection map from $(S^1\times D^2) \times \R$ to $S^1\times D^2$,
\item $F^+(A) = \frac{1}{2}(1+*F(A))$ where $*$ is the Hodge star operator with respect to the product metric on $Y\times \R$, and similarly for $\pi(A)^+$, and 
\item $\tilde{f}_i\colon  S^1\times D^2\times \R \ri Y\times \R$ is the embedding given by $f_i\times id$ for each $i$. 
\end{itemize}

We define $\|\pi\|= \| (f,h , q) \|   := \|q \circ h\|_{C^m}$. We also define non-degenerate and regular perturbations for elements in $\mathcal{P}^*(Y)$ in the same way as in the case of $\mathcal{P}(Y)$ (See \cite{Ma18} for further details.) For an oriented homology $3$-sphere $Y$ and a fixed metric $g_Y$ on $Y$, there exists a positive integer $d$, a collection of embeddings $f \in  \cal{F}_d$ and a Bair subset $Q$ of $C^m_{\text{ad}}(SU(2), \R)^d \times C^m (\R^d, \R)$ such that $(f,u)$ is non-degenerate and regular for $u \in Q$. 

For a $4$-manifold $W$ with cylindrical ends, we also use a large class of perturbations. Let $\cal{F}_d(W)$ be the set of $d$ embeddings from $S^1 \times D^3$ to $W$ for any $d \in \Z_{>0}$.  Fix a volume form $ d\nu$ on $D^3$ such that $\text{supp } d\nu \subset \text{int} D^3$ and $\int_{D^3} d\nu =1$. 
We define 
\[
\cal{P}^* (W):= \bigcup_{d \in \Z_{>0}} \cal{F}_d(W) \times C^m_{\text{ad}}(SU(2), \R)^d \times C^m (\R^d, \R).
\]
 For $\pi=(g,h, q) \in \cal{P}^* (W)$, we have the {\it perturbed ASD equation with respect to }$\pi$ defined by 
 \begin{align}\label{pASD1}
F^+(A)+\left(  \sum 
_{i=1}^m \partial_i q_{ (h_i(\hol(a_t)_{f_i(t, x, s)}))} dh_i \hol_{g_i(t,x)}(A)   \otimes (g_i)_*\pr_2^* d \nu \right)^+=0, 
 \end{align}
 where $\pr_2$ is the projection $\pr_2 : S^1\times D^3 \to D^3$.
 We will often write the part 
 \[
 \sum 
_{i=1}^m \partial_i q_{ (h_i(\hol(a_t)_{f_i(t, x, s)}))} dh_i \hol_{g_i(t,x)}(A)   \otimes (g_i)_*\pr_2^* d \nu
\]
by $\pi(A)$. 

 \subsection{Moduli spaces of perturbed ASD equations}
 In this section, we review the construction of the cobordism map in instanton Floer theory. 
In this paper, we only use the moduli space of solutions to ASD equations on manifolds of the form $Y \times \R$ and $W^*$, where $W$ is a negative definite cobordsim from $Y$ to itself and $W^*$ is defined by 
\begin{align}\label{def of W*}
Y \times \R_{\leq 0}\cup_Y W \cup_Y Y\times \R_{\geq 0} .
\end{align}
 We assume that $H_1(W; \R)=0$.
Fix a regular non-degenerate perturbation $\pi \in \mathcal{P}^*(Y)$. For two irreducible critical points $a$, $b \in \wt{R}^*(Y)_\pi$, we will define moduli spaces $M^{Y}(a,b)_\pi$ and $M(a,W^*, \theta)_{\pi_W}$. 

Fix a positive integer $q\geq3$. Let $A_{a,b}$ be an $SU(2)$-connection on $Y \times \R$ satisfying $A_{a,b}|_{Y\times (-\infty,1]}=p^*a$ and $A_{a,b}|_{Y\times [1,\infty)}=p^*b$ where $p$ is the projection $Y\times \R \ri Y$.
We then define
\begin{align}\label{*}
M^Y(a,b)_\pi:=\left\{A_{a,b}+c  \Bigm| c \in \Omega^1(Y\times \R)\otimes \su_{L^2_q}\text{ satisfying } \eqref{pASD} \right\}/ \G(a,b),
\end{align}
where $\G(a,b)$ is given by 
\[
\G(a,b):=\left\{ g \in \aut(P_{Y\times \R})\subset {\End(\mathbb{C}^2)    }_{L^2_{q+1,\text{loc}}} \Bigm| \nabla_{A_{a,b}}(g) \in L^2_q \right\}.
\]
The action of $\G(a,b)$ on $\left\{A_{a,b}+c \Bigm| c \in \Omega^1(Y\times \R)\otimes \su_{L^2_q}\text{ satisfying }\eqref{pASD} \right\}$ is given by pull-backs of connections. The space $\R$ acts on $M^Y(a,b)_\pi$ by translation. We denote by $\theta$ the product $SU(2)$-connection on $Y$.
We also have moduli spaces $M^{Y}(a,\theta)_\pi$ defined by similar way as $M^Y(a,b)_\pi$ but we use a weighted norm to define $M^{Y}(a,\theta)_\pi$. (See \cite{NST19}.)

Next, for two irreducible critical points $a$, $b \in \wt{R}(Y)_\pi$, let $A_{a,b}$ be an $SU(2)$-connection on $W^*$ satisfying $A_{a,b}|_{Y\times (-\infty,1]}=p^*a$ and $A_{a,b}|_{Y\times [1,\infty)}=p^*b$ where $p$ is the projection $Y\times \R \ri Y$.
We define 
\begin{align}\label{**}
M(a,W^*, b)_\pi:=\left\{A_{a,b}+c  \Bigm| c \in \Omega^1(W^*)\otimes \su_{L^2_q}\text{ satisfying } \eqref{pASD1} \right\}/ \G(a,b),
\end{align}
where $\G(a,b)$ is given by the same formula as in the case of $Y \times \R$.

\subsection{Invariants $\{r_s(Y)\}_{s \in [-\infty, 0]}$ and Daemi's invariants $\{\Gamma_Y(k)\}_{k \in \Z}$} \label{r0}
In this section, we review of two families of $(0,\infty]$-valued homology cobordism invariants $\{r_s(Y)\}_{s \in [-\infty, 0]}$ and $\{\Gamma_Y(k)\}_{k \in \Z}$ of homology $3$-spheres. To define $\{r_s(Y)\}_{s \in [-\infty, 0]}$, we use $\Z$-graded filtered instanton Floer homology whose filtration comes from the Chern-Simons functional. On the other hand, Daemi used $\Z/8\Z$-graded instanton homology with some local coefficient coming from Chern-Simons functional to define $\{\Gamma_Y(k)\}_{k \in \Z}$. For more details on $\{r_s(Y)\}$ and $\{\Gamma_Y(k)\}$, see \cite{NST19} and \cite{D18}.

\subsubsection{Invariants $\{r_s(Y)\}_{s \in [-\infty, 0]}$ }\label{ss}
For an oriented homology $3$-sphere $Y$, we review the definition of $\{r_s(Y)\}_{s \in [-\infty, 0]}$. This invariant was defined in \cite{NST19} to analyze the structure of the homology cobordism group of homology $3$-spheres.  For  $r,s \in [-\infty ,\infty)$ satisfying $-\infty \leq s \leq 0 \leq r<\infty$ such that $r$ is a regular value of $\cs_Y$, we have a filtered instanton Floer cohomology $I^1_{[s,r]}(Y)$. 
In this paper, we use the class of perturbations $\cal{P}^*(Y)$ instead of $\cal{P}(Y)$ used in \cite{NST19}. 
 \begin{defn}\label{epsilon perturbation} Let $Y$ be an oriented homology $3$-sphere and $g_Y$ be a Riemannian metric on $Y$.
 For $\varepsilon>0$, we define a class of perturbations $\mathcal{P}_\varepsilon^*(Y,g)$ as a subset of $\mathcal{P}^* (Y)$ consisting of elements which satisfy 
 \begin{enumerate} 
 \item \label{eps}$| h_\pi(a) | < \varepsilon \text{ for all } a \in \widetilde{\B}(Y)$ and 
 \item \label{eps2}$\| \grad_g h_{\pi}(a) \|_{L^4}< \frac{\varepsilon}{2}, \| \grad_g h_{\pi}(a) \|_{L^2}< \frac{\varepsilon}{2}$  for all $a \in \widetilde{\B}(Y)$.
 \end{enumerate}
 \end{defn}
 We choose a suitable small $ \varepsilon$ by the following argument. 
Let $\{R_\alpha\}$ be the connected components of ${R}^*(Y)$. Let $U_\al$ be a neighborhood of $R_\al$ in $\B(Y)$ with respect to the $C^\infty$-topology such that $U_\al \cap U_\beta = \emptyset$ if $\al \neq \beta$ and $\{U_\al\}$ is a covering of $R^*(Y)$.  We take all lifts of $U_\al$ with respect to $\text{pr}:\widetilde{\B}_Y \to \B_Y$. Since $\operatorname{Map} (Y,SU(2)) / \operatorname{Map} _0 (Y,SU(2))$ is isomorphic to $\Z$, we denote all lifts by $\{U_\al^i\}_{i \in \Z}$.
 In addition, we assume the following conditions on $U_\al^i$.
\begin{itemize}
\item If $a \in U^i_\al$,  $|\cs(a)-\cs(R_\al)| <  \min \{ \frac{ d(r, \Lambda_Y)}{8}, \frac{ d(s, \Lambda_Y)}{8}\}$, where $d(r, \Lambda_Y)$ is given by
\[
d(r, \Lambda_Y) := \min\{| r - a | \in \R_{>0} \mid a \in \Lambda_Y\}.
\]
\item $U^i_\al$ has no reducible connections.
\end{itemize}
Note that, for any element $\rho \in \wt{R}(Y)$, we have unique $\al$ and $i\in \Z$ such that $\rho \in U_\al^i$. 
 
By the Uhlenbeck compactness theorem, we can take a sufficiently small real number $\epsilon_1(Y,g, \{U_\al\})>0$ satisfying the following condition:
\begin{align}\label{nbd}
\text{ If }a \in \B^*(Y) \text{ and }\|F(a)\|_{L^2} \leq \epsilon_1(Y,g, \{U_\al\}) \text{, then } a \in U_\al \text{ for some $\al$}.
\end{align}

\begin{defn}Now we take the supremum value
\[
\epsilon_1(Y,g,r,s) := \frac{1}{2}\sup_{ \{U_\al\}}  \epsilon_1(Y,g, \{U_\al\}) ,
\]
where $\{U_\al\}$ runs over all coverings of $\{R_\al \}$ given as above method. 
We define
\[
\epsilon(Y,r,s,g):= \begin{cases}  
\min \{ \epsilon_1(Y,g), \frac{ d(s, \Lambda_Y)}{8}, \frac{ d(r, \Lambda_Y)}{8},\frac{ \lambda_Y }{32} \} & \text{if $s \in \R_Y$,} \\ 
\min \{ \epsilon_1(Y,g), \frac{ d(r, \Lambda_Y)}{8},\frac{ \lambda_Y }{32} \} & \text{if $s \in \Lambda_Y$}, \\
 \end{cases}
\]
where $\lambda_Y:= \min \{ |a-b| \mid a,b \in \Lambda_Y \text{ with }a\neq b  \}$. 
We then define 
 \[
 \cal{P}^*(Y,g,r,s) := \cal{P}^*_{\varepsilon(Y, g ,r ,s ) }(Y,g). 
 \]
 \end{defn}
 We also use the notation  $\lambda_Y:= \min \{ |a-b| \mid a,b \in \Lambda_Y \text{ with }a\neq b  \}$. 
Then we define a class of perturbations which we will use later. 
  For a non-degenerate perturbation $\pi \in \cal{P}^*(Y)$, we consider a map 
 \begin{align}\label{index}
 \ind : \wt{R}(Y)_{\pi} \to \Z
 \end{align}
 called the {\it Floer index} in order to construct $\Z$-gradings on Floer's chain complexes. 
 Fix two elements $r,s \in [-\infty ,\infty)$ satisfying $-\infty \leq s \leq 0 \leq r<\infty$ and $r \in \Lambda^*_Y$. 
 For a metric $g$ on $Y$, a non-degenerate regular perturbation $\pi \in \mathcal{P}(Y,r,s,g)$,
the (co)chains of the filtered instanton Floer (co)homologies are defined by
\[ 
CI^{[s,r]}_i(Y, \pi):=\begin{cases} 
 \Z \left\{ [a] \in \widetilde{R}^*(Y)_\pi \Bigm| \ind(a)=i,\  s<\cs_{Y,\pi}(a)<r \right\} & \text{if $s \in \R_{Y}$,} \vspace{0.5ex}\\
 \Z \left\{ [a] \in \widetilde{R}^*(Y)_\pi \Bigm| \ind(a)=i,\  s- \frac{\lambda_Y}{2} <\cs_{Y,\pi}(a)<r \right\} & \text{if $s \in \Lambda_{Y}$} \\
\end{cases} 
\]
and
\[
CI^i_{[s,r]}(Y, \pi):= \hom(CI_i^{[s,r]}(Y, \pi),\Z), 
\]
 where $\lambda_Y:= \min \{ |a-b| \mid a\neq b , a,b \in \Lambda_Y\}$. 
The (co)boundary maps 
\[
\partial^{[s,r]} \colon CI_i^{[s,r]}(Y, \pi) \ri CI_{i-1}^{[s,r]}(Y, \pi) \ (\text{resp.\ $\delta^r \colon CI^i_{[s,r]}(Y)\ri CI^{i+1}_{[s,r]} (Y)$})
\]
are given by the restriction of Floer's usual differential 
\[
\partial (a) := \sum_{b \in \widetilde{R}^*(Y)_\pi \text{ satisfying } \ind(b)=i-1}\# (M^Y(a,b)_\pi/\R) b\  
\] 
  (resp.\ $\delta^{[s,r]}:=(\partial^{[s,r]})^*$).  For further details of $\partial$, see \cite[Subsection 5.2]{Do02}. 

There is a cohomology class $\theta^{[s,r]}_Y \in I^1_{[s,r]} (Y)$ defined by 
\begin{align}
\theta^{[s,r]}_Y([a]):= \# (M^Y(a,\theta)_{\pi }/\R).
\end{align}

 As in the discussion in \cite[Section~3.3.1]{Do02}, one can see that $I^{[s,r]}_*(Y)$ and $\theta^{[s,r]}_Y \in I^1_{[s,r]} (Y)$ does not depend on the choice of $\pi \in \mathcal{P}^*(Y,g,r,s)$. Therefore, $I^{[s,r]}_*(Y)$ and $\theta^{[s,r]}_Y \in I^1_{[s,r]} (Y)$ are equivalent to the original definitions in \cite{NST19}. 
\begin{defn}\cite[Definition 3.1]{NST19}
We define
\[
r_s(Y):= \sup \left\{ r\in \R_{\geq 0}  \middle| 0= \theta_Y^{[s,r]}\in I^1_{[s,r]}(Y) \right\}. 
\]
\end{defn} 
In this paper, we will use the following property of the class $\theta_Y^{[s,r]} $.
Let $W$ be a negative definite cobordism such that $\partial W= Y_0 \cup (-Y_1)$ satisfying $H_1(W; \R) =0$. Let $I(W)\colon I^1_{[s,r]} (Y_0) \to I^1_{[s,r]} (Y_1)$ be the cobordism map introduced in \cite{NST19}. 
\begin{prop}\cite[Lemma 2.12]{NST19} Suppose that $H_* (W; \R) \cong H_*(S^3; \R)$.  For two real numbers $r,s \in \R$ satisfying $s \leq 0 \leq r$ and $r$ is regular value of $\cs_Y$, 
\[
I(W) (\theta_{Y_1}^{[s,r]}) = c(W) \theta_{Y_0}^{[s,r]}, 
\]
where $c(W)= \# H_1(W; \Z)$.
\end{prop}
\begin{thm}\cite[Theorem 1.1]{NST19} \label{main NST}The invariants $ \{r_s(Y)\}_{s \in [-\infty, 0]}$ satisfy the following conditions: 
\begin{itemize}
\item  For $s$, $s_1$, $s_2\in [-\infty , 0 ]$ with $s=s_1 +s_2 $, 
\begin{align}\label{conn sum formula}
r_s(Y_1 \# Y_2 ) \geq \min \{r_{s_1}(Y_1)+s_2, r_{s_2}(Y_2)+s_1 \}
\end{align}
holds. 
\item  If there exists a negative definite cobordism $W$ with $\partial W= Y_1 \amalg -Y_2$, then the inequality
\begin{align} \label{def ineq}
r_s (Y_2) \leq r_s(Y_1) \end{align}
holds for any $s \in [-\infty, 0]$.
\end{itemize}
\end{thm}
\subsubsection{Daemi's invariants $\{\Gamma_Y(k)\}_{k \in \Z}$}
Let $\Lambda$ be the Novikov ring 
\[
\Lambda := \left\{ \sum_{i=1}^{\infty} a_i \lambda^{r_i} \middle| a_i \in \mathbb{Q} , \ r_i  \in \R , \ \lim_{i\to \infty } r_i =\infty \right\} ,
\]
where $\lambda$ is a formal variable.
We have an evaluating function $\mdeg\colon  \Lambda  \to \R$ defined by 
\[
 \mdeg\left( \sum_{i=1}^{\infty} a_i \lambda^{r_i} \right) := \min_{i \in \Z_{>0}} \{r_i \mid a_i \neq 0\}.
 \] 
 Fix a non-degenerate regular perturbation $\pi$ and orientations of the determinant line bundles $\mathbb{L}_a$ in \cite[Section 2]{NST19}.
 Note that the Floer index \eqref{index} descends to a map
\[
 \ind : {R}(Y)_{\pi} \to \Z/ 8\Z. 
 \]
Then define a $\Z/8\Z $-graded chain complex $C^{\Lambda}_*(Y)$ over $\Lambda$ by
  \[
C^{\Lambda}_i(Y):= C_i(Y) \otimes \Lambda = \Lambda \set { [a] \in R^* (Y)_\pi  | \ind (a)=i \mod 8}
  \]
with the differential
$$
d^{\Lambda} ([a]) 
\displaystyle
:= \sum_{\ind (a) -\ind(b) \equiv 1 (\text{mod } 8), A \in M^Y([a],[b])_\pi}  \# (M^Y([a],[b])_\pi/\R) 
\cdot \lambda^{\mathcal{E}(A)} [b],
$$
where $\ind$ is the function \eqref{index}, 
\[
\mathcal{E}(A) :=\frac{1}{8\pi^2} \int_{Y \times \R} \trace ( (F(A) + \pi(A)) \wedge (F(A)+\pi(A)) )
\]
and $M^Y([a],[b])_\pi$ denotes $M^Y(a, b)_\pi$ for some representatives $a$ and $b$ of $[a]$ and $[b]$ satisfying $\ind (a)  -\ind (b)= 1$. Extend the function $\mdeg$ to $C^{\Lambda}_*$ by
$$
\mdeg \left(\sum_{1 \leq k \leq n} \eta_k [a_k] \right) = \min_{1 \leq k \leq n} \{ \mdeg(\eta_k) \}.
$$

In addition, we define two maps: 
\begin{itemize}
\item[(i)]  The map $D_1 \colon C^{\Lambda}_1(Y) \to \Lambda$ is given by
$$
D_1([a]) = (\# M^Y([a],[\theta])_\pi/\R) \cdot \lambda^{\mathcal{E}(A)},
$$
where $A \in M^Y([a],[\theta])_\pi$ and $M^Y([a],[\theta])_\pi$ denotes $M^Y(a, \theta^i)_\pi$ for some lifts $a$ and $\theta^i$ of $[a]$ and $[\theta]$ satisfying $\ind (a)  -\ind (\theta^i)= 1$.
\item[(ii)] The map 
$U \colon C^{\Lambda}_*(Y) \to C^{\Lambda}_{*-4}(Y)$ is defined by 
\[
\dis U([a]): =\sum_{\substack{[b] \in \wt{R}(Y)_\pi \\ 
\ind ([b])-\ind ([a])= 4} } -\frac{1}{2} \# N(a,b) [b]\cdot \lambda^{\mathcal{E}(A)}, 
\]
 where the space $N^Y(a,b)$ is the codimension $4$-submanifold of $M^Y(a,b)$ given by  
\[
N(a,b):= \{ [A] \in M^Y(a,b) | s_1(r([A] )) \text{ and }  s_2(r([A])) \text{ are linearly dependent} \}
\]
and $A$ is an element in $N(a,b)$. 
Here $r \colon M^Y(a,b)  \to \B^*(Y\times (-1,1))$ is the restriction map and $s_1$ and $s_2$ are generic sections of the bundle $\bb{E} \otimes  \co \to \B^*(Y\times (-1,1))$. The $SO(3)$-bundle $\bb{E}$ is given by a basepoint fibration of $\B^*(Y\times (-1,1))$. 
\end{itemize}
Now, in our conventions, 
$\Gamma_{-Y}(k)$ is given by 
\begin{align}\label{gamma}
\Gamma_{-Y}(k) = \lim_{|\pi| \to 0}  \left(
\inf_{\substack{\alpha \in C_*^{\Lambda}(Y), d^{\Lambda}(\alpha)=0 \\ D_1 U^j(\alpha) = 0 ( 1 \leq j < k-1) \\ 
D_1 U^{k-1}(\alpha) \neq 0  }}
\left\{
\mdeg(D_1U^k (\alpha))- \mdeg (\alpha)   
\right\}
\right)
\end{align}
for $k \in \Z_{>0}$. In \cite{D18}, Daemi also introduced $\Gamma_Y(k)$ for any negative $k\in \Z$. The invariants $\{\Gamma_Y(k)\}_{k \in \Z_{\leq 0} }$ use information of $D_2$ and $U$. 
However, in this paper, we only use the positive part $\{\Gamma_Y(k)\}_{k \in \Z_{>0} }$.
\begin{thm}[ \cite{D18}]The sequence of invariants $\{\Gamma_Y(k)\}_{k \in \Z_{>0}}$ has the following properties: 
\begin{itemize}
\item[(i)] 
If there exists a negative definite cobordism $W$ with $\partial W= Y_1 \amalg -Y_2$, then the inequality
\begin{align}\label{def eq1}
 \Gamma_{Y_1}(k) \leq \Gamma_{Y_2}(k) 
 \end{align}
holds for any $k \in \Z_{>0}$.
\item[(ii)] The invariant $\Gamma_Y(k) <\infty$ for $k\in \Z_{>0}$ if and only if $k\leq 2 h(Y)$. 
\end{itemize}
\end{thm}

\subsection{Relations between $\{r_s(Y)\}$ and $\{\Gamma_Y(k)\}$}
It is natural to ask if there is a relation between $\{r_s(Y)\}_{s\in [-\infty, 0]} $ and $\{\Gamma_Y(k)\}_{k \in \Z_{>0}}$. In \cite{NST19}, the following equality is showed. 
\begin{thm}[\cite{NST19}] For any oriented homology $3$-sphere $Y$,
\[
r_{-\infty} (Y) = \Gamma_{-Y}(1) .
\]
\end{thm}
Therefore, $\{r_s(Y)\}_{s\in [-\infty, 0]} $ and $\{\Gamma_Y(k)\}_{k \in \Z_{>0}}$ satisfy the following inequalities: 
\[
r_0(Y) \leq \cdots \leq r_s(Y) \leq \cdots \leq r_{-\infty}(Y)= \Gamma_{-Y}(1) \leq\cdots \leq \Gamma_{-Y}(k) .
\]

It is also natural to ask if there is an oriented homology $3$-sphere $Y$ such that $\{r_s(Y)\}$ and $\{\Gamma_{-Y}(k) \}$ do not coincide. In \cite{NST19}, we proved $\{-\Sigma(2,3,6k+1) \# \Sigma(2,3,5)\}_{k \in \Z_{>0}}$ gives examples whose $\{r_s(Y)\}$ and $\Gamma_{-Y}(k)$ do not coincide. In this case, $\{r_s(Y)\}$ is not a constant with respect to $s$. Our connected sum formula implies $r_0(-\Sigma(2,3,6k+1) \# \Sigma(2,3,5)) = \frac{1}{24(6k+1)}$. On the other hand, since $h(-\Sigma(2,3,6k+1) \# \Sigma(2,3,5))=0$, $\Gamma_{\Sigma(2,3,6k+1) \# (-\Sigma(2,3,5))} (1)=\infty$. There is also an example of $Y$ such that $\Gamma_Y(k)$ is not constant with respect to $k$: in \cite{D18}, Daemi calculated 
\[
\Gamma_{\Sigma(2,3,5)}(k)  = 
\begin{cases} \frac{1}{120} \text{ if }k=1\\
\frac{49}{120} \text{ if } k=2 \\ 
\infty \text{ if } k\geq 3 
\end{cases} .
\]
For $-\Sigma(2,3,5)$, we have $r_s(-\Sigma(2,3,5)) = \Gamma_{\Sigma(2,3,5)} (1)$ for any $s \in [-\infty, 0]$.

\begin{rem}
 In \cite{DST19}, we will give a generalization $\mathcal{J}_Y (k,s )$ of both of $r_s(Y)$ and $\Gamma_Y(k)$. A theorem similar to \cref{rep1} can be proven for $\mathcal{J}_Y (k,s )$. 
 \end{rem}

\section{The invariants $l^s_Y$, $l^k_Y$ and $l_Y$} \label{Invariants lY}
\subsection{Perturbations and Invariants $l^s_Y$, $l^k_Y$ and $l_Y$}\label{per}
Let $Y$ be an oriented homology $3$-sphere and $g_Y$ a Riemann metric on $Y$, and fix a perturbation $\pi \in P^*_Y$.
Suppose that ${R}^*_\pi(Y) $ is a submanifold of ${\B}^*(Y)$ as the zero set of the gradient vector field of $\cs_{Y, \pi}$. 
For any point $a \in {R}^*_\pi(Y)$, we have the operator 
\[
\text{Hess}_{a}(cs_{Y, \pi})= *d_a + \text{Hess}_a h_\pi \colon \ker d_a^* \to \ker d_a^*, 
\]
where $\ker d_a^*$ is a model of $T_{a} {\B}^*(Y)$ and $d_a^*\colon \Om^1_Y\otimes \su \to \Om^0_Y\otimes \su$.   Note that \[
\text{Hess}_{a}(cs_{Y, \pi})\colon \ker d_a^* \to \ker d_a^*
\]
 is a self adjoint elliptic operator. 
 \begin{defn}
We call $\pi$ a {\it Morse-Bott perturbation} if 
\begin{align}
\text{Hess}_{a}(cs_{Y, \pi})\colon  (T_{a} {R}(Y))^{\perp_{L^2}} \cap  \ker d_a^*  \to (T_{a} {R}(Y))^{\perp_{L^2}}  \cap  \ker d_a^*
\end{align}
is invertible for any $a\in {R}^*_\pi(Y)$. 
\end{defn}
If $\cs_{Y, \pi}$ is Morse-Bott for a perturbation $\pi$ with $h_\pi=0$, then we call $\cs_Y$ {\it Morse-Bott}. In this paper, we set
\[
H^1_{a}(\pi):= \ker \text{Hess}_{a}(cs_{Y, \pi})|_{  \ker d_a^* }.
\]
If $h=0$, then we write $H^1_{a}$.
If we use this notation, $\cs_{Y, \pi}$ is Morse-Bott if and only if $H^1_{a}(\pi) = T_{a} {R}^*_\pi(Y) $ for each $a \in {R}^*_\pi(Y)$. (In general, $T_{a} {R}^*_\pi(Y)  \subset  H^1_{a}(\pi)$ holds.) Note that the condition $H^1_{a} = T_{a} {R}^*(Y) $ does not depend on the choice of metric. Next, we define the notion of {\it Morse-Bott perturbation at level} $r$.
\begin{defn}
We say $\cs_{Y, \pi}$ is {\it Morse-Bott at the level} $r$ if 
\begin{align}
\text{Hess}_{a}(cs_{Y, \pi})\colon  (T_{a} {R}(Y))^{\perp_{L^2}} \cap  \ker d_a^*  \to (T_{a} {R}(Y))^{\perp_{L^2}}  \cap  \ker d_a^*
\end{align}
is invertible for any $a\in {R}^*_\pi(Y) \cap\cs_{Y, \pi}^{-1}(r)$. 
\end{defn}
If $\cs_Y$ is Morse-Bott at the level $r$, one can show $\cs_Y$ is Morse-Bott at the level $r+m$ for any $m\in \Z$.
Set $\Lambda_Y:= \im\cs_Y|_{\wt{R}(Y)}$. If $0$ is a Morse-Bott perturbation for any element $r \in \Lambda_Y$, then $\cs_Y$ is Morse-Bott. 
 \begin{lem} \label{conn sum of MB}If the Chern-Simons functionals for $Y_1$ and $Y_2$ are Morse-Bott, then the Chern-Simons functional for $Y_1 \# Y_2$ is also Morse-Bott.
\end{lem}
\begin{proof}Note that 
\[
 {R}^*(Y_1 \# Y_2 ) = {R}^*(Y_1)\times  {R}^*(Y_2)\times SO(3) \amalg  {R}^*(Y_1 ) \amalg  {R}^*(Y_2).
 \]
  There are three pattens $a_1 *_h a_2$ ($h\in SO(3)$), $a_1 * \theta$ and $\theta * a_2$ of elements in ${R}^*(Y_1 \# Y_2 )$, where $[a_1] \in  {R}^*(Y_1)$ and $[a_2] \in  {R}^*(Y_2)$. Suppose that $H^1_{[a_1]} (Y_1)=  T_{[a_1]} {R}^*(Y_1 ) $ and $H^1_{[a_2]}(Y_2)=  T_{[a_2]} {R}^*(Y_2 ) $. It is sufficient to prove $\dim H^1_{a_\#} (Y_1\# Y_2) = \dim T_{a_\#}{R}^*(Y_1 \# Y_2 )$ for any $a_\# \in {R}^*(Y_1 \# Y_2 )$. Fix critical points $a_1 \in {R}^*(Y_1 )$, $a_2 \in {R}^*(Y_2 )$ and $a_\# \in {R}^*(Y_1 \# Y_2 )$. The Meyer-Vietoris sequence of the local coefficient cohomology implies the existence of the following exact sequence:
  \[
  0 \to H^0_{a_1} (Y'_1) \oplus H^0_{a_2} (Y'_2)  \to   H^0(S^2) \to H^1_{a_\#}(Y_1\# Y_2) \to H^1_{a_1} (Y'_1) \oplus H^1_{a_2} (Y'_2) \to 0 
\]
  where $Y_i'$ is a punctured $Y_i$ for $i=1$ and $2$.
  The other sequence implies \[
  H^0_{a_i}(Y_i) \to H^0_{a_i} (Y'_i)\oplus \R^3 \to H^0_{\theta} ( S^2) \to H^1_{a_i}(Y_i) \to H^1_{a_i} (Y'_i) \to 0
 \]
 is also exact for $i=1$ and $2$. If both of $a_i$ are irreducible, then 
 \[
 H_{a_\#}^1 (Y_1\#  Y_2) \cong  H^1_{a_1} (Y_1) \oplus H^1_{a_2} (Y_2)\oplus \R^3.
 \]
 If $a_1$ is irreducible and $a_2=\theta$, then 
 \[
 H_{a_\#}^1 (Y_1\#  Y_2) \cong  H^1_{a_1} (Y_1) \oplus H^1_{a_2} (Y_2).
 \]
 This proves the desired result.  
  \qed \end{proof}
  \begin{cor}
  Let $Y$ be a finite connected sum of Seifert homology $3$-spheres.
  Then the Chern-Simons functional of $Y$ is Morse-Bott. 
 \end{cor}
 \begin{proof}
 It is shown that the Chern-Simons functional of any Seifert homology $3$-sphere is Morse-Bott in \cite{FS90}.  \cref{conn sum of MB} then gives the conclusion. 
 \qed \end{proof}

Here, we will introduce the invariant $l_Y$. For any Riemann metric $g_Y$ on $Y$, there exists a sequence of non-degenerate regular perturbations $\{\pi_n\}$ such that $\|\pi_n\| \to 0$. We define two quantities: 
\[
l(Y, g) := \min \left\{  \sup_{n \in \Z_{>0}} \# R^*_{\pi_n}  (Y)  \middle| \{\pi_n\} \text{: non-deg regular}, \|\pi_n\| \to 0 \right\} 
\in \Z_{\geq 0} \cup \{\infty\}
\]
and
\[
l(Y, g, r,i ):= \min \{  \sup_{n \in \Z_{>0}} \# \{a \in  \wt{R}^*_{\pi_n}  (Y) |  |cs_{\pi_n}(a) -r|<\lambda_Y , \ind (a)= i \}  |
\]
\[
 \{\pi_n\} \text{: non-deg regular}, \|\pi_n\| \to 0 \} \in \Z_{\geq 0} \cup \{\infty\}
\]
for a given $r \in \Lambda_Y$, where $\lambda_Y= \frac{1}{2} \min \set{ | a-b | | a,b \in  \Lambda_Y}$. 
We now give two invariants for homology $3$-spheres. 
\begin{defn}
We define invariants $l_Y$ and $l_{Y,r , i}$ by 
\[
l_Y := \min \set { l(Y, g_Y) | g_Y \text{ : Riemann metric} } \in \Z_{>0} \cup \{\infty\}
\]
and 
\[
l_{Y,r,i}:= \min \set { l(Y, g_Y,r,i)  | g_Y \text{ : Riemann metric}  } \in \Z_{>0} \cup \{\infty\}
\]
for given $r \in \Lambda^*_Y$ and $i \in \Z$.
\end{defn}
Note that $l_Y  =l_{-Y}$ and $l_{Y,r, i}= l_{-Y,-r, -i-3}$ by definition. 
We combine $l_{Y,r,i}$, $r_s(Y)$ and $\Gamma_Y(k)$ and define $l^s_Y$ and $l^k_Y$.
\begin{defn}
We set
\[
 l^s_Y:= \begin{cases}  1 \ \ \ \ \ \ \ \ \  \text{ if } r_s(Y)=\infty \\
 l_{Y,r_s(Y),1} \ \ \ \   \text{ if } r_s(Y)<\infty \\ 
 \end{cases}.
 \]
 and 
 \[
 l^k_Y:= \begin{cases}  1 \ \ \ \ \ \ \ \ \  \text{ if } \Gamma_Y(k)=\infty \\
 \dis \sum_{m \in \Z} l_{Y,\Gamma_Y(k), 1+ 8m} \ \ \ \   \text{ if } \Gamma_Y(k)<\infty \text{ and }k \in 2\Z+1\\ 
 \dis \sum_{m \in \Z}l_{Y,\Gamma_Y(k), 5+8m} \ \ \ \   \text{ if } \Gamma_Y(k)<\infty \text{ and }k \in 2\Z
 \end{cases}
 \]
 for $s \in [-\infty, 0]$ and $k \in \Z_{>0}$.
 \end{defn}
 \begin{lem}
 For any homology $3$-sphere $Y$, $l^s_Y \geq 1$ and $l^k_Y \geq 1$.
 \end{lem}
 \begin{proof}Suppose $r= r_s(Y)< \infty$ for $s \in [-\infty, 0]$. By the definition of $r_s(Y)$, for a sequence $\{ r_n\}_{n\in \Z_{>0}}$ such that $r_n>r$, $\dis \lim_{n\to \infty} r_n=r$ and $r_n \in \R \setminus \Lambda_Y$, we have $0\neq \theta_Y^{[s,r_n]} \in I^1_{[s,r_n]}(Y)$. If $r'<r$, then $0= \theta_Y^{[s,r']} \in I^1_{[s,r']}(Y)$. Then, for any Riemann metric $g_Y$ on $Y$ and any sequence of perturbations $\{\pi_n\}_{n\in \Z_{>0}}$ with $\| \pi_n\| \to 0$, there is a sequence $\{c_n\}_{n\in \Z_{>0}}$ of critical points of $\cs_{Y, \pi_n}$ such that $M^Y ( c_n , \theta)_{\pi_n} \neq \emptyset$, $\ind (c_n)=1$ for all $n$ and 
 \[
 \lim_{n\to \infty} \cs_Y (c_n) = r_s(Y).
 \]
  Therefore, $l_Y^s \geq 1$. Next we see $l^k_Y \geq 1$. Suppose $r= \Gamma_Y(k)<\infty$ for $k\in \Z_{>0}$. 
Suppose that $k$ is odd and $l^k_Y=0$. The assumption $l^k_Y=0$ implies that there exist 
a Riemann metric $g_Y$ on $Y$ and a sequence of perturbations $\{\pi_n\}_{n\in \Z_{>0}}$ with $\| \pi_n\| \to 0$ such that 
\begin{align}\label{empty}
\emptyset= \bigcup_{m \in \Z} \set {a \in  \wt{R}^*_{\pi_n}  (Y) |  |cs_{\pi_n}(a) -\Gamma_Y(k)|<\lambda_Y , \ind (a)= 1+ 8m } . 
\end{align}
Then, for $g$ and $\{\pi_n\}_{n\in \Z_{>0}}$, 
\[
\Gamma_Y (k) = \lim_{n \to \infty}  \left(
\inf_{\substack{\alpha \in C_*^{\Lambda}(Y, \pi_n), d^{\Lambda}(\alpha)=0 \\ D_1 U^j(\alpha) = 0 ( 1 \leq j < k-1) \\ 
D_1 U^{k-1}(\alpha) \neq 0  }}
\left\{
\mdeg(D_1U^k (\alpha))- \mdeg (\alpha)   
\right\}
\right).
\]
This implies there is a sequence $\{\al_n\}_{n \in \Z_{>0}}$ of elements in $C_1^{\Lambda} (Y, \pi_n)$ such that 
\begin{align}\label{gammay}
\lim_{n \to \infty} \left(\mdeg(D_1U^k (\alpha_n))- \mdeg (\alpha_n) \right)= \Gamma_Y(k) .
\end{align}
 We write $\dis \al_n = \sum_{i \in \Z_{>0} } q_i^n [c_i^n] \lambda^{s_i^n}$, where $q_i^n \in \q$, $[c_i^n] \in R^* (Y)_{\pi_n}$ and $s_i^n \in \R$ with $\dis \lim_{i\to \infty} s_i^n =\infty$. Then, by taking suitable lift $c_{i_n}^n$ of $[c_{i_n}^n]$ for each $n$, \eqref{gammay} implies  
 \[
 \lim_{n\to \infty} \cs_Y (c_{i_n}^n)  = \Gamma_Y (k).
 \]
 This contradicts \eqref{empty}. The proof for $k \in 2\Z$ is the same. 
 \qed
 \end{proof}
 The following proposition provides a relation between $l_Y$ and $l_{Y,r,i}$:
\begin{prop} 
We write $ \{a_1, \dots,  a_n\} = (0,1] \cap \Lambda^*_Y$. Then, we have 
\[
\sum_{1 \leq i \leq n, j \in \Z} l_{Y,a_i, j }  =l_Y. 
\]
\end{prop}
\begin{proof}
For any metric $g_Y$, there exists $\varepsilon>0$ such that any perturbation $\pi$ with $\| \pi \|< \varepsilon$, we have 
\[
\bigcup_{1 \leq i \leq n, j \in \Z} \set {a \in  \wt{R}^*_{\pi}  (Y) |  |cs_{\pi}(a) -a_i|<\lambda_Y , j= \ind (a)  } \cong R^*_{\pi}  (Y).
 \]
This implies the conclusion. \qed
\end{proof}
 We will see a connected sum formula of $l_Y$  and $l_{Y, r, i}$ under some assumptions in \cref{Morse type perturbation}.

\subsection{Calculation of $l_Y$}
In this section, we give several ways to calculate $l_Y$ or $l_{Y, r, i}$. 
\begin{lem}[Morse inequality for $\cs$]For an oriented homology $3$-sphere $Y$, the inequality $ l_Y \geq \sum_{i=0}^7 \rank I_i (Y) $ holds. 
\end{lem} 
\begin{proof}By the definition of instanton homology, we have
\[
\# R^*_{\pi}(Y)  \geq  \sum_{i=0}^7 \rank I_i (Y)
\]
for every non-degenerate regular perturbation.  This completes the proof. 
\qed \end{proof}
The following lemma give explicit calculations:
\begin{lem}For a Seifert homology $3$-sphere of type $\Sigma(p,q,r)$, 
\[
l_{\Sigma(p,q,r)}= 2| \lambda(\Sigma(p,q,r)) | ,
\]
where $\lambda(Y)$ is the Casson invariant of $Y$.
\end{lem}
\begin{proof}
For a Seifert homology $3$-sphere $\Sigma(p,q,r)$, it is shown in \cite{FS90} that $\cs_{\Sigma(p,q,r)}$ is non-degenerate and Floer indices of all of its critical points are even. Therefore, $\pi_n=(\emptyset, 0 , 0 )$ gives a sequence of non-degenerate regular perturbations. 
This implies the conclusion.  
\qed \end{proof}
We give calculations of $l_Y$ for the degenerate case $\Sigma(a_1, \cdots, a_n)$ in \cref{non-deg}. 

\begin{lem}For an oriented homology $3$-sphere $Y$ and $r \in \Lambda_Y$, 
 \[
 l_{Y,r, i} \geq \rank I_i^{[r+\lambda_Y, r-\lambda_Y]} (Y) 
 \]
 holds, where $\lambda_Y= \frac{1}{2} \min \set { | a-b | | a,b \in  \Lambda_Y }$. 
\end{lem} 
\begin{proof}Take a Riemann metric $g_Y$ and a sequence $\{\pi_n\}$ of non-degenerate regular perturbations such that 
\[
l(Y, g, r,i )=  \sup_{n \in \Z_{>0}} \# \set {a \in  \wt{R}^*_{\pi_n}  (Y) |  |cs_{\pi_n}(a) -r|<\lambda_Y, \ind (a)= i }
\]
and $\|\pi_n\| \to 0$. Then the chain complex of $I_i^{[r+\lambda_Y, r-\lambda_Y]} (Y)$ is generated by the elements $\{ a \in \wt{R}^*_{\pi_n} (Y) |r+\lambda_Y <cs_{\pi_n}(a)<r-\lambda_Y\}$. By the definition of $l_{Y,r, i}$, we have
\[ 
l(Y, g, r,i ) \geq  \rank I_i^{[r+\lambda_Y, r-\lambda_Y]} (Y) .
\]
\qed
\end{proof}

\section{Some remarks for $\{cs_{X,c}^j\}_{j \in \Z_{>0}} $} \label{Some remarks for QX}
In this section, we give several properties of $\{cs_{X,c}^j\}_{j \in \Z_{>0}} $ including the connected sum formula and the surgery formula. 
\subsection{Connected sum formula}
First we show a connected sum formula for $\{cs_{X,c}^j\}_{j \in \Z_{>0}} $. Let $X_1$ and $X_2$ be oriented closed 4-manifolds with fixed classes $c_1 \in H^1(X_1; \Z)$ and $c_2 \in H^1(X_2; \Z)$. Fix embeddings $l_i \colon S^1 \times D^3 \to X_i $ with $c_i (l_i(S^1\times * ))=1 $. For a diffemorphism $\psi :\partial (\im l_1) \to \partial (\im l_2)$, one can define a connected sum $X_1 \#_{\psi} X_2$ of $X_1$ and $X_2$ along $\psi$ by 
\[
X_1 \#_{\psi} X_2:= (X_1 \setminus \text{int}\im l_1) \cup_{\psi} (X_2 \setminus \text{int}\im l_2) .
\]
The class $c_1$ determines a class $c_\# \in H_1(X_1 \#_{\psi} X_2; \Z)$. 
We write the $j$-covering space corresponding to 
\[
\pi_1(X_i) \xrightarrow{\text{Ab}} H_1(X_i;\Z) \to \Z/ j \Z
\]
 by $p_i^j \colon (X_i)_{j, c} \to X_i$
 for $i=1$ and $i=2$. We fix lifts $\tilde{l}_i\colon S^1 \times D^3 \to (X_i)_{j, c}$ of $l_i \colon S^1 \times D^3 \to X_i $ for $i=1$, $2$.
\begin{prop}[Connected sum formula]\label{connected sum}
\[
\im\cs_{X_1, c_1} \cup \im\cs_{X_2, c_2}  \subset \im\cs_{X_1 \#_{\psi} X_2, c_\#} .
\]
\end{prop}
\begin{proof}Choose an element $\rho \in \wt{R}^* ((X_1)_{j, c_1})$ such that $\cs_{X_1, c_1}^j  (\rho) <1$. We can see that $j_{c_\#} X_1 \#_{\psi} X_2$ is obtained by gluing of $(X_1)_{j, c_1} \setminus \im \tilde{l_1} $ and $(X_2)_{j, c_2} \setminus \im \tilde{l_2} $ along $S^1 \times S^2$. The restriction of $\rho$ to $\partial (\im \tilde{l_1} )$ determines an element of $\hom (\Z, SU(2))/ SU(2)$. Then the flat connection  $\rho|_{\partial (\im \tilde{l_1} )= \partial (\im \tilde{l_2} )}$ can be extended whole of $(X_2)_{j, c_2} \setminus \im \tilde{l_2} $ using a homomorphism \[
\pi_1((X_2)_{j, c_2} \setminus \im \tilde{l_2}  ) \to H_1((X_2)_{j, c_2} \setminus \im \tilde{l_2}  ;\Z)  \to \Z \subset  SU(2).
\]
 We denote this extension of $\rho|_{\partial (\im \tilde{l_1} )= \partial (\im \tilde{l_2} )}$ by $\tilde{\rho}$. By construction, we have 
\[
\cs_{X_1, c_1}^j (\rho) =\cs_{X_1 \#_{\psi} X_2, c_\#}^j(\tilde{\rho}).
\]
This completes the proof. 

\qed \end{proof}

\subsection{Mapping tori}
Let $X_h(Y)$ be the mapping torus of a fixed orientation preserving diffeomorphism $h \colon Y \to Y$ on an oriented 3-manifold $Y$ and $P= Y \times SU(2)$. The map $h$ gives an action $h^* \colon {R}^*(Y) \to {R}^*(Y) $. We put 
\[
R^h(Y) := \{ \rho \in R^*(Y) | h^* \rho = \rho \}. 
\]
For the mapping torus, we have a convenient formula for $\cs_{X, [c] }$. 
\begin{prop}\label{mapping torus}Suppose $c=[Y] \in H_3(X_h(Y); \Z) $. Then 
\[
\im\cs_{X_h (Y), [Y] }^j= \im\cs_{X_{h^j}(Y), [Y]} = \left\{ \cs_Y(a)\in (0,1] \middle| a\in R^{(h^j)}(Y) \right\}
\]
 holds for any $j \in \Z_{ >0}$.
\end{prop}
\begin{proof}
In \cite{RS04}, it is shown that the inclusion $i: Y \to X_h(Y)$ induces a two-to-one correspondence 
\begin{align}\label{two to one}
i ^* \colon R^* (X_h(Y)) \to R^h(Y).
\end{align}
 We have the following commutative diagram: 
\[
  \begin{CD}
   R^* (X_h(Y))  @>cs_{X_h(Y), [Y]}>> (0,1] \\
  @V{i^*_Y}VV    @V{\id}VV \\
     R^h(Y)   @>{cs_Y}>> (0,1] . 
  \end{CD}
\]
This completes the proof. \qed
\end{proof}
\subsection{Surgery along $S^1 \times D^3$}
In order to prove \cref{emb1}, we need a surgery formula for $\cs^j_{X, c}$. 
Let $X$ be a closed connected oriented $4$-manifold with $0\neq c \in H_3(X;\Z)$ and $0\neq d \in H_1(X;\Z)$ satisfying $c\cdot d=0$.  Let $l : S^1 \times D^3 \to X$ be an embedding such that $[l|_{S^1\times \{0\} }]=d$ and $Y$ be a closed connected oriented $3$-submanifold of $X$ such that $[Y] = c$ and $Y \cap l  = \emptyset$.  We define the $4$-manifold $X_l$ obtained by surgery along $l$ by 
\begin{align}\label{22}
 X_l := X \setminus \operatorname{int}(\im l ) \cup_{S^1 \times S^2} D^2 \times S^2.  
 \end{align}
 We have two inclusion maps $i_1\colon X \setminus \im l \to X$ and $i_2\colon X \setminus \im l \to X_l$. Set $c':= i^*_1(\operatorname{PD} (c) ) \in H^1(X \setminus \im l ,  \Z)$.
Then we can take $0\neq c^* \in H^1(X_l ;\Z)$ such that $i^*_2(c')= c^*$.
\begin{prop}\label{surgery}For any $j \in \Z_{>0}$
\[
\im \cs^j_{X_l, c^*} \subset \im \cs^j_{X, c } .
\]
\end{prop}
\begin{proof}
Suppose that $\rho \in \wt{R}^* ((X_l)_{j, c^*})$ such that $\cs^j_{X_l,c^*} (\rho) <1 $. Note that $(X_l)_{j, c^*}$ has a decomposition 
\[
W[0,j-1]= W_0 \cup_Y W_1 \cup_Y \cdots \cup_Y W_{j-1},
\]
 where $W_i$ is a copy of $W_0=\overline{ X_{l} \setminus Y}$. Since $Y$ and $\im l$ are disjoint, $W_i$ can be written as 
 \[
( X_l \setminus (Y \cup D^2 \times S^2) )  \cup D^2 \times S^2. 
 \]
 We denote by $V_0 \subset W_0$ the submanifold corresponding to $D^2 \times S^2$. 
 We define
 \[
 W_*[0,j-1]:= (W_0 \setminus \operatorname{int}V_0) \cup_Y (W_1 \setminus \operatorname{int}V_1)  \cup_Y \cdots \cup_Y (W_{j} \setminus \operatorname{int}V_{j}),
 \]
 where  $(W_i, V_i)$ is a copy of $(W_0, V_0)$ for $i \in \Z_{>0}$.
 The flat connection $\rho$ determines flat connections $\rho^j$ and $\rho^j_*$ on $W[0,j-1]$ and $W_*[0,j-1]$ via pull-back. Note that by identifying the boundaries of $W_0 \setminus \operatorname{int}V_0 \cup_{S^1\times S^2} S^1 \times D^3 $, we recover $X$. Here we see that the flat connection $\rho^j_*|_{W_0 \setminus \operatorname{int}V_0  } $ extends to a flat connection $W_0 \setminus \operatorname{int}V_0 \cup_{S^1\times S^2} S^1 \times D^3$. To see this, we consider the restriction $(\rho^j_*)|_{\partial ( V_0 )= S^1 \times S^2 }$. Since the flat connection $(\rho^j_*)|_{\partial ( V_0 )}$ extends to $V_0$, the holonomies for all loops in $\partial V_0 =S^1 \times S^2$ are zero. Thus, $(\rho^j_*)|_{\partial ( V_0 )= S^1 \times S^2 }$ is isomorphic to trivial flat connection. This fact allows us to extend the flat connection $\rho^j_*|_{W_0 \setminus \operatorname{int}V_0  } $ to a flat connection on $W_0 \setminus \operatorname{int}V_0 \cup_{S^1\times S^2} S^1 \times D^3$ trivially. Therefore, the connection $\rho^j_*$ extends to a connection on 
 \begin{align}\label{aaaa}
(  W_0 \setminus \operatorname{int}V_0 \cup_{S^1\times S^2} S^1 \times D^3)\cup_Y \cdots \cup_Y (W_{j} \setminus \operatorname{int}V_j \cup_{S^1\times S^2} S^1 \times D^3).
\end{align}
Identifying the boundaries of \eqref{aaaa} gives a construction of $X_{j,c}$. Therefore the extension on \eqref{aaaa} gives a connection $\rho'_j$ on $X_{j,c}$. By the construction of $\rho'_j$, we conclude 
\[
\cs_{X,c}^j (\rho') = \cs_{X_l, c^*} (\rho). 
\]
 This completes the proof. 
\qed \end{proof}
\subsection{Calculations}
We give explicit calculations for several mapping tori of Seifert manifolds. 
Regarding $\Sigma(p,q,r)$ as $\{(x,y,z) \in \co^3 | x^p+y^q+z^r=0\} \cap S^5$, we define $\tau \colon \Sigma(p,q,r) \to \Sigma(p,q,r)$ and $\iota \colon \Sigma(p,q,r) \to \Sigma(p,q,r)$  by 
\begin{align} \label{tau}
\tau \colon (x,y,z) \mapsto (x,y, e^{\frac{2\pi i }{r}} z) 
\end{align}
and 
\begin{align}\label{iota}
\iota \colon (x,y,z) \mapsto (\overline{x},\overline{y}, \overline{z}) .
\end{align}
This gives the following calculations: 
\begin{prop}\label{mapping tori}Let $(p,q,r)$ be a relatively prime triple of positive integers. 
For any $j \in \Z_{>0}$, 
\[
\im\cs_{X_\tau (\Sigma(p,q,r) ), [\Sigma(p,q,r)]  } = \Lambda_{\Sigma(p,q,r)} \cap (0,1]
\]
and 
\[
\im\cs_{X_\iota (\Sigma(p,q,r) ), [\Sigma(p,q,r)]  } = \Lambda_{\Sigma(p,q,r)} \cap (0,1]. 
\]
\end{prop}
 \begin{proof} By applying \cref{mapping torus}, we have 
 \[
\im\cs_{X_{\tau^j}, [Y]} = \{ \cs(a)\in (0,1] | a\in {R}^*(Y)  (\tau^j)^*a =a \}= \Lambda_{Y} \cap (0,1].
\]
It is shown in \cite{CS99} and \cite{Sa02} that $\tau^* \colon  R^*(\Sigma(p,q,r)) \to R^*(\Sigma(p,q,r))$ and $\iota^* \colon  R^*(\Sigma(p,q,r)) \to R^*(\Sigma(p,q,r))$ are equal to the identity. 
This completes the proof. \qed
 \end{proof}
This property of $\iota^* \colon  R^*(\Sigma(p,q,r)) \to R^*(\Sigma(p,q,r)) $ is shown for any homology $3$-sphere of type $\Sigma(a_1, \cdots, a_n)$. Therefore, \cref{mapping tori} can be proved for such Seifert homology $3$-spheres. 
\section{Invariants of $2$-knots}\label{Invariants of 2-knots}
In this section, by the use of $\cs_{X, c}$, we will introduce an invariant of oriented 2-knots. 
\subsection{Invariants $\{cs_{K, j}\}_{j \in \Z_{>0}}$}\label{Inv QK}
Let $K$ be an oriented $2$-knot in $S^4$. It is known that the normal bundle $\nu_K$ of $K$ is always trivial. Moreover, trivializations of $\nu_K$ are unique up to isotopy. Therefore, we fix such a tubular neighborhood $ \nu_K\colon S^2 \times D^2 \to S^4$. Then we have an oriented homology $S^1 \times S^3$ defined by 
\[
X(K) :=D^3 \times S^1 \cup_{\nu_K} S^4 \setminus  \nu_K.
\]
Note that an orientation of $S^4$ gives an orientation of $X(K)$. If we have a Seifert hypersurface $Y_K$ of $K$, then we can regard $Y_K$ as a generator of $H_3(X(K); \Z)$ by the following discussion. First we regard $Y_K\setminus B(\varepsilon) $ as a submanifold in $S^4 \setminus  \nu_K$. Then we cap off $Y_K\setminus B^3 $ in $X(F)$. The orientation of $K$ determines a generator $[1_K]$ of $H_3(X(K) ; \Z)$  such that $[* \times S^1]\cdot [Y_K] =1$. 
Note that the class $[Y_K] \in H_3(X(K) ;\Z)$ satisfies \cref{assc}.
The class $[1_K]$ gives an isomorphism class of $\Z$-covering space 
\begin{align}\label{covering}
p_K \colon \widetilde{X}(K) \to X(K). 
\end{align}
We denote by $(X(K))_{j , 1_K} $ the total space of the $\Z/j\Z$ covering space corresponding to the kernel of the composite homomorphism $\pi_1(X(K)) \to \Z \to \Z/  j \Z$ for $j \in \Z_{>0}$. 
The map \eqref{covering} gives a covering map 
\[
p^j_K \colon \widetilde{X}(K) \to (X(K))_{j , 1_K} 
\]
for each $j \in \Z_{>0}$.

\begin{defn}\label{cs for k}
For an oriented knot $K$ and $j \in \Z_{>0}$, we define 
\[
cs_{K, j } :=\cs_{X(K), 1_K}^j : R( X(K)_{j, 1_K} ) \to (0,1 ] 
\]
\end{defn}
\begin{ex}For the unknot $U$, $X(K) \cong S^1 \times S^3$. Therefore, $\im\cs_{U,j} =\{ 1\}$ for any $j\in \Z$.
\end{ex}
The lemma below is an easy consequence of Van-Kampen's theorem. 
\begin{lem}\label{fund group} For any $2$-knot $K$, 
\[
\pi_1( S^4 \setminus K  ) \cong \pi_1( X (K)  ).
\]
Moreover, $\ker \psi_j \cong \ker \phi_j$, where $\psi_j$ and $\psi_j$ are introduced in \eqref{psij} and \eqref{phij}. In particular, $G_j( K) \cong \pi_1( X(K)_{j, 1_K} )$.
\end{lem}
Via \cref{fund group}, we regard $cs_{K, j }$ as maps from $R(K,j)$ to $(0,1]$. When $j=1$, we write $cs_K$ instead of $cs_{K,1}$. By \cref{func}, we see that $\im cs_{K, j }$ is an isotopy invariant for each $j \in \Z_{>0}$.
If we have a Seifert hypersurface $Y$, then we have the following formula: 
\begin{prop}[\cref{2-knot1}, 1]\label{value of QK}For any oriented $2$-knot $K$ and $j \in \Z_{>0}$, 
\[
\im cs_{K, j }  \subset (0,1] \cap \Lambda_Y,
\]
where $\Lambda_Y= \im \cs|_{\wt{R}(Y)}$.
\end{prop}
\begin{proof}
This is an immediate corollary of \cref{value of QX}. 
\qed \end{proof}
\begin{lem}[\cref{2-knot1}, 3]\label{j-prop}For any positive integer $ m$,
\[
\im\cs_{K, j } \subset \im\cs_{K,mj}
\]
for any $j \in \Z_{>0}$.
\end{lem}
\begin{proof}
This is a corollary of \cref{fund prop}. 
\qed \end{proof}
\begin{rem}\label{general cs_K} The functionals $\{\cs_{K,j}\}_{j \in \Z_{>0}}$ can be generalized to functionals $\{\cs_{(K,X), j} \}_{j \in \Z_{>0}}$ for any $2$-knot $K$ embedded into any closed oriented $4$-manifold $X$ with $0=[K] \in H_2(X; \Z) $. Moreover, \cref{value of QK} and \cref{j-prop} hold for $\{\cs_{K,j}\}_{j \in \Z_{>0}}$. 
\end{rem}
\subsection{Connected sum formula}
Let $K_1$ and $K_2$ be $2$-knots. We denote by $K_1 \# K_2$ the connected sum of $K_1$ and $K_2$.
\begin{prop} [\cref{2-knot1}, 2]
\[
\im\cs_{K_1,j}  \cup \im\cs_{K_2, j }  \subset \im\cs_{K_1\# K_2, j } 
\]
hold. 
\end{prop}
\begin{proof}
Since $X(K_1\# K_2) _{j, 1_{K_1\# K_2} }$ is obtained the connected sum of $X(K_1)_{j ,1_{K_1}}$ and $X(K_2)_{j ,1_{K_2}}$ along some embedded $S^1 \times D^3$, this is also a corollary of \cref{connected sum}. 
\qed \end{proof}
\subsection{Calculation for $\cs_{K,j}$}
First, we give simplest examples. 
\begin{prop}For any ribbon $2$-knot $K$, 
 $\im\cs_{K,j}=\{ 1\}$ holds. 
\end{prop}
\begin{proof} In \cite{Y69}, it is shown that the ribbon $2$-knots have the finite connected sums of $S^1 \times S^2$'s as Seifert hypersurfaces. By \cref{value of QK}, $\im\cs_{K,j} \in (0,1] \cap \Lambda_Y$ for any $j\in \Z_{>0}$. If $Y$ is a finite connected sum of $S^1 \times S^2$'s, then the space of $SU(2)$-representation of $Y$ is connected. This implies $\Lambda_Y= \Z$. This completes the proof. \qed
\end{proof}
Next, we give non-trivial calculations for twisted spun knots.
Let $k$ be an oriented knot in $S^3$. 
We denote by $K(k,m)$ the $m$-{\it twisted spun knot} of $k$. In \cite{Z65}, Zeeman showed $m$-fold branched covering space $\Sigma^m(k)$ gives a Seifert hypersurface of $K(k,m)$. When $m\neq 0$, $K(k,m)$ is a fibered 2-knot with fiber $\Sigma^m(k)$. Moreover, the monodromy is given by the covering transformation of $\Sigma^m(k)$.  Now, we give a proof of \cref{cal2}.
 \begin{prop}[\cref{cal2}]
 Let $T(p,q)$ be the $(p,q)$-torus knot and $M(p,q,r)$ the Montesinos knot of type $(p,q,r)$ for a pairwise relative prime tuple $(p,q,r)$ of positive integers. Let $k(p/q)$ be a $2$-bridge knot such that $\Sigma^2(k (p/q) ) = L(p,q)$, where $\Sigma^2(k)$ is the double branched cover of $k \subset S^3$.
\begin{enumerate}
\item For any $m \in \Z_{>0}$ and $j \in  \Z_{>0}$, 
 \[
\im\cs_{K(T(p,q),m)}^j =\im\cs_{\Sigma(p,q,m)}  , 
 \]
where $K(k, m)$ is the twisted $m$-spun knot of the knot $k$.
  \item For any $j \in  \Z_{>0}$, 
 \[
\im\cs_{K(M(p,q,r),2)}^j = \im\cs_{\Sigma(p,q,r)}  . 
 \]
 \item 
Here we also suppose that $p$ is odd and satisfies the condition: 
 \[
 \left\{ s \in \set{ 2, \cdots , p-2} \middle | \frac{s^2-1} {p} \in \Z \right\} = \emptyset.  
 \]
  For any $j \in  \Z_{>0}$, 
 \[
 \im\cs_{K(k (p/q),2)}^j = \left\{ -\frac{n^2r}{p} \mod 1 \middle| 1 \leq n \leq  \left\lceil  \frac{p}{2} \right\rceil   \right\} \cup \{1\},
 \] 
  where $r$ is any integer satisfying $qr\equiv -1 \mod p$.
  \end{enumerate}
 \end{prop}

{\em Proof of \cref{cal2}.} 
It is known that the $m$-fold branched cover of $T(p,q)$ (resp. the double branched cover of $M(p,q,r)$) is $\Sigma(p,q, m)$ (resp. $\Sigma(p,q,r)$)(\cite{BZ85}). Moreover, the covering transformations are given by $\tau$ and $\iota$ in \eqref{tau} and \eqref{iota}. As mentioned in the proof of \cref{mapping tori}, $\iota$ and $\tau$ induce bijection between $R^* (\Sigma(p,q,m) \setminus B )$ and $R^* (\Sigma(p,q,r) \setminus B )$, where $B$ are small open balls. On the other hand, if $m>0$, $K(T(p,q),m)$ is a fibered knot whose fiber is $\Sigma(p,q, m)$ and the monodromy is given by $\tau$. Similarly $K(M(p,q,r),2)$ is also a fibered knot whose fiber is $\Sigma(p,q, r)$ and the monodromy is given by $\iota$. Next, we prove the statement for a rational knot. It is known that the double branched cover of $k(p/q)$ is given by $L(p,q)$. The condition
 \[
 \left\{ s \in \set{ 2, \cdots , p-2} \middle | \frac{s^2-1} {p} \in \Z \right\} = \emptyset.  
 \]
 implies that every order $2$-diffeomorphism on $L(p,q)\setminus B$ induces maps $\id$ or $\psi_{p-1}$ on $\pi_1(L(p,q) \setminus B) \cong \Z/ p\Z$, where $B$ are small open balls and $\psi_{p-1}$ is given by $1 \mapsto p-1$. Therefore, every order $2$-diffeomorphism $h$ on $L(p,q)\setminus B$ satisfies $h^*=1 : R^* (L(p,q)) \to R^* (L(p,q))$.
In \cite{KK91}, Kirk and Klassen showed
\[
\left\{- \frac{n^2r}{p} \mod 1  \middle| 1 \leq n \leq   \left\lceil  \frac{p}{2} \right\rceil   \right\} \cup \{1\}= \Lambda_{L(p,q)} \cap (0,1],
\]
 where $r$ is any integer satisfying $qr\equiv -1 \mod p$.
This completes the proof. \qed

In \cite{FS90}, for Seifert homology $3$-sphere of type $\Sigma(a_1, \cdots ,a_n )$, Fintushel and Stern gave an algorithm to compute $\Lambda_{Y}$.

\section{Morse type perturbation}\label{Morse type perturbation}

In this section, we will prove the following theorem: 
\begin{thm}\label{finiteness} Let $r$ be an element of $\Lambda^*_Y$. Suppose that the Chern-Simons functional of $Y$ is Morse-Bott at the level $r$. 
Then 
\[
l_{Y,r, i} < \infty
\]
for any $i \in \Z$.
In particular, if $\cs$ is Morse-Bott, then 
\[
l_Y < \infty.
\]
\end{thm}
The lemma below is key to proving \cref{finiteness}. The proof uses the essentially the same technique as in the proof of Theorem 5.11 \cite{Sav02} and \cite{HC99}. 
\begin{lem}\label{MB per}Suppose that the Chern-Simons functional of $Y$ is Morse-Bott at the level $r\in \R/ \Z$.  We denote by $C_r:= {R}^*(Y) \cap \cs^{-1}(r)$.  Let $g_r \colon C_r \to \R$ be a Morse function. Then there exists a smooth family of perturbations $\{\pi_\varepsilon \}$ parametrized by $\varepsilon \in (0,\varepsilon_0)$ for some $\varepsilon_0>0$ and a small neighborhood of $U$ of $C_r$ such that the critical points of $\cs_{\pi_\varepsilon}|_U$ correspond to $\displaystyle  \operatorname{Crit} (g_r)$, $\pi_\varepsilon$ is non-degenerate on $U$ for any $\varepsilon  \in (0,\varepsilon_0)$, $\displaystyle \lim_{\varepsilon \to 0} \|\pi_\varepsilon\| \to 0$ and $\operatorname{supp} h \subset U$. Moreover, there is an embedding \[
L: \displaystyle  \operatorname{Crit} (g_r) \times [0, \varepsilon_0) \to \B^*(Y)
\]
 such that $\im L|_{ \operatorname{Crit} (g_r) \times \{\varepsilon\} }$ coincides with the critical point set of $\cs_{\pi_\varepsilon}|_U$ and $\im L |_{ \operatorname{Crit} (g_r) \times \{0\}} =  \operatorname{Crit} (g_r)$. Here, we consider the $L^2_k$-topology on $\B^*(Y)$ for a fixed $k>2$. 
\end{lem}
We define a nice class of perturbations called {\it Morse-Bott type perturbations}. We can take orientation preserving diffemorphisms $f^r_i \colon S^1 \times D^2 \to Y$ ($1\leq i\leq N$) such that the smooth map 
$\psi \: {\B}^*(Y) \to \R^N$ defined by 
\[
\psi ([A]):=  \left(\int_{D^2} \trace ( \hol(A)_{f^r_i(-,x)} d\mu  )\right)_{ 1\leq i \leq N} 
\]
 gives a diffeomorphism from $C_r$ to its image (\cite{Do02}). We fix a closed tubular neighborhood $p_r:N_r\to \psi(C_r)$ of $\psi(C_r)$ and a Euclidian metric on $N_r$.  Fix a smooth bump function $\rho_r \: \R^N \to \R$ such that $\rho_r (v )=1 $ if $|v|<1$, $\rho (v )=0$ if $|v|>2$ and $\rho_r$ depends only on $|(x,t)|$, where $|-|$ is a metric on $N_r$. We consider the pull-back $p^* (g_r\circ \psi^{-1} ) \: N_r \to \R$. 
 \begin{defn}[Morse-Bott type perturbation]
 Now we define $h \:  {\B}^*(Y) \to \R$ by the composition of $\psi \: {B}^* (Y) \to \R^N$ and the function $q \: \R^N \to \R$ given by
 \[
 q_r (x) := \begin{cases} \rho_r(x )  g_r\circ \psi^{-1} \circ p( x)   \text{ if } x \in N_r  \\ 
 0  \text{ if } x \notin N_r  .
 \end{cases}
 \]
 We call a pair $\pi^r(g) := ( (f^r_i),  (h_i= \trace ), q_r) \in \cal{P}^*(Y)$ a {\it Morse-Bott type perturbation} at $C_r$.  When $\cs_Y$ is Morse-Bott, depending on the choice of $g : R^*(Y) \to \R$, we can define Morse-Bott perturbations $\pi(g)\in  \cal{P}^*(Y)$ in the same way. 
 \end{defn}
 Now, we give a proof of \cref{MB per}.
 
{ \em Proof of \cref{MB per}.} 

 We take a Morse-Bott type perturbation $\pi^r(g)$ for $C_r$. 
 Put $\cs_{Y, \pi_t}:= \cs_Y +t  h_{\pi^r(g)}  \colon {\B}^*(Y) \to \R$. 
For a fixed element $x \in {\B}^*(Y)$, we consider an essentially self-adjoint elliptic operator 
\[
\operatorname{Hess}_{x}(cs_Y)= *d_x : \ker d_x^* \cap L^2_k(\Om^1_Y \otimes \su) \to \ker d_x^* \cap L^2_{k-1}(\Om^1_Y \otimes \su) .
\]
The formal tangent bundle $T^k{\B}^*(Y)$ of ${\B}^*(Y)$ is defined as the quotient bundle of 
\[
\A^* (Y) \times \bigcup_{x \in \A^* (Y)} ( \ker d^*_x\cap L^2_k(\Om^1_Y \otimes \su))
\]
  by the action of $\G(Y)$.
  Then, the operator $\operatorname{Hess}_{x}(cs)$ defines a bundle map 
  \[
  \operatorname{\bf Hess}_{x}(cs) : T^k{\B}^*(Y) \to T^{k-1}{\B}^*(Y). 
  \]
  We define $\varepsilon_0>0$ by 
\[
\varepsilon_0 := \min \{ |\lambda| \ | \lambda \text{ is a non-zero eigenvalue of }
\]
\[
\operatorname{Hess}_{x}(cs): \ker d_x^* \cap L^2(\Om^1_Y \otimes \su)\to  \ker d_x^* \cap L^2(\Om^1_Y \otimes \su) , x \in C_r \}.
\]
We take an open neighborhood $U$ of $C_r$ in ${\B}^*(Y)$ such that
\[
\left\{ \frac{1}{2}\varepsilon_0 \right\}   \cap  \{ |\lambda| \ | \lambda \text{ is an eigenvalue of }
\]
\[
\operatorname{Hess}_{x}(cs):  \ker d_x^* \cap L^2(\Om^1_Y \otimes \su)\to  \ker d_x^* \cap L^2(\Om^1_Y \otimes \su), x \in U \}= \emptyset.
\]
Then,  \[
\displaystyle L^k_0:= \bigcup_{x \in U} \left\{\text{$L^2$-eigenspaces of } \operatorname{Hess}_{x}(cs) \text{ whose eigenvalues }\lambda \text{ satisfy } |\lambda| < \frac{ 1}{2}\varepsilon_0 \right\} \to U
\]
 gives a finite rank subbundle of $T^kU:= T^k {B}^*(Y)|_U$. This gives a decomposition $T^kU = L^k_0 \oplus L^k_1$.
 We have the following section: 
\[
\operatorname{\bf grad}_1: U \times (-\varepsilon,\varepsilon) \to  L^{k-1}_1 , (c,t) \mapsto \pr_{L^{k-1}_1}  \grad_x (\cs_{Y, \pi_t}) . 
\]
For a point $(c,t) \in U \times (-\varepsilon,\varepsilon)$ and a small neighborhood $V_{x,}$ of $x$, by taking a trivialization, we can regard the section $\operatorname{\bf grad}_1$ as a smooth map 
\[
\operatorname{\bf grad}'_1: V_{x} \times (-\varepsilon,\varepsilon) \to (L^{k-1}_1)_x. 
\] 
Since $\cs_Y$ is Morse-Bott function, $d \operatorname{\bf grad}'_1|_{V_x \times \{0\}}$
is surjective for $x \in C_r$. Therefore if $|t|$ is sufficiently small, $d\operatorname{\bf grad}'_1|_{V_x \times \{t\}}$ is sujective. Thus $\operatorname{\bf grad}_1^{-1}(0)$ is a smooth submanifold and diffeomorphic to $C_r \times (-\varepsilon,\varepsilon)$ for a small $\varepsilon>0$. We take such a diffeomorphism $H \: C_r \times (-\varepsilon,\varepsilon) \to \operatorname{\bf grad}_1^{-1}(0)$. Then, the critical points of $cs_{Y, \pi_t}$ can be seen as the zero set of the following map: 
\[
\operatorname{\bf grad} :=\operatorname{\bf grad}_0 \circ H :C_r \times (-\varepsilon,\varepsilon) \to L_0^{k-1}.
\]
For $x \in C_r$ and its neighborhood $V_x$ in $C_r$, we regard $\operatorname{\bf grad}$ as a smooth map
\[
\operatorname{\bf grad}' =\text{trivialization} \circ \operatorname{\bf grad}_1 \circ H :V_x \times (-\varepsilon,\varepsilon) \to (L_1^{k-1})_x. 
\]
We have the following Taylor expansion of $\operatorname{\bf grad}' $ with respect to $t$ at $t=0$: for $x=(c,t) \in C_r\times (-\varepsilon, \varepsilon )$, 
\begin{align*}
\operatorname{\bf grad}' (c,t)  = & \operatorname{\bf grad}'_{(c,0)}+t \frac{d}{dt} (\operatorname{\bf grad}'_{(c,t)} )|_{t=0} + O(t^2)\\ 
=  &t \frac{d}{dt} ( \pr_{(L_1)_x}  \grad_x (\cs_{Y, \pi_t} ) \circ H(c,t))|_{t=0}+ O(t^2) \\
= & t \frac{d}{dt} ( \pr_{(L_1)_x} \grad_x (t \rho_r(x )  g_r\circ \psi^{-1}\circ p \circ \psi(x)  ) ) \circ H(c,t))|_{t=0} +  O(t^2)  \\
= &  t\pr_{(L_1)_x}\grad_x ( g_r\circ \psi^{-1}\circ p\circ \psi ) \circ d_tH(c,t)) +  O(t^2) .
\end{align*}

Now we have a $ C^2$ section $\operatorname{\bf grad}^* \: C_r \times (-\varepsilon,\varepsilon) \to L_1$ given by
\[\operatorname{\bf grad}^*  (c,t):= 
\begin{cases} 1/t \operatorname{\bf grad}' (c,t) \text{ if } t\neq 0 \\ 
 \pr_{L_1}\grad_x ( g_r\circ \psi^{-1}\circ p\circ \psi )  \circ d_tH(c,t))  \text{ if } t=0 \\
 \end{cases}.
 \]
 Then $\operatorname{\bf grad}^* (c,t)$ is transverse when $t=0$. Therefore, for $t$ sufficiently small, $\operatorname{\bf grad}^*|_{C_r \times \{t \}}$ is transverse. One can see that the zero set of $\operatorname{\bf grad}^* (c,0)$ corresponds to the set 
 \[
 \left\{ c \in C_r \middle| \pr_{L_1}\grad_x ( g_r\circ \psi^{-1}\circ p\circ \psi )  \circ d_tH(c,0) = 0 \right\} \cong \operatorname{Crit}(g_r) \subset C_r.
 \]
We construct an embedding $L: \displaystyle  \operatorname{Crit} (g_r) \times [0, \varepsilon_0) \to \B^*(Y)$. Note that for $\varepsilon$ sufficiently small, $\operatorname{\bf grad}^* \: C_r \times (-\varepsilon,\varepsilon) \to L_1$ is transverse to the zero section. By the implicit function theorem, one can see that $(\operatorname{\bf grad}^*|_{C_r \times (-\varepsilon,\varepsilon) } )^{-1} (0)$ is a compact $1$-dimensional manifold. 
 If we take $\varepsilon$ sufficiently small, then
 \[
(\operatorname{\bf grad}^*|_{C_r \times \{0\} } )^{-1} (0)  \to (\operatorname{\bf grad}^*|_{C_r \times (-\varepsilon,\varepsilon) } )^{-1} (0) \to (-\varepsilon,\varepsilon)
\]
gives a fiber bundle. Note that $(\operatorname{\bf grad}^*|_{C_r \times \{0\} } )^{-1} (0)$ is a finite set and $(\operatorname{\bf grad}^*|_{C_r \times (-\varepsilon,\varepsilon) } )^{-1} (0)$ is a union of $1$-dimensional open intervals. 
As a trivialization of the above bundle on $[0, \varepsilon)$, we have a diffeomorphism \[
 L' \:  (\operatorname{\bf grad}^* |_{C_r \times \{0\} } )^{-1} (0)\times [0, \varepsilon) \to  \bigcup_{ t\in [0, \varepsilon) } (\operatorname{\bf grad}^*|_{C_r \times \{t \} } )^{-1} (0). 
 \]
  
 Therefore, the set of critical points of $(cs+t h_{\pi^r(g)})|_{U}$ corresponds to $\operatorname{Crit}(g_r)$ for small $t>0$. The embedding $L$ is given by the composition of $L'$ and the following map
 \[
\bigcup_{ t\in [0, \varepsilon) } (\operatorname{\bf grad}^*|_{C_r \times \{t \} } )^{-1} (0) \xrightarrow{\text{inclusion}}  C_r\times [0, \varepsilon) \xrightarrow{\pr_{C_r}} C_r  \to \B^*(Y).
 \]
 One can see that $(cs_Y+t h_{\pi^r(g)})|_{U}$ is non-degenerate at $x \in U$ if and only if $\operatorname{\bf grad}^*|_{U\times \{ t\} }$ is transverse to zero section. Since $\operatorname{\bf grad}^*|_{U\times \{0\} }$ is transverse to zero section and transversality is an open condition, $\operatorname{\bf grad}^*|_{U\times \{t\} }$ is transverse to zero section for sufficiently small $t>0$.

 \qed

Using \cref{MB per} and some technique of \cite{SaWe08}, we will show the following lemma: 
\begin{lem}\label{key}Suppose that the Chern-Simons functional of $Y$ is Morse-Bott at the level $r\in \Lambda^*$.
There exists a family of perturbations $\pi_n=(f, h_n, q_n) \in \mathcal{P}^*(Y, g)$ such that $\| \pi_n \|\to 0$, the perturbations $\pi_n$ are non-degenerate and regular and 
\[
\sup \# \set { a \in  \wt{R}^*_{\pi_n}(Y)| r-\lambda_Y <cs_{Y, \pi_n}(a) < r+  \lambda_Y  } < \infty, 
\]
where $\lambda_Y = \frac{1}{2} \min \set { |a-b| | a, b \in \Lambda^*_Y, a\neq b }$.
In particular, $\dis \sum_{i \in \Z}l_{Y,r, i}<\infty$.
\end{lem}
\begin{proof}We regard $r$ as an element of $\R/\Z$ via mapping $\R$ to $\R/\Z$.
Suppose that $C_r= {R}^*(Y) \cap \cs_Y^{-1}(r)$ is a submanifold in $\B^*(Y)$.
We take a Morse function $g_r\: C_r \to \R$. Then, by taking a Morse-Bott perturbation $\pi^r(g)$ \cref{MB per}, we obtain a sequence of perturbation $\{\pi_t\}_{t \in (0,\varepsilon_0) }= \{(f,h,\varepsilon q)\}$ such that 
\[
cs_{Y, \pi_t}:{ \B}^*(Y)\to \R /\Z
\]
satisfying the conclusion of \cref{MB per}.
Next, we perturb the other part ${R}^*(Y) \setminus \cs^{-1} (r )$.  By the argument to construct perturbations as in \cite[Subsection 5.5.1]{Do02}, one can take a sequence of perturbations $\{\pi'_n\}_{n\in \Z_{>0}}= \{(g,h_n, q_n)\}_{n \in \Z_{>0}}$ supported in a small neighborhood of ${R}^*(Y) \setminus \cs_Y^{-1} (r )$ such that $\cs_{\pi'_n}$ is non-degenerate for each $n$ and $\dis \lim_{n\to \infty} \|\pi_n'\|\to 0$. Then, $\pi_n'' := (f\cup g, h'_n+ h, q'_n+ \frac{1}{n} q)$ satisfies the following conditions: for a large $n>0$, 
\[
 \# \set { a \in  \wt{R}^*_{\pi_n}(Y)| r-\lambda_Y <cs_{Y, \pi_n}(a) < r+  \lambda_Y  } = \# \operatorname{Crit}(g), 
\]
where $\lambda_Y= \frac{1}{2} \min \set { | a-b | | a,b \in  \Lambda_Y }$.
Using the technique of \cite[Section 8]{SaWe08}, one can add a small perturbation $\{\pi^*_n\}= \{(f , h^*_n+ h_n,q^*_n+ q_n)\}$ such that the perturbation $\pi^*_n$ is non-degenerate and regular for each $n\in \Z_{>0}$ and the critical point sets of $\cs_{\pi_n}$ and $\cs_{\pi^*_n}$ coincide. This completes the proof. 
\qed \end{proof}
\begin{thm}\label{non-deg}For a Seifert homology $3$-sphere of type $\Sigma(a_1, \dots , a_n)$, 
\[
l_{ \Sigma(a_1, \dots , a_n)} =2 | \lambda ( \Sigma(a_1, \dots , a_n)) | .
\]
\end{thm}
\begin{proof}Saveliev showed $R^*(\Sigma(a_1, \dots , a_n))$ has a perfect Morse function whose critical points have odd Floer indices in \cite{Sa02}. By the use of a Morse-type perturbation as in \cref{MB per}, one can see that, there is no differential for such perturbations. This completes the proof. 
\qed
\end{proof}
The following theorem provides a connected sum formula of $l_Y$ under suitable assumptions:
 \begin{thm}[Connected sum formula of $l_Y$] \label{conn sum of ly}Suppose that $Y_1$ and $Y_2$ are Seifert homology $3$-spheres, then
 \[
 l_{Y_1\# Y_2} \leq 4l_{Y_1} l_{Y_2} + l_{Y_1} + l_{Y_2}.
 \] 
 \end{thm}
 \begin{proof} The critical submanifold of $\cs_{Y_1 \# Y_2}$ in $\B^*(Y_1 \#Y_2)$ is given by 
 \[
 R^*(Y_1 \# Y_2) =  R^* (Y_1 ) \amalg R^*(Y_2)  \amalg R^* (Y_1 ) \times  R^*(Y_2) \times SO(3).
 \]
 Since $Y_1$ and $Y_2$ are Seifert homology $3$-spheres, the Chern-Simons functionals of $Y_1$ and $Y_2$ have Morse-Bott type perturbations such that the critical points correspond to $\operatorname{Crit}(f_1)$ and $\operatorname{Crit}(f_2)$, where $f_1:  R^* (Y_1 )\to \R$ and $f_2:  R^* (Y_2 ) \to \R$ are perfect Morse function such that $l_{Y_1} = \# \operatorname{Crit}(f_1)$ and $l_{Y_2} = \# \operatorname{Crit}(f_2)$. Note that $SO(3)$ has a Morse function $s$ whose critical point set is the four point set. Then $R^*(Y_1 \# Y_2)$ has a Morse function such that the number of critical points is $4l_{Y_1} l_{Y_2} +  l_{Y_1} + l_{Y_2}$. Thus, one can take Morse-Bott type perturbations whose critical points correspond to $4l_{Y_1} l_{Y_2} +  l_{Y_1} + l_{Y_2}$ points. This completes the proof. 
 \qed
 \end{proof}
 The connected sum formula below is useful to calculate $l^s_Y$ and $l^k_Y$.
 \begin{thm}\label{conn sum of lys}Let $r$ be an element in $\Lambda_{Y_1 \# Y_2}$. Suppose that $Y_1$ and $Y_2$ are Seifert homology $3$-spheres. Then
  \[
 l_{Y_1\# Y_2,r, i} \leq 4\sum_{\substack{r=r_1+ r_2\\ i=i_1+i_2}}   l_{Y_1,r_1, i_1} l_{Y_2,r_2, i_2} + l_{Y_1, r, i} + l_{Y_2, r, i}.
 \] 
 \end{thm}
 \begin{proof}
 We decompose the set 
 $R^* (Y_1 \# Y_2) \cap \set { a  | \ind (a) = i ,\cs_{Y_1 \# Y_2} (a) =r } $
 as the union of 
 \[
 \coprod_{\substack{r=r_1+ r_2\\ i=i_1+i_2}} \set { a \in R^* (Y_1 )  | \ind (a) = i_1 ,\cs_{Y_1 } (a) =r_1 } \times \set { a\in R^* (Y_2 )   | \ind (a) = i_2 ,\cs_{Y_2 } (a) =r_2 } \times SO(3)
 \]
  \[
  \text{ and }
   \set { a\in R^* (Y_1)   | \ind (a) = i ,\cs_{Y_1 } (a) =r } \amalg \set { a\in R^* (Y_2 )   | \ind (a) = i ,\cs_{Y_2 } (a) =r }.
  \]
  Then, the same proof as for \cref{conn sum of ly} can be used to prove the desired result. \qed
\end{proof}
We will calculate $l_Y$, $l_Y^s$ and $l_Y^k$ for a certain class of homology 3-spheres in \cref{Examples}. 
\section{Convergence theorem and Proof of main theorems}\label{Convergence}
\subsection{Convergence theorem}
In this section, we fix an oriented homology $3$-sphere $Y$ embedded in an oriented negative definite 4-manifold $X$. In this section, we assume that $H_1(X;\R)\cong \R$ and $0 \neq [Y ] = c \in H_3(X; \Z)$. Let $ W_0$ be the oriented compact $4$-manifold with $\partial W_0 = Y \cup (-Y)$ obtained by taking the closure of $X \setminus Y$. For a positive integer $l$, the manifold $W[0,l]$ is the compact oriented $4$-manifold obtained from $X$ defined in \eqref{def of W} and $W^*[0,l]$ is the cylindrical end $4$-manifold written by $(W[0,l])^*$ in \eqref{def of W*}. Here we recall $\partial W_i = Y_i^+ \cup Y_i^-$. Fix a small collar neighborhood of $Y_i^+$ in $W^*[0,l]$ denoted by 
\begin{align}\label{collar}
Y_i^+\times I  \subset W^*[0,l]
\end{align} and a Riemann metric $g_Y$ on $Y$. We also fix a Riemann metric $g_{W^*[0,l]}$ on $W^*[0,l]$ such that the restrictions of $g_{W^*[0,l]}$ to $Y\times (-\infty, 0]$, $Y\times [0,\infty)$ and $Y^+_j \times I$ for $j\in \{0,\cdots, l^s_Y\}$ equal $g_Y\times dt^2$.

In this setting, we will show the following existence theorem:
\begin{thm}\label{seq}
\begin{enumerate}
\item 
Suppose $r_s(Y)<\infty$ for some $s\in [-\infty, 0]$.  
Then, for any $l \in \Z_{\geq 0} $, a sequence of non-degenerate regular perturbations $\{\pi_n\}$ with $\| \pi_n\|\to 0$ and a sequence of perturbations $\{\pi^n_{W^*[0,l-1]} \}$ of ASD equations on $W^*[0,l]$ compatible with  $\{\pi_n\}$ on $Y \times \R_{ \leq 0}$ and $Y \times \R_{ \geq 0}$,  there exist sequences of critical points $\{a_n\}_{n\in \Z_{>0}}$ and $\{b_n\}_{n \in \Z_{>0}} $ of $\cs_{\pi_n}$ such that 
\[
M(a_n, W^*[0,l-1] , b_n) \neq \emptyset \text{, }\displaystyle \lim_{n\to \infty}\cs_{\pi_n}(a_n)  =r_s(Y),
\]
and $\cs_{\pi_n} (a_n) - \cs_{\pi_n} (b_n) \to 0$ as $n\to \infty$. 
\item Suppose that $\Gamma_{-Y}(k)<\infty$ for some $k \in \Z_{>0}$. 
Then, for any $l\in \Z_{\geq 0} $, a sequence of non-degenerate regular perturbations $\{\pi_n\}$ with $\| \pi_n\|\to 0$ and a sequence of perturbations $\{\pi^n_{W^*[0,l-1]} \}$  of ASD equations on $W^*[0,l]$  compatible with  $\{\pi_n\}$ on $Y \times \R_{ \leq 0}$ and $Y \times \R_{ \geq 0}$,  there exist sequences of critical points $\{a_n\}_{n\in \Z_{>0}}$ and $\{b_n\}_{n \in \Z_{>0}} $ of $\cs_{\pi_n}$ such that 
\[
M(a_n, W^*[0,l-1] , b_n) \neq \emptyset \text{, }\displaystyle \lim_{n\to \infty}\cs_{\pi_n}(a_n)  =\Gamma_{-Y}(k),
\]
and $\cs_{\pi_n} (a_n) - \cs_{\pi_n} (b_n) \to 0$ as $n\to \infty$. 
\end{enumerate}

\end{thm}
\begin{rem}In \cite{D18} and \cite{NST19}, Daemi and Nozaki-Sato and the author proved similar existence results for solutions of perturbed ASD-equations.  The author expects that \cref{seq} can be proved for $\mathcal{J}_Y(k,s)$ which will be defined in \cite{DST19} for positive $k\in \Z$. 
\end{rem}
\begin{proof}
 The proof consists of two parts. 
\begin{enumerate}
\item Put $r_s(Y)=r$. Take a sequence of positive $\{ \epsilon_n\}$ such that $\epsilon_n \to 0$ and $\pi_n \in \mathcal{P}^* (Y, s,r+ \epsilon_n, g)$ for each $n$.
We consider the cobordism map $CW[0,l] \: CI^1_{[s,r+ \varepsilon_n]}(Y) \to  CI^1_{[s,r+ \varepsilon_n]} (Y)$ given by counting the moduli spaces $M (a,W^*[0,l] ,b)$. 
In \cite{NST19}, we showed the counts of the following oriented 0-dimensional compact manifolds:
\begin{itemize} 
\item 
$\displaystyle
\bigcup_{b \in \wt{R}^\ast(Y)_{\pi_n}, \ind (b)=0, \cs_{\pi_n}(b)<r+\varepsilon_n } M^Y(a,b)_{\pi_n}/\R\times M(b, W^\ast[0,l] , \theta)
$
\item $\displaystyle
 M^{Y_1}(a,\theta)_{\pi_n} /\R\times M(\theta,W^\ast[0,l],\theta) 
$
\item $-\displaystyle
\bigcup_{c \in \wt{R}^\ast(Y_2)_{\pi_2}, \ind (b)=1, \cs_{\pi_2} (c) <r+\varepsilon_n } M(a,W^\ast[0,l], c) \times M(c,\theta)_{\pi_n}/\R
$
\end{itemize}
are zero for each generator of $a \in C_1^{[s,r+ \epsilon_n]} (Y)$.
This implies that 
\begin{align} \label{important}
  n^{[s,r+\varepsilon_n] }  \partial^{[s,r+\varepsilon_n] } (a)  + c(W)\theta^{[s,r+\varepsilon_n ]}_{Y}(a)= \theta^{[s,r+\varepsilon_n ]}_{Y}CW^{[s,r+\varepsilon_n ] } (a)
 \end{align}
 for any $a \in CI^{[s,r+\varepsilon_n ]} _1(Y)$ for some $c(W)>0$. By the choice of $r$, for an element $a \in CI^{[s,r-\lambda_Y]} (Y)$ with $\partial^{[s,r-\lambda_Y]}(a)=0$, the equation $\theta^{[s,r-\lambda_Y]}_Y (a) = 0$ holds. However, since $[\theta^{[s,r+\varepsilon_n ]}_Y]\neq 0$ for any $n$, we have a sequence $\{\zeta_n\}$ of elements in $CI^{[r-\lambda,r+\varepsilon_n ]} _1(Y)$ such that $\partial^{[s,r+\varepsilon_n ]}(\zeta_n )=0$ and $\theta^{[s,r+\varepsilon_n]}_Y(\zeta_n) \neq 0$. 
 Then we put $a=\zeta_n$ and obtain the following equations: 
 \begin{align} \label{imp}
 c(W)\theta^{[s,r+\varepsilon_n ]}_{Y}(\zeta_n)= \theta^{[s,r+\varepsilon_n ]}_{Y}(CW^{[s,r+\varepsilon_n ] }(\zeta_n))
 \end{align}
 for each $n$. Since the left hand side of \eqref{imp} is non-zero, $\theta^{[s,r+\varepsilon_n ]}_{Y}CW^{[s,r+\varepsilon_n ] }(\zeta_n)$ is also non-zero. Since $CW^{[s,r+\varepsilon_n ] }(\zeta_n)$ is a cycle, if we write 
 \[
 \zeta_n = \sum_i s_i^n a^n_i \text{ and }CW^{[s,r+\varepsilon_n ] }(\zeta_n)= \sum_j r^n_j b^n_j
 \] ($s^n_i , r^n_j \in \q$), then there exist $i_0$ and $j_0$ such that $r + \epsilon_n>\cs_{\pi_n} (b^n_{j_0})>r -\lambda_Y$ and 
 \[
 M(a_{i_0}^n, W^*[0,l-1] ,b^n_{j_0})\neq \emptyset. 
 \]
 Put $a_n:= a_{i_0}^n$ and $  b_n:= b^n_{j_0}$.
By the choices of $\{a_n\}_{ n \in \Z_{>0}}$ and $\{b_n\}_{n\in \Z_{>0}}$, we conclude that $$\dis \lim_{n\to \infty} \cs_{Y, \pi_n} (a_n)  =\lim_{n\to \infty} \cs_{Y,\pi_n} (b_n) =r_s(Y) \dis .$$
 
\item Fix $1 \leq l  \leq l_Y^k$. We will also use the formula essentially showed in \cite{Fr02}. 
There exists a sequence $\{b_j \}$ of rational numbers such that 
 \begin{align} \label{Daemi1}
 D_1^{Y} U^{k-1}_{Y} CW[0,l] (a) = c(W[0,l]) (D_1^Y U^{k-1}_{Y} (a) ) + \sum_{1 \leq j <k-1} b_j  D_1^Y U^{j}_{Y} (a),
 \end{align}
  where $U_{Y}$ and $D_1^Y$ (resp. $U_{Y}$ and $D_1^{Y}$) are $U$-map and $D_1$ map for $Y$ (resp. for $Y$) and 
  \[
  CW[0,l]\: C^{\Lambda}_i(Y) \to C^{\Lambda}_i(Y)
  \]
   is a cobordism a map as in the same in \cite{D18}. We take a sequence $\{\zeta_n\} \subset C_*^{\Lambda}(Y)$ such that 
 $ d^{\Lambda}(\zeta_n)=0$, $D_1 U^j(\zeta_n) = 0 ( 1 \leq j < k-1)$, $D_1 U^{k-1}(\zeta_n) \neq 0$ and 
  \begin{align}\label{Daemi2}
    \Gamma_{-Y}(k) = \lim_{ |\pi_n| \to 0} \mdeg (D_1 U^{k-1}(\zeta_n)  ) - \mdeg (\zeta_n). 
    \end{align} 
   By combining \eqref{Daemi1} and \eqref{Daemi2}, we have 
   \[
   D_1^{Y} U^{k-1}_{Y} CW[0,l] (\zeta_n) = c(W[0,l]) (D_1^Y U^{k-1}_{Y} (\zeta_n) ).
   \]
   Thus, $\mdeg( D_1^{Y} U^{k-1}_{Y} CW[0,l] (\zeta_n))= \mdeg  (D_1^Y U^{k-1}_{Y} (\zeta_n) )$. 
   Note that elements $\{CW[0,l] (\zeta_n)\}_{n\in \Z_{>0}} $ also satisfy the following three equations 
   \[
    d^{\Lambda}(c)=0, D_1 U^j(c) = 0 ( 1 \leq j < k-1), \text{ and }D_1 U^{k-1}(c) \neq 0.
    \] 
   This implies 
   \[
   \Gamma_{-Y}(k) \leq \mdeg ( D_1^{Y} U^{k-1}_{Y} CW[0,l] (\zeta_n)) - \mdeg (  CW[0,l] (\zeta_n)) .
   \] 
Since
\[
\mdeg (  \zeta_n)  \leq 
\mdeg (  CW[0,l] (\zeta_n)) + \delta_n
\]
is proved in \cite{D18}, where $\{\delta_n \}$ is a sequence of positive number with $\delta_n \to 0$, we have 
   \begin{align*}
 \lim_{n\to \infty}   \mdeg ( D_1^{Y} U^{k-1}_{Y} CW[0,l] (\zeta_n))- \mdeg (CW[0,l] (\zeta_n)  ) = \Gamma_Y(k).
   \end{align*}
   This implies that 
   \begin{align}\label{identified}
   \displaystyle  \lim_{n \to \infty}\left( \mdeg (CW[0,l] (a_n)  ) - \mdeg (a_n)\right)=0 . 
   \end{align}
We write $\zeta_n$ by $\dis \sum_{1\leq i \leq N} m^n(i) \lambda^{r^n_i} a^n_i$ and 
\[
CW[0,l] (\zeta_n) =  \sum_{i, j} M(b^n_j,W^*[0,l],  a^n_i )m^n(i) \lambda^{r^n_i- \mathcal{E}(b^n_j, a^n_i)} b^n_j.
\]
This implies for some $i_0, j_0\in \Z_{>0}$, 
\[
M(b^n_{j_0},W^*[0,l],  a^n_{i_0} )\neq \emptyset
\]
and \eqref{identified} implies $\dis \lim_{n\to \infty}\mathcal{E}(b^n_{j_0}, a^n_{i_0}) \to 0$.  Put $a_n:= a_{i_0}^n$ and $  b_n:= b^n_{j_0}$. By the choices of $\{a_n\}_{ n \in \Z_{>0}}$ and $\{b_n\}_{n\in \Z_{>0}}$, we conclude that $$\dis \lim_{n\to \infty} \cs_{Y, \pi_n} (a_n) =\lim_{n\to \infty} \cs_{Y, \pi_n} (b_n) = \Gamma_{-Y}(k).$$
 \end{enumerate} 
\qed \end{proof}

The following theorem is a key theorem to prove \cref{emb1}. In the proof of the next theorem, we use the finiteness of $l^s_Y$ or $l^k_Y$.
\begin{thm}  \label{Convergence1}
\begin{enumerate}
\item Suppose that $l^s_Y< \infty$, $r_s(Y)<\infty$ for some $s\in [-\infty, 0]$. Then, there exists a positive integer $l\leq l^s_Y$ and a flat connection $A_\infty$ on $W[0,l-1]$ such that $A_\infty$ is an irreducible flat connection on $ W[0,l-1] $ with equal boundary flat connections and $\cs_{X, [Y]}^l ( A_\infty) = r_s(Y)$ when we regard $A_\infty$ as a connection on $X_{l, [Y]}$.
\item  Suppose that $l^k_Y< \infty$, $\Gamma_Y(k) <\infty$ for some $k \in \Z_{>0}$. 
Then, there exists a positive integer $l\leq l_Y$ and a flat connection $A_\infty$ on $W[0,l-1]$ such that $A_\infty$ is an irreducible flat connection on $ W[0,l-1] $ with equal boundary flat connections and $\cs_{X, [Y]}^l ( A_\infty) = \Gamma_Y(k) $ when we regard $A_\infty$ as a connection on $X_{l, [Y]}$.
 \end{enumerate}
\end{thm}
\begin{proof}
The proof is comprised of two parts. 
\begin{enumerate}
\item 
Suppose $\{ \pi_n\}_{n\in \Z_{>0}}$ and $g_Y$ are a sequence of non-degenerate regular holonomy perturbations of $\cs_Y$ and a Riemann metric on $Y$ such that 
\begin{align}\label{choice of p}
l^s_Y =\# \sup_{n\in \Z_{>0}} \set{[a] \in \wt{R}^*_{\pi_n} (Y), \ind ([a])=1, r_s(Y)- \lambda_Y<\cs_{\pi_n}([a]) <r_s(Y)+ \lambda_Y}
\end{align}
and let $\{\pi^n_{W^*[0,l^s_Y-1]}\}_{n\in \Z_{>0}}$ be a sequence of regular perturbations of the ASD-equations on $W^*[0,l^s_Y-1]$ such that the perturbed equation of $\pi^n_{W^*[0,l^s_Y-1]}$ coincides with $F^+(A)+ \pi_n^+(A)=0$ on $Y\times (-\infty, 0]$, $Y\times [0,\infty)$ and $Y^+_j \times I$ for $j\in \{0,\cdots, l^s_Y\}$ for any $n$. Here, we apply 1 in \cref{seq}. Then there exist sequences of critical points $\{a_n\}_{n\in \Z_{>0}}$ and $\{b_n\}_{n \in \Z_{>0}} $ of $\cs_{\pi_n}$ such that 
\[
M(a_n, W^*[0,l_Y^s-1] , b_n) \neq \emptyset \text{, }\displaystyle \lim_{n\to \infty}\cs_{\pi_n}(a_n)  =r_s(Y),
\]
and $\cs_{\pi_n} (a_n) - \cs_{\pi_n} (b_n) \to 0$ as $n\to \infty$.  Take an element $A_n \in M(a_n, W^*[0,l^s_Y-1] , b_n)$ for every $n \in \Z_{>0}$. The condition $\cs_{\pi_n} (a_n) - \cs_{\pi_n} (b_n) \to 0$, $\|\pi_n\|\to 0$, and $\|\pi^n_{W^*[0,l^s_Y-1]}\| \to 0$ as $n\to \infty$ imply 
\begin{align}\label{flat conv}
\|F(A_n) + \pi^n_{W^*[0,l^s_Y-1]} (A)\|_{L^2 (W^*[0, l^s_Y-1])} \to 0.
\end{align}
This implies the connection $A_n$ is close to some critical point of $\cs_{Y, \pi_n}$ near $Y^+_j \times I$. On the other hand, we have gauge transformations $\{g^j_n\}$ on $Y^+_j \times I$ and critical points $\{c^j_n\}$ of $\cs_{Y, \pi_n}$ such that 
\begin{align}\label{critical near}
\| (g^j_n)^* A_n|_{Y^+_j \times I} - p^*c^j_n \|_{C^l} \leq c ( \| F(A_n|_{Y^+_j \times I} ) \|_{L^2(Y^+_j \times I)} + \| \pi_n\| ) ,
\end{align}
where the constant $c$ depends only on the metric $g_Y$ and $l$, which we take to be a positive integer greater $2$, and $p: Y^+_j\times I \to Y$ is the projection. Since $\cs_{Y, \pi_n}$ is non-degenerate, $c^j_n$  in \eqref{critical near} is unique up to gauge transformations. Note that if $n$ is large, then 
\[
c^j_n \in \set{[a] \in \wt{R}^*_{\pi_n} (Y), \ind ([a])=1, r_s(Y)- \lambda_Y<\cs_{\pi_n}([a]) <r_s(Y)+ \lambda_Y}. 
\]
By \eqref{choice of p}, we can use the Pigeonhole principle. Then for each $n$, there exist $j_n$ and $j'_n$ such that $c^{j_n}_n \cong c^{j'_n}_n$ as connections on $Y$. Moreover, such patterns of choices of  $j_n$ and $j'_n$ are finite. By changing gauge transformations $g^j_n$, we assume $c^{j_n}_n = c^{j'_n}_n$ 
Thus, after taking a subsequence of $\{A_n\}_{n \in \Z_{>0}}$, we can assume that $j_n$ and $j'_n$ do not depend on the choices of $n$. By summarizing the above discussion, we obtained integers $0\leq j$ and $j' \leq l^s_Y$, gauge transformations $g^j_n$ and $g^{j'}_n$ on $Y^+_j \times I$ and $Y^+_{j'} \times I$, and critical points $\{c_n\}$ of $\cs_{Y, \pi_n}$ such that 
\begin{align}\label{critical near1}
\| (g^j_n)^* A_n|_{Y^+_j \times I} - p^*c_n \|_{C^l} \leq c ( \| F(A_n|_{Y^+_j \times I} ) \|_{L^2(Y^+_j \times I)} + \| \pi_n\| ) 
\end{align}
\begin{align}\label{critical near2}
\| (g^{j'}_n)^* A_n|_{Y^+_{j'} \times I} - p^*c_n \|_{C^l} \leq c ( \| F(A_n|_{Y^+_{j'} \times I} ) \|_{L^2(Y^+_j \times I)} + \| \pi_n\| ) 
\end{align}
hold. After taking a subsequence of $\{c_n\}$ and pull-back by gauge transformations $\{h_n\}$, we can assume that $\{h_n^*c_n\}$ is $C^\infty$-convergent to $c_\infty$. Then $\{h_n^*(g^j_n)^* A_n|_{Y^+_j \times I}\}$ and $\{h_n^*(g^{j'}_n)^* A_n|_{Y^+_{j'} \times I}\}$ are convergent sequences on $Y^+_j \times I$ and $Y^+_{j'} \times I$. We can extend gauge transformations $g^j_n\circ h_n$ and $g^{j'}_n \circ h_n$ to gauge transformations $\{g_n\}$ on $W^* [0, l^s_Y-1]$ such that $\{g_n^* A_n \}$ converges to $\{g_n^* A_{n}\}$ on all of $Y \times \R$. We set $\dis \lim_{n\to \infty} g_n^* A_{n} =A_\infty$. The limit $A_\infty$ has the following properties: there are $j, j' \leq l_Y^s$ such that $A_\infty$ is an $SU(2)$-irreducible flat connection on $ W^*[0,l^s_Y-1] $ with $A_\infty|_{Y^+_j} \cong A_\infty|_{Y^+_{j'}}$ and 
\begin{align}\label{cal of cs}
\cs_Y (A|_{Y^+_j}) = r_s(Y)
\end{align}
for any $j$.  Suppose $j<j'$. Then $A_\infty|_{W[j, j']}$ satisfies the conclusion. 
By \eqref{cal of cs}, we conclude that 
\[
r= \lim_{n\to \infty} \cs_Y(a_n) =  \cs_{\overline{W[j, j']}, [Y^+_j]} ( A_\infty) = \cs_{X,  [Y] }^{j'-j+1}(A_\infty),
\]
where $\overline{W[j, j']}$ is the oriented closed $4$-manifold obtained by identiying the boundaries $Y^+_j$ and $Y^-_{j'}$ of $W[j,j']$. Here, since $A_\infty$ is a connection on $W[j, j']$ such that $A_\infty|_{Y^+_j} \cong A_\infty|_{Y^-_{j'}}$, we regard $A_\infty$ as a connection on $\overline{W[j, j']}$. Moreover, $\overline{W[j, j']}$ can be identified with $X_{j'-j+1, [Y]}$. This completes the proof. 
 
\item The second proof is essentially the same as the first. 
Suppose $\{ \pi_n\}_{n\in \Z_{>0}}$ and $g_Y$ are a sequence of non-degenerate regular holonomy perturbations of $\cs_Y$ and Riemann metric on $Y$ such that 
\[
l^k_Y =\# \sup_{n\in \Z_{>0}} \left\{[a] \in {R}^*_\pi (Y), \ind ([a])\equiv  \begin{cases} 1 \text{ if $k$ is odd}\\ 
5 \text{ if $k$ is even}\end{cases}{\mod 8} , |\Gamma_{-Y}(k)-\cs_{\pi_n}([a]) |< \lambda_Y \right\}
\]
and $\{\pi^n_{W^*[0,l^k_Y-1]}\}_{n\in \Z_{>0}}$ be a sequence of regular perturbations of ASD-equations on $W^*[0,l^k_Y-1]$ such that the perturbed equation of $\pi^n_{W^*[0,l^k_Y-1]}$ coincide with $F^+(A)+ \pi_n^+(A)=0$ on $Y\times (-\infty, 0]$, $Y\times [0,\infty)$ and $Y^+_j \times I$ for $j\in \{0,\cdots, l^k_Y\}$ for any $n$. Then we can do the same discussion as in the first proof. 
\end{enumerate}

\qed \end{proof}

\subsection{Proof of main theorems}\label{Proof of main theorems}
In this section, we prove several theorems in \cref{Introduction}.

\subsubsection{Proof of \cref{emb1}}
In this section, we will give a proof of \cref{emb1}.
\begin{thm}[\cref{emb1}]Let $Y$ be an oriented homology $3$-sphere and $X$ be an oriented negative definite $4$-manifold. 
Suppose there exists an embedding from $Y$ to $X$ with $0\neq [Y] \in H_3(X;\Z)$. 
\begin{enumerate}
\item  If $l_Y^s < \infty$ and $r_s(Y)<\infty$ for some $s \in [-\infty, 0]$, 
\[
r_s(Y) -  \lfloor r_s(Y) \rfloor  \in  \bigcup_{1\leq j \leq l^s_Y} \im\cs_{j_{[Y]}X, [Y]} .
\]
\item If $l_Y^k < \infty$ and $\Gamma_{-Y}(k)<\infty$ for some $k \in \Z_{>0}$, 
\[
\Gamma_{-Y} (k) -  \lfloor \Gamma_{-Y}(k) \rfloor \in  \bigcup_{1\leq j \leq l^k_Y} \im\cs_{j_{[Y]}X, [Y]} .
 \]
\end{enumerate}
If $[Y]=0$, then 
\[
r_s(Y)=r_s(-Y) = \Gamma_Y (k) = \Gamma_{-Y}(k) =  \infty
\]
for any $s$ and $k$. 
\end{thm}
The proof of \cref{emb1} is decomposed into two cases: $[Y] \neq 0$ and $[Y] = 0$.
The first case gives an extension of the main theorem in \cite{Ma18}. However, the proof is not the same. In \cite{Ma18}, the author used  ASD moduli spaces for $4$-manifolds with periodic ends to prove the main result. However, in this paper, we only use ASD moduli spaces for $4$-manifolds with cylindrical ends. The essential part of the proof is contained in the proof of \cref{seq}.

{\em Proof of \cref{emb1} }
First, we assume $[Y] \neq 0$.
Suppose that $Y$ is a homology $3$-sphere $Y$ embedded in a negative definite $4$-manifold $X$ with $0\neq [Y] \in H_3(X; \Z)$. We assume that $X$ is connected. Put $c:= PD([Y]) \in H^1(X; \Z) = \hom (H_1(X;\Z), \Z)$. 
By the following discussion, we may assume that $H_1(X;\R)\cong \R$ to prove \cref{emb1}. Suppose $\rank H_1(X;\R) \geq 2$.
We have an exact sequence: 
\[
0 \to \bigoplus_{\rank{H_1(X;\Z)}} \Z \oplus \text{torsion}\cong \ker c \to H_1(X;\Z) \xrightarrow{c} \Z .
\]
We take a generator  $d_1\in H_1(X;\Z)$ of the free part of $\ker c$. Then $d_1 \cdot [Y]=0$. 
We also have the following exact sequence:
\[
H_1 (X \setminus Y ; \Z ) \to H_1(X ; \Z )  \to H_1( Y\times I ;\partial (Y\times I ); \Z  ) \cong H^3(Y; \Z)\cong \Z .
\]
The map $H_1(X; \Z  )  \to H_1( Y\times I , \partial (Y\times I ); \Z  ) \cong H^3(Y; \Z)\cong \Z $ corresponds to the pairing of $c$. As  $d_1 \in \ker c$, the class $d_1$ is represented by a closed oriented manifold $l'_1 \subset X \setminus Y$. If $[Y] \neq 0$, $X \setminus Y$ is connected. Therefore, by considering the connected sum, we can assume $l_1'$ is connected. We extend $l_1'$ to a framed loop $l_1\colon S^1 \times D^3 \to X $ such that $\im l \cap Y =\emptyset$.  Then by considering surgery along $l_1$, we obtain an oriented connected $4$-manifold $X_{l_1}$. It is shown in \cite{M95} that the $4$-manifold $X_{l_1}$ obtained by surgery along $l_1$ has the same intersection form as $X$ so $X_{l_1}$ is also negative definite. By \cref{surgery}, there exists a class $0\neq c_1 \in H^1(X_{l_1};\Z)$ such that 
\[
\im \cs^j_{X_{l_1},PD(c_1)} \subset \im \cs^j_{X, [Y]}.
\]
 Note that $b_1(X) = b_1( X_{l_1})  -1$. By induction, there exists a sequence of negative definite 4-manifolds $\{ X_{l_j}\}$ and classes $\{c_j\}$ for $1\leq j \leq \rank \ker c$ such that 
 \[
 \im\cs_{X_{l_N},PD(c_N)}^j \subset \cdots \subset \im \cs^j_{X_{l_1},PD(c_1)} \subset \im\cs_{X,[Y]}^j
 \]
 for any $j$, where $l_j$ are embeddings $S^1 \times D^3 \to X_{l_{j-1}}$ satisfying $[l_{S^1\times \{0\} } ]= c_j$. This implies that we can suppose that $H_1(X_{l_N} ; \R)= \R$.  Here we put $W_0:= \overline{X \setminus Y}$ and suppose $H_*(W_0; \R) \cong H_*(S^3; \R)$. 
 Suppose that $r_s(Y)<\infty$ and $l^s_Y < \infty$ for $s \in [-\infty, 0]$ (resp. $\Gamma_Y(k)<1$ and $l^k_Y < \infty$).
\cref{Convergence1} implies that there exists an integer $j \leq l^s_Y$ (resp. $j \leq l^k_Y$) and a representation $\rho \colon \pi_1(W[0,j-1]) \to SU(2)$ satisfying the following conditions: 
\begin{itemize}
\item The restrictions of $\rho$ to the components of $\partial (W[0,j-1]) = Y_0^+ \amalg Y^-_{j-1}$ coincide via the identification $Y_0^+ \to Y^-_{j-1}$. 
\item $\cs^j_{ X, c} (\rho) = r_s(Y)$.
\end{itemize}
This implies that $r_s(Y) \in \im\cs_{X,c}^j$(resp. $\Gamma_Y(k) \in \im\cs_{X,c}^j$).

If $[Y]=0$, since $X$ is connected, $X$ is decomposed into two parts: $X_1 \cup_Y X_2=X$. Since $X$ is negative definite, $X_1$ and $X_2$ are negative definite $4$-manifolds. Therefore, both $ Y$ and $-Y$ bound a negative definite $4$-manifold. By using \eqref{def ineq} and \eqref{def eq1}, we have $\infty = \Gamma_{-Y}(k)=\Gamma_{Y}(k)= r_s(Y) = r_s(-Y)$. \qed

 In \cite{Ma18}, the author constructed an obstruction class to the existence of embeddings by developing gauge theory for 4-manifolds with periodic ends. The main theorem of \cite{Ma18} is as follows.
\begin{thm} [\cite{Ma18}] \label{Ma18} Let $Y$ be an oriented homology $3$-sphere and $X$ an oriented homology $S^1\times S^3$. Suppose that the Chern-Simons functional of $Y$ is non-degenerate and there exists an embedding from $Y$ to $X$ such that $[Y]$ generates $H_3(X;\Z)$. Then 
\[
0=\theta^{[-\infty,r]}_Y  \in I^1_{[-\infty, r]} (Y) 
\]
for $\displaystyle 0\leq r \leq \min_{1\leq j \leq 2\#R(Y) +3}  \min \{ r \in \im \cs^j_{X,[Y]}\} $,
where $R(Y)$ is the quotient set of $\hom (\pi_1(Y), SU(2))$ by the conjugation of $SU(2)$.
\end{thm}
Now, we prove \cref{Ma18} without using any gauge theory on 4-manifolds with periodic ends.
 \begin{thm}\cref{emb1} implies \cref{Ma18}. 
 \end{thm}
 \begin{proof}
The inequality 
\begin{align} \label{imp ineq}
r_0(Y) \leq r_{-\infty}(Y) 
\end{align}
was proved in \cite{NST19}. 
Moreover, when the Chern-Simons functional of $Y$ is non-degenerate, by a similar discussion given in \cite[Theorem 8.4]{SaWe08}, we can take a sequence of non-degenerate regular perturbations $\{\pi_n\} $ such that 
\begin{itemize}
\item $ \|\pi_n\| \to 0$,
\item $\pi_n$ is zero near a neighborhood of $R(Y)$, and 
\item $R_{\pi_n}(Y)=R(Y)$ for any $n$.
\end{itemize} 
This proves 
\begin{align}\label{imp ineq1}  
l_Y \leq \#R(Y). 
\end{align} 
 The inequalities \eqref{imp ineq} and \eqref{imp ineq1} imply that \cref{emb1} recovers \cref{Ma18}. \qed
 \end{proof}

\cref{2-knot2} is a corollary of \cref{emb1}. 
\begin{thm}[\cref{2-knot2}]
Let $Y$ be an oriented homology $3$-sphere and $K$ an oriented $2$-knot. The invariants $\{cs_{K, j}\}_{j \in \Z_{>0}} $ satisfy the following properties: 
  \begin{itemize}
\item 
   Suppose that $l^s_{Y}$ is finite and $r_s(Y)<\infty$ for some $s \in [-\infty, 0]$.
If $Y$ is a Seifert hypersurface of $K$, then 
\[
 r_s(Y) -  \lfloor r_s(Y) \rfloor  \in \bigcup_{1 \leq j \leq l^s_{Y}} \im\cs_{K,j} 
\]
holds.
\item Suppose that $l^k_{Y}$ is finite and $\Gamma_{-Y}(k)<\infty$ for some $k \in \Z_{>0}$.
If $Y$ is a Seifert hypersurface of $K$, then 
\[
 \Gamma_{-Y} (k)-  \lfloor \Gamma_{-Y}(k) \rfloor    \in \bigcup_{1 \leq j \leq l^k_{Y}} \im\cs_{K,j} 
\]
holds.
\end{itemize}
\end{thm}

{\em Proof of \cref{2-knot2} }
The Seifert hypersurface $Y$ determines a codimension $1$-submanifold of $X(K)$ with $1_K =  [Y] \in H_3(X(K);\Z) $. Since $l^s_Y<\infty$, we can apply \cref{emb1}. Then, we obtain 
\[
  r_s(Y) \in \bigcup_{1 \leq j \leq l^s_Y} \im\cs_{ X(K),1_K}^j = \bigcup_{1 \leq j \leq l^s_Y} \im\cs_{K,j}.
\]
This completes the proof. The proof of the second statement is the same. 
\qed 

\subsubsection{Proof of \cref{rep1} }
In this section, we prove \cref{rep1}. 
\begin{thm}[\cref{rep1}]
 Let $Y$ be an oriented homology $3$-sphere and $K$ an oriented $2$-knot. 
 \begin{enumerate}
 \item  If $r_s(Y)<\infty$, $l^s_Y<\infty$  for some $s \in [-\infty, 0]$ and $Y$ is a Seifert hypersurface of $K$, then there exists a positive integer $l \leq l^s_Y$ such that there exists an irreducible representation 
 \[
 \rho\colon G_l(K)  \to SU(2).
 \]
 \item  If $\Gamma_{-Y}(k)<\infty$, $l^k_Y<\infty$  for some $k \in \Z_{>0}$ and $Y$ is a Seifert hypersurface of $K$, then there exists a positive integer $l$($ \leq l^k_Y$) such that there exists an irreducible representation 
 \[
 \rho\colon G_l(K)  \to SU(2).
 \]
 \end{enumerate}
\end{thm}

{\em Proof of \cref{rep1} } 
The proofs of (1) and (2) are the same. We show (1) in \cref{rep1}. 
\cref{2-knot2} implies 
\begin{align}\label{key}
r_s(Y)  \in  \im\cs_{K,j} \cap (0,1) 
\end{align}
 for given $s \in [-\infty, 0]$ and some $j  (\leq l_{Y})$. 
We combine \eqref{key} and \cref{existence of irred} and obtain the conclusion. 

\qed

\section{Extendability of $SU(2)$-representations}\label{Extendability}
 \cref{emb1} shows existence of finite irreducible $SU(2)$-representations. In this section, we provide a method to prove the existence of infinitely many irreducible $SU(2)$-representations.
\begin{defn}Let $(X, Y)$ be a pair consisting of a closed $4$-manifold $X$ and a codimension-$1$ smooth submanifold $Y$ of $X$ and let $\rho : \pi_1(Y)  \to SU(2)$ be a representation. If $\rho$ is extended to a representation $\wt{\rho} :\pi_1(X) \to SU(2)$, then we call $\rho$ an {\it extendable representation} for $(X,Y)$. 
\end{defn}

By using our method, one can give a partial answer to the following question: 
\begin{ques} For a given pair $(X,Y)$, which $SU(2)$-representations of $\pi_1(Y)$ are extendable for $(X,Y)$?
\end{ques}
The easiest examples are $(X=Y\times S^1, Y)$. In this case, all representations are extendable. 
For example, \cref{emb3} proves that if $ X$ is negative definite and $Y= -\Sigma(2,3,5)$, then all $SU(2)$-representations of $\pi_1(Y)$ are extendable. 
The goal of this section is to prove the following theorem: 
\begin{thm}\label{extendable} Let $Y$ be an oriented homology $3$-sphere and let $\rho : \pi_1(Y)  \to SU(2)$ be an $SU(2)$-representation of $\pi_1(Y)$. 
\begin{itemize}
\item
Suppose $\cs_Y ( \rho) = r_s(Y)< \infty$ and $l^s_Y<\infty$ for some $s\in [-\infty, 0]$, the Chern-Simons functional of $Y$ is Morse-Bott at the level $r_s(Y)$, and for each component $\bigcup_\alpha C_\alpha = \cs_Y^{-1}(r_s(Y)) \cap R^*(Y)$, we assume there exist Morse functions $g_\alpha : C_\alpha \to \R$ such that 
\[
 \set { p \in C_\alpha | \ind_{g_\alpha} (p) + \ind (C_\alpha )=1, \ d(g_\alpha)_p= 0 }= \emptyset
\]
if $\rho \notin C_\alpha$ and 
 \[
1=\# \set { p \in C_{\alpha_0} | \ind_{g_{\alpha_0}} (p) + \ind (C_{\alpha_0} )=1, \ d(g_{\alpha_0})_p= 0 }, 
\]
where $C_{ \alpha_0}$ is the component which contains $\rho$ and $\ind (C_\alpha )$ is the Morse-Bott index of $C_\alpha$. Let $X$ be a closed negative definite $4$-manifold containing $Y$ as a submanifold.

Then $\rho$ is extendable for the pair $(X,Y)$. Moreover, all $SU(2)$-representations which lie the same component of $\rho$ are extendable for $(X,Y)$.
\item  
Suppose $\cs_Y ( \rho) = \Gamma_{-Y}(k)< \infty$ and $l^k_Y<\infty$ for some $k\in \Z_{>0}$, the Chern-Simons functional of $Y$ is Morse-Bott at the level $\Gamma_{-Y}(k)$, and for each component $\bigcup_\alpha C_\alpha = \cs_Y^{-1}(\Gamma_{-Y}(k)) \cap R^*(Y)$, we assume there exist Morse functions $g_\alpha : C_\alpha \to \R$ such that 
\[
 \left\{ p \in C_\alpha | \ind_{g_\alpha} (p) + \ind (C_\alpha )=\begin{cases}1 \text{ if $k$ is odd}\\ 5 \text{ if $k$ is even} \end{cases} (\mod 8)  , \ d(g_\alpha)_p= 0 \right\}= \emptyset
\]
if $\rho \notin C_\alpha$ and 
 \[
1=\#  \left\{ p \in C_{\alpha_0} | \ind_{g_{\alpha_0}} (p) + \ind (C_{\alpha_0} )=\begin{cases}1 \text{ if $k$ is odd}\\ 5 \text{ if $k$ is even} \end{cases} (\mod 8), \ d(g_{\alpha_0})_p= 0 \right\}, 
\]
where $C_{ \alpha_0}$ is the component which contains $\rho$ and $\ind (C_\alpha )$ is the Morse-Bott index of $C_\alpha$. Let $X$ be a closed negative definite $4$-manifold containing $Y$ as a submanifold.

Then $\rho$ is extendable for the pair $(X,Y)$. Moreover, all $SU(2)$-representations which lie the same component of $\rho$ are extendable for $(X,Y)$.
\end{itemize}
\end{thm}
\begin{proof} Put $r= r_s(Y) = \cs_Y (\rho) $. Since $r_s(Y) < \infty$, we can suppose $0\neq [Y]\in H_3(X; \Z )$ by \cref{emb1}. 
First, we can assume the class $[Y]$ generates $H_3(X; \R) \cong \R$ by the same argument as in the proof of \cref{seq}. 
Put $g':= \coprod g_\alpha : \coprod C_\alpha = \cs_Y^{-1} (r) \cap R^*(Y) \to \R$.
We fix another Morse function $g:\cs_Y^{-1} (r) \cap R^*(Y) \to \R$ obtained by deforming the Morse function $g'$ satisfying the following conditions: 
\begin{itemize}
\item the point $\rho \in C_\alpha$ is a critical point of $g$ whose Morse index is $1 -\ind (C_{\alpha_0}) $, 
\item  $\set { p \in C_\alpha | \ind_{g_{C_\alpha}} (p) + \ind (C_\alpha )=1, \ d(g_{C_\alpha})_p= 0 }= \emptyset$ if $\alpha \neq \alpha_0$ and 
\item $1=\# \set { p \in C_{\alpha_0} | \ind_{g_{C_{\alpha_0}}} (p) + \ind (g_{C_{\alpha_0}} )=1, \ d(g_{C_{\alpha_0}})_p= 0 }$. 
\end{itemize}
By using \cref{MB per}, we take a Morse-Bott type perturbation $\pi_\epsilon(g)$ for $\epsilon \in (0,\epsilon_0)$. We take a sequence of non-generate regular perturbations $\{\pi_n\}$ such that $\cs_{\pi_{1/n}(g)} |_{U} = \cs_{\pi_n}|_U$ and $\|\pi_n\| \to 0$, where $U$ is a neighborhood of $\cs_Y^{-1} (r) \cap R^*(Y)$ in $B^*(Y)$. By \cref{MB per}, we can see $\{ \pi_n\}$ satisfies the following conditions: 
\begin{itemize}
\item there is a correspondence
\[
 L_n:  \set { \text{the critical point set of } g: \cs_Y^{-1} (r) \cap R^*(Y) \to \R } \to \set{ \text{critical point set of }\cs_{\pi_n}|_U} 
 \]
  for each $n$ such that $\dis \lim_{n\to \infty}  L_n(a) = a$, 
\item $\# \set {a \in \wt{R}^*_{\pi_n}(Y) | \ind (a) =1, \ |\cs_{\pi_n}(a)-r| < \frac{1}{2} \lambda_Y  }= 1$, and 
\item  $\dis \{\rho \} =\set { p: \text{the critical point set of } g | \lim_{n\to \infty} L_n(p) \in C_{\alpha_0} }$. 
\end{itemize}
Thus, we conclude that $l_Y^s= 1$.
 Then we apply \cref{Convergence1} and obtain an irreducible flat connection $A_\infty$ on $W[0,0]$ such that  
  \[
cs_{X, [Y]}^1 ( A_\infty) = r_s(Y). 
\]
Moreover, by the proof of \cref{Convergence1}, we can see that 
\[
A_\infty|_{Y^+_0} \cong A_\infty|_{Y^-_0} \cong \rho.
\]
Thus, $A_\infty$ gives an extension of $\rho$. Moreover, for another point $\rho'$ which lies in the same component of $\rho$,  we can take another Morse function $g$ with the following conditions: 
\begin{itemize}
\item the point $\rho' \in C_\alpha$ is a critical point of $g$ whose Morse index is $1 -\ind (C_{\alpha_0}) $, 
\item  $\set { p \in C_\alpha | \ind_{g_{C_\alpha}} (p) + \ind (C_\alpha )=1, \ d(g_{C_\alpha})_p= 0 }= \emptyset$ if $\alpha \neq \alpha_0$ and 
\item $1=\# \set { p \in C_{\alpha_0} | \ind_{g_{C_{\alpha_0}}} (p) + \ind (g_{C_{\alpha_0}} )=1, \ d(g_{C_{\alpha_0}})_p= 0 }$
\end{itemize}
by modifying $g'$.
Then, by the same discussion, we have an extension $A_\infty$ of $\rho'$.
\qed
\end{proof}
We give a sufficient condition for the assumptions of \cref{extendable} to hold. 
\begin{cor}\label{perfect}
Let $Y$ be an oriented homology $3$-sphere and let $\rho : \pi_1(Y)  \to SU(2)$ be an $SU(2)$-representation of $\pi_1(Y)$. 
\begin{enumerate}
\item Let $l$ be a positive integer.
Suppose $\cs_Y ( \rho) = r_s(Y)< \infty$ and $l^s_Y=1$ for some $s\in [-\infty, 0]$, the Chern-Simons functional of $Y$ is Morse-Bott at the level $r_s(Y)$, $H_* ( \cs_Y^{-1}(r_s(Y)) \cap R^*(Y); \Z) \leq 1$ for any $*$ and $\cs_Y^{-1}(r_s(Y)) \cap R^*(Y)$ has a perfect Morse function.  Let $X$ be a closed negative definite $4$-manifold containing $Y$ as a submanifold.
Then $\rho$ is extendable for the pair $(X,Y)$. Moreover, all $SU(2)$-representations which lie the same component of $\rho$ are extendable for $(X,Y)$.
\item  Let $l$ be a positive integer.
Suppose $\cs_Y ( \rho) = \Gamma_{-Y}(k)< \infty$ and $l^k_Y=1$ for some $k\in \Z_{>0}$, the Chern-Simons functional of $Y$ is Morse-Bott at the level $\Gamma_{-Y}(k)$, $H_* ( \cs_Y^{-1}(\Gamma_{-Y}(k)) \cap R^*(Y); \Z) \leq 1$ for any $*$ and $\cs_Y^{-1}(r_s(Y)) \cap R^*(Y)$ has a perfect Morse function.Let $X$ be a closed negative definite $4$-manifold containing $Y$ as a submanifold. Then $\rho$ is extendable for the pair $(X,Y)$. Moreover, all $SU(2)$-representations which lie the same component of $\rho$ are extendable for $(X,Y)$.
\end{enumerate}
\end{cor}
\begin{proof}We check that the assumptions of \cref{extendable} are satisfied. By assumption, we have a Morse function $g: \cs_Y^{-1}(r_s(Y)) \cap R^*(Y) \to \R$ such that the differentials of the Morse complex with respect to $g$ are zero. 
Since $H_*( \cs_Y^{-1}(r_s(Y)) \cap R^*(Y); \Z) \leq 1$, 
\[
\#\{ p \in \cs_Y^{-1}(r_s(Y)) \cap R^*(Y) | \ind (p)=*, \ dg_p =0 \}  \leq 1 \ \  \forall * \in \Z_{\geq 0} . 
\]
Note that $\cs_Y^{-1}(r_s(Y)) \cap R^*(Y)$ is connected.  
Then we use a Morse-Bott type perturbation $\pi_\epsilon(g)$ for $\epsilon \in (0,\epsilon_0)$ given in \cref{MB per} to define filtered instanton chain complexes $CI^{[s, r]}_* (Y)$ given in \cref{ss}. 
Since the value of $r_s(Y)$ depends only on the index-1-critical points of $\cs_{Y, \pi_\epsilon(g)} $ in $\cs_Y^{-1}(r_s(Y)) \cap R^*(Y)$, we have
\[
 \#\{ p \in \cs_Y^{-1}(r_s(Y)) \cap R^*(Y) | \ind_{g} (p) +  \ind_{\cs_Y} (\cs_Y^{-1}(r_s(Y)) \cap R^*(Y)) =1 , p\in C_\al  , \ dg_p =0 \}  \geq 1
\]
We suppose that $g$ has $\rho$ as a critical point satisfying $\ind_{g} (p) +  \ind_{\cs_Y} (C_\alpha) =1$. 
Thus, the assumptions of 1 of \cref{extendable} are satisfied. A similar argument proves 2. 

 \qed
\end{proof}

\section{Examples and Applications}\label{Examples}
In this section, we give several examples of computations of $l_Y$, $l_Y^s$ and $l_Y^k$ and applications of main theorems in \cref{Introduction}. 
\subsection{Calculations of $l_Y$, $l_Y^s$ and $l_Y^k$}
In this section, we give several computations of $l_Y$, $l_Y^s$ and $l_Y^k$. 
\subsubsection{Seifert homology 3-spheres} 
In this section, we discuss the invariants $l^s_Y$, $l^k_Y$ and $l_Y$ for Seifert homology $3$-spheres of type $\Sigma(a_1, \cdots, a_n)$. In \cite{FS90}, the $R$-invariant of $\Sigma(a_1, \cdots, a_n)$: 
\begin{align}\label{R}
R(a_1, \cdots, a_n) = \frac{2}{a} -3 +n +\sum_{i=1}^n \frac{2}{a_i}\sum_{k=1}^{a_i-1} \cot \left( \frac{a \pi k}{a_i} \right) \cot \left(\frac{\pi k}{a_i}\right) \sin^2 \left(\frac{\pi k}{a_i}\right)
\end{align}
is introduced, where $a= a_1 \cdot \cdots \cdot a_n$. In \cite{FS90}, Fintushel and Stern gave an algorithm to calculate $SU(2)$-representation spaces $R(\Sigma(a_1, \cdots, a_n))$. Fix a sequence of integers $b$, $b_1$, $\cdots$ , $b_n \in \Z$ such that $\displaystyle a \sum_{1 \leq k\leq n} \frac{b_k}{a_k} = 1 + ab$. We call the pair $(b, (a_1, b_1), \cdots , (a_n,b_n) )$ a {\it Seifert invariant}. For a such sequence, the fundamental group of $\Sigma(a_1, \cdots, a_n)$ is given by 
\[
\pi_1(\Sigma(a_1, \cdots, a_n)) =  \left\{ x_1, \cdots, x_n, h \middle| [x_i, h]=1, x_i^{a_i} = h^{-b_i} , x_1 \cdots x_n = 1 \right\}.
\]
Note that if $\rho:\pi_1(\Sigma(a_1, \cdots, a_n))\to SU(2)$ is an irreducible representation, then $\rho(h)=\pm 1$. Suppose that $a_1$ is even. We choose $\{b_i\}$ so that $b_j$ is even for $j \neq 1$. The number $l_i$ is even if and only if $\rho(h)^{b_i} = 1$. 
Then connected components of the $SU(2)$-representations $\rho$ of $\pi_1(\Sigma(a_1, \cdots, a_n))$, imposing some technical conditions, are parametrized by the rotation numbers $(l_1, \cdots, l_n)$ given by \[
g_i^{-1} \rho (x_i) g_i = 
\begin{pmatrix}
e^{\pi i l_i/ a_i} & 0 \\
0 & e^{-\pi i l_i/ a_i}  \\
\end{pmatrix} \text{ for some } g_i\in SU(2) \  ( 1\leq i\leq n)
\]
 and $0\leq l_i  \leq a_i$. 
For this parametrization, the value of the Chern-Simons functional is then given by \begin{align}\label{cs_Yof Seifert}
cs(\rho(l_1, \cdots, l_n ) ) =\left(\sum_{i=1}^n a l_i /a_i\right)^2/4a \mod 1.
\end{align}
Moreover, the Floer index\footnote{When the Chern-Simons functional of an oriented homology 3-sphere $Y$ is Morse-Bott, we can also define a Morse-index. (See \cite{FS90}.)} $\ind (\rho_{ (l_1, \cdots , l_n)})$ is given by 
\[
\ind \left(\rho_{ (l_1, \cdots , l_n)} \right)= \frac{2e^2}{a} + 3 -m + \sum_{i=1}^m \frac{2}{a_i}\sum_{k=1}^{a_i-1} \cot \left( \frac{a \pi k}{a_i} \right) \cot \left(\frac{\pi k}{a_i}\right) \sin^2 \left(\frac{\pi e k}{a_i}\right) \mod 8. 
\]

\begin{ex}\label{Brieskorn}Suppose that $n=3$. Since $R(2,3,6k-1)=1$, it is proved in \cite[Theorem 6]{D18} and \cite[Corollary 1.4]{NST19} that
\[
\Gamma_{\Sigma(2,3,6k-1)}(1)= r_s(-\Sigma(2,3,6k-1)) = \frac{1}{24(6k-1)}
\]
 for any $s \in [-\infty, 0]$. Since $\lambda(\Sigma(2,3,6k-1)) = k$, \[
 l_{\Sigma(2,3,6k-1)}=2k
 \]
  by \cref{non-deg}.  Moreover, it is shown in \cite{FS90} that the Chern-Simons functional of $\Sigma(2,3,6k-1)$ is Morse.
Then $(l_1, l_2 , l_3)$ determines a representation if and only if 
\[
|l_1/a_1- l_2/a_2 | < l_3 /a_3 <1- |1 - (l_1/a_1- l_2/a_2)|. 
\]
Here, we choose $a_1=2$, $a_2=3$, $a_3=6k-1$. 
Then, the space of the representations are parametrized by 
\begin{align}\label{236k-1}
R^* (\Sigma(2,3,6k-1) ) =\{ (1, 2, l_3) | l_3 \in   \{ k, k+1, \cdots , 5k-1\} \cap 2 \Z \}.
\end{align}
Then the Chern-Simons functional is given by \[
cs( \rho _{(1, 2, l_3)}) =  \frac{12( 3k^2 - k + 3 l^2_3) +1
 }{24 (6k-1)} \mod 1.
\]
If $k$ is odd, $\rho_{(1,2,5k-1)}$ has the unique minimal value $\frac{1}{6(6k-1)}$ of $\cs$ in $(0,1]$ and if $k$ is even, $\rho_{(1,2,k)}$ has the unique minimal value $\frac{1}{6(6k-1)}$ of $\cs$ in $(0,1]$. Therefore, we can conclude that \[
l^1_{\Sigma(2,3,6k-1)} = l_{\Sigma(2,3,6k-1) }^s = 1
\]
 for any $s \in [-\infty,0]$. In the case of $k=1$, we can also see that $l^2_{\Sigma(2,3,5)}=1$.
\end{ex}
\begin{ex}\label{2357}
Suppose $n=4$, $a_1=2$, $a_2=3$, $a_3=5$ and $a_4=7$. One can check that $R(2,3,5,7) =1$. It is shown that \[
\Gamma_{\Sigma(2,3,5,7)} (1) = r_s(-\Sigma(2,3,5,7)) =  \frac{1}{840}
\]
 for any $s \in [-\infty, 0]$. The space $R^*(\Sigma(2,3,5,7))$ has six $S^2$-components and $16$ points. We put critical values of the Chern-Simons functional for $-\Sigma(2,3,5,7)$ in \cref{figure} which are given in \cite{FS90} and \cite{Sav02}. 
In \cref{figure}, the Floer indices for Morse-Bott components with dimension greater than $0$ are described as bold numbers. 
\cref{figure} implies that $\nu (-\Sigma(2,3,5,7)) = r_s(-\Sigma(2,3,5,7))$ for any $s \in [-\infty, 0]$. 
\begin{table}[htb]\label{2357} 
\begin{center}
  \begin{tabular}{| l|c|c|c|c|c|}\hline
  $(l_1, l_2, l_3, l_4)$ & $\cs$ & $\ind$  & $(l_1, l_2, l_3, l_4)$ & $\cs$ & $\ind$ \\ \hline
  $(1,0,2,2)$ &${681}/{840}$  &7  & $(0, 2, 4, 2) $& ${184}/{840}$  &3 \\ \hline
     $(1,0,2,4)$ &${561}/{840}$  &5  &$(2, 2, 2, 2)$  &${436}/{840}$ &5 \\ \hline
    $(1,0,2,6)$ &  ${81}/{840}$& 1 &$ (2, 2, 2, 4)$ &${316}/{840}$ &3  \\ \hline
    $ (1, 0, 4, 4)$ &${729}/{840}$  &7  & $(2, 2, 4, 4) $&${484}/{840}$ &5 \\ \hline
    $ (1, 2, 0, 2)$ &${625}/{840}$  & 7 & $ (2, 2, 4, 6)$ &${4}/{840}$ &1 \\ \hline 
   $ (1, 2, 0, 4)$  & ${505}/{840}$ &5 &$(1,2,2,2)$  & ${121}/{840}$&{\bf 1} \\ \hline 
   $ (1, 2, 2, 0)$  & ${721}/{840}$&7 &$(1,2,2,4) $ & ${1}/{840}$ &{\bf 7} \\ \hline 
   $ (1, 2, 4, 0)$ &${49}/{840}$ &1 &$(1,2,2,6) $& ${361}/{840}$&{\bf 3} \\ \hline
   $(0, 2, 2, 2)$ & ${16}/{840}$ &1 &$(1,2,4,2) $& ${289}/{840}$&{\bf 3} \\ \hline
 $  (0, 2, 2, 4)$&${736}/{840}$ &7 &$ (1,2,4,4) $ &${169}/{840}$ &  {\bf 1} \\ \hline
$(0, 2, 2, 6)$ &${256}/{840}$ &3 &$(1,2,4,6) $& ${529}/{840}$&{\bf  5} \\ \hline
  \end{tabular}
  \caption{Critical values of $-\Sigma(2,3,5,7)$}\label{figure}
  \end{center}
  The usual instanton group of $-\Sigma(2,3,5,7)$ is isomorphic to $\Z^{28}$ and the Casson invariant is $14$. This implies that $l_{-\Sigma(2,3,5,7)}=28$. On the other hand, for $s \in [-\infty, 0]$, $l^s_{-\Sigma(2,3,5,7)} =1$. 
  Moreover, if we let $Y$ be the connected sum $ \#_{n} (-\Sigma(2,3,5,7))$ for a positive $n$, then the connected sum formula of $r_0(Y)$ implies that $r_0(Y)= \frac{1}{840}$. Thus, we can also see that $l^0_{Y} =1$. 
\end{table}
\end{ex}Next, we consider the finite connected sums of Seifert homology $3$-spheres. 

\begin{ex} \label{cal3} Fix $k \in \Z_{>0}$.
Let $Y$ be the linear combination \[
\displaystyle n\Sigma(2,3,6k-1) \# ( \displaystyle \#_{6(6k-1)<a_1 \cdots a_n } m_{(a_1,\cdots ,a_n)} \Sigma (a_1,\cdots ,a_n)),
\]
where $n<0$, $k \in \Z_{>0}$ and $m_{(a_1,\cdots ,a_n)}$ is a sequence of integer parametrized by $(a_1,\cdots ,a_n)$ with 
\[
a_1 \cdots a_n < 24 (6k-1) .
\]
 Then \cref{main NST} \eqref{conn sum formula} implies 
 \[
\frac{1}{24(6k-1)}= r_0 (n\Sigma(2,3,6k-1)) \geq \min \set { r_0(Y), r_0( \displaystyle \#_{6(6k-1)<a_1 \cdots a_n } m_{(a_1,\cdots ,a_n)} \Sigma (a_1,\cdots ,a_n)) }.
 \]
By the condition $6(6k-1)<a_1 \cdots a_n$  and the formula \eqref{cs_Yof Seifert}, we conclude $\frac{1}{24(6k-1)}= r_0(Y)$. Moreover, we can see $\# \cs^{ -1}_Y \left(\frac{1}{24(6k-1)} \right)=1$. Therefore, \[
l^0_Y=1.
\] Because of the choice of $(a_1, \cdots, a_n)$, we have $\frac{1}{24(6k-1)} =  r_0(Y)$. 
\end{ex}

\subsubsection{A hyperbolic homology $3$-sphere}
In \cite{NST19}, Nozaki, Sato, and the author tried to calculate $r_s(Y)$ for the hyperbolic 3-manifold obtained along $1/2$-surgery along the mirror image of $5_2$.  
\begin{ex}\label{hyperbolic}Let $Y$ be $S^3_{1/2}(5_2^*)$. The Casson invariant of $Y$ is $4$. In \cite{NST19}, it is checked that all critical points of $\cs$ of $Y$ are non-degenerate and there are eight $SU(2)$-representations. Moreover, approximate critical values of $\cs$ are given in Table 2 of \cite{NST19}. We put the minimal value of $\cs$ of $S^3_{1/2}(5_2^*)$ in $(0,1]$ by $t_0$. Then, $r_s(Y)=\Gamma_{-Y}(1) = t_0$ for any $s \in [-\infty, 0]$. Since $\#cs^{-1}(t_0)=1$, we can conclude that \[
l^s_Y=l^1_Y=1
\]
 for any $s \in [-\infty, 0]$. The manifold $S^3_{1/2}(5_2^*)$ also satisfies $\nu (S^3_{1/2}(5_2^*)) = r_s (S^3_{1/2}(5_2^*))$ for any $s \in [-\infty, 0]$.
\end{ex}
\subsection{Applications}
In this section, we give several applications of the main theorems \cref{2-knot1}, \cref{2-knot2}, \cref{rep1}, \cref{emb1} and \cref{emb4}. 
\subsubsection{Ribbon $2$-knots}\label{Ribbon}
The class of ribbon $2$-knots is one of the fundamental classes of $2$-knots. There are several studies of the properties of ribbon $2$-knots (\cite{Ka08}, \cite{Y69}, \cite{Y69II} and \cite{Y70III}). For example, it is shown in \cite{Y69II} that the knot group $G(K)$ is torsion free for any ribbon $2$-knot $K$. In \cite{Ka08}, Kawauchi gave a characterization of Alexander modules of ribbon $2$-knots.  Moreover, every ribbon $2$-knot has a finite connected sum of copies of $S^1\times S^2$ (\cite{Y69}) as a Seifert hypersurface. However, the classification problem for ribbon 2-knots is open. There several invariants which obstruct ribon-nnes of $2$-knots. Ruberman (\cite{R90}) introduced the {\it Gromov norm} $|K| \in [0,\infty)$ of 2-knots $K$ by using the Gromov norm of Seifert hypersurfaces. This invariant satisfies the inequality 
\[
| K | \leq |Y| 
\]
for any Seifert hypersurface $Y$ of $K$, where $|Y|$ is the Gromov norm of $Y$. If $K$ is a ribbon $2$-knot, then $|K|=0$. Moreover, the following gauge theoretic invariants can be useful in obstructing the ribbon property of 2-knots: 
 \begin{enumerate}

 \item Mrowka, Ruberman and Saveliev (\cite{MRS11}) constructed a version of Seiberg-Witten invariant $\lambda_{MRS}(X)$ for homology $S^1\times S^3$'s. 
Then one define $\lambda_{MRS}(K) := \lambda_{MRS}(X(K))$. Moreover,  if $K$ is ribbon, then $\lambda_{MRS}(K)=0$ (\cite{LRS18}).
 \item 
 Levine and Ruberman (\cite{LR19}) defined a $\Z$-valued invariant $\tilde{d}(X,y)$ of homology $S^1 \times S^3$ by considering some variant of $d$-invariant of cross sections. Using $\tilde{d}(X,y)$, Levine and Ruberman defined invariants of 2-knots $\tilde{d}(K)$, $\tilde{d}(-K)$, $\tilde{d}(\overline{K})$ and $\tilde{d}(-\overline{K})$. They show that $\tilde{d}(K)= \tilde{d}(-K)= \tilde{d}(\overline{K})=\tilde{d}(-\overline{K})=0$ if $K$ is ribbon.
 \end{enumerate}
 \begin{rem}Furuta and Ohta (\cite{FO93}) also constructed a Casson type invariant $\lambda_{FO}(X)$ of homology $S^1\times S^3$'s satisfying the condition $H_*(\wt{X}; \q) \cong H_*(S^3;  \q)$. Such homology $S^1 \times S^3$'s are called $\Z[\Z]$ homology $S^1 \times S^3$. 
  It is conjectured that $\lambda_{FO}(X)=- \lambda_{MRS}(X)$ for any $\Z[\Z]$ homology $S^1 \times S^3$ $X$.   In \cite{LRS18}, this conjecture is checked for a certain class of mapping tori.  One can also define $\lambda_{FO}(K)$ for $2$-knots $K$ satisfying the condition $\Delta_K(t)=1$. By construction, if $\lambda_{FO}(K) \neq 0$, then $G(K)$ has an $SU(2)$-irreducible representation. 
  \end{rem}
Our invariants $\{\im\cs_{K,j}\}_{k \in \Z_{>0}}$ are also useful for obstructing the ribbon property of 2-knots.
\begin{cor}\label{2-knot3}Let $K$ be an oriented $2$-knot and $Y$ a Seifert hypersurface of $K$.
\begin{enumerate}
\item 
If $Y$ is an oriented homology $3$-sphere,  $l^s_Y< \infty$, $r_s(Y)<\infty$ and $r_s(Y)\notin \Z$ for some $s\in [-\infty,0  ]$, then $K$ is not ribbon. 
\item If $Y$ is an oriented homology $3$-sphere,  $l^k_Y< \infty$, $\Gamma_{-Y}(k)<\infty$ and $\Gamma_{-Y}(k)\notin \Z$ for some $k \in \Z_{>0}$, then $K$ is not ribbon. 
\end{enumerate}
 \end{cor}
 \begin{proof}We only prove (1) since (2) is similar. Suppose such a Seifert hypersurface $Y$ of $K$ exists. Then \cref{2-knot2} implies that 
 \[
 1 \neq r_s(Y) - \lfloor r_s(Y)\rfloor  \in \im cs_{K, j_0} 
 \]
  for some $j_0\in \Z_{>0}$. If $K$ is a ribbon $2$-knot, then $\im cs_{K, j} = \{1\}$ for any $j$ by \cref{ribbon}. This contradicts to $r_s(Y) - \lfloor r_s(Y)\rfloor  \in \im cs_{K, j_0} \cap (0,1) \neq \emptyset$.  \qed
 \end{proof} 
 \begin{rem}
 Note that if $K$ admits a homology $3$-sphere as a Seifert hypersurface, then $\Delta_K(t)=1$.
 \end{rem}
The following corollary seems difficult to show using $|K|$, $\lambda_{MRS}(K)$, $\lambda_{FO}(K)$ and $\tilde{d}(K)$: 
 \begin{cor}\label{non-ribbon}
 Let $k$ be a positive integer. Any $2$-knot $K$ having $(-\Sigma(2,3,6k+5)) \# \Sigma (2,3, 5) $ as a Seifert hypersurface is not ribbon. 
 \end{cor}
 \begin{proof}By \cref{cal3}, 
 \[
 r_0((-\Sigma(2,3,6k+5)) \# \Sigma (2,3, 5)) <1 \text{ and } l^0_{(-\Sigma(2,3,6k+5)) \# \Sigma (2,3, 5)}=1.
 \]
  Therefore, \cref{2-knot3} tells us that such a knot $K$ cannot be a ribbon. \qed
 \end{proof}

\subsubsection{Rigidity of $\{\im cs_{K,  j} \}_{j \in \Z_{>0}}$}
The invariant  $\tilde{d}(K)$ is determined by the $d$-invariant of a Seifert hypersurface. We will see that similar properties hold for $\im cs_{K,  j} $ of a certain class of $2$-knots. 
We give a sufficient condition on Seifert hypersurfaces to determine $\im cs_{K,  j}$. 

\begin{thm} \label{Image}
Suppose that an oriented homology 3-sphere $Y$ is a Seifert hypersurface of a given oriented 2-knot $K$ and there exist $k_1, \cdots ,k_m \in \Z_{>0}$ and $s_1, \cdots ,s_n \in [-\infty ,0]$ such that
\[
\bigcup_{k_1, \cdots k_m, s_1, \cdots s_n}  (r_s(Y) \cup \Gamma_{-Y}(k) ) = \Lambda_Y \cap (0,1)\]
 and $l^{s_i}_Y= l^{l_i}_Y=1$ for any $i$. Then 
 \[
 \im\cs_{K,j} = \Lambda_Y \cap (0,1)
 \]
 for any $j \in \Z_{>0}$.  In particular, for any Seifert hypersurface $Y'$ and $r \in  \Lambda_Y \cap (0,1)$, there exists an $SU(2)$-representation $\rho_r$ on $Y'$ such that 
 $\cs_{Y'} ( \rho_r ) =r$.
\end{thm}
\begin{proof}By \cref{2-knot1} and \cref{2-knot2}, 
\[
 \Lambda_Y \cap (0,1]= \bigcup_{k_1, \cdots k_m, s_1, \cdots s_n}  (r_s(Y) \cup \Gamma_{-Y}(k) )  \subset \im\cs_{ K,1} \subset  \Lambda_Y \cap (0,1].
\]
This implies that $\im\cs_{ K,1} =  \Lambda_Y \cap (0,1]$. We use \cref{2-knot1} again and obtain 
\[
  \Lambda_Y \cap (0,1] = \im\cs_{ K,1} \subset  \im\cs_{ K,j} \subset  \Lambda_Y \cap (0,1].
  \]
  This completes the proof. \qed
\end{proof}
\begin{cor}\label{-235}For any $2$-knot having $-\Sigma(2,3,5)$ as a Seifert hypersurface, 
 \[
 \im\cs_{K,j} = \left\{ \frac{1}{120}, \frac{49}{120} , 1\right\}
 \]
 for any $j \in \Z_{ >0}$. For such a $2$-knot $K$ and its Seifert hypersurface $Y'$, $Y'$ has at least two $SU(2)$-representations of $\pi_1(Y')$.
\end{cor}
\begin{proof} We check that $-\Sigma(2,3,5)$ satisfy the assumption of \cref{Image}. In \cite{D18}, Daemi proved
\[
\Gamma_{\Sigma(2,3,5)} (1) = \frac{1}{120} \text{ and } \Gamma_{\Sigma(2,3,5)} (2) = \frac{49}{120}.
\]
Note that $\Lambda_{-\Sigma(2,3,5)} \cap (0,1] =  \left\{ \frac{1}{120}, \frac{49}{120} , 1\right\}$ and the two elements in $R^*(\Sigma(2,3,5))$ are non-degenerate. Therefore, $l_{\Sigma(2,3,5)}^1=l_{\Sigma(2,3,5)}^2=1$. \qed
\end{proof}

\subsubsection{Seifert hypersurfaces of $2$-knots}
In this section, we treat the following problem: what are the Seifert hypersurfaces for a given $2$-knot? Solving this problem has two parts. The first part is constructing Seifert hypersurfaces of a given oriented $2$-knot.  For twisted spun $2$-knots, Zeeman (\cite{Z65}) constructed Seifert hypersurfaces. In general, for a given $2$-knot, there are several formulations of diagrams of them containing the motion picture, the ch-diagram and the surface diagram(\cite{K17}). For such diagrams, there are several ways to construct Seifert hypersurfaces (\cite{CS97}, \cite{CS98}). The second part is to obstruct the existence of a certain class of $3$-manifolds as Seifert hypersurfaces. 
\begin{thm}\label{OBS}Let $(p,q)$ be a a relative prime pair. Let $k(p/q)$ be a $2$-bridge knot such that $\Sigma^2(k (p/q) ) = L(p,q)$, where $\Sigma^2(k)$ is the double branched cover of $k \subset S^3$.
\begin{enumerate}
\item 
For an oriented $3$-manifold $Y$ and, if the twisted spun knot $K(k(p/q), 2)$ has $Y$ as a Seifert hypersurface, then
\[
\left\{-\frac{n^2r}{p} \mod 1  \middle| 0 \leq n \leq  \left\lceil  \frac{p}{2} \right\rceil   \right\} \subset \Lambda_Y \cap (0,1],
\]
 where $r$ is any integer satisfying $qr\equiv -1 \mod p$.
\item If $Y$ is an oriented homology $3$-sphere and $r_s(Y)<\infty$, $l^s_Y =1$ for some $s \in [-\infty, 0]$ and 
\[
r_s(Y)- \lfloor r_s(Y) \rfloor   \notin \left\{- \frac{n^2r}{p} \mod 1  \middle| 0 \leq n \leq  \left\lceil  \frac{p}{2} \right\rceil   \right\} , 
\]
 then $K(k(p/q), 2)$ does not have $Y$ as a Seifert hypersurface, where $r$ is any integer satisfying $qr\equiv -1 \mod p$.
\item If $Y$ is an oriented homology $3$-sphere, $ \Gamma_{-Y} (k) <1$, 
\[
\Gamma_{-Y} (k)-  \lfloor \Gamma_{-Y} (k) \rfloor   \notin \left\{- \frac{n^2r}{p} \mod 1  \middle| 0 \leq n \leq  \left\lceil  \frac{p}{2} \right\rceil   \right\} , 
\] and $l^k_Y =1$ for some $k\in \Z_{>0}$, then $K(k(p/q), 2)$ does not have $Y$ as a Seifert hypersurface, where $r$ is any integer satisfying $qr\equiv -1 \mod p$.
\end{enumerate}
\end{thm}
\begin{proof}This is a corollary of \cref{cal2}, \cref{2-knot1} and \cref{2-knot2}. \qed
\end{proof}
As a corollary, we have the following result:
\begin{cor}Let $(p,q)$ be a a relative prime pair. Let $k(p/q)$ be a $2$-bridge knot such that $\Sigma^2(k (p/q) ) = L(p,q)$, where $\Sigma^2(k)$ is the double branched cover of $k \subset S^3$. Then any oriented homology $3$-sphere given in \cref{cal3} with $6n-1>p$ cannot be a Seifert hypersurface of $K(k(p/q), 2)$. 
\end{cor}
\begin{proof}This is a corollary of \cref{cal3} and \cref{OBS}.  \qed
\end{proof}

\subsubsection{Embedding from homology $3$-spheres into $4$-manifolds}
In this section, we treat the existence problem of embeddings from $3$-manifolds into $4$-manifolds. This problem has been studied in several situations (\cite{H96}, \cite{GL83}, \cite{K88}, \cite{H09}, \cite{D15}). In this section, we develop a gauge theoretic method and focus on a certain class of homology $3$-spheres and negative definite $4$-manifolds. We give a relationship between existence of embeddings and $ SU(2)$-representations of fundamental groups. 
First, we prove \cref{emb3}. 
\begin{thm}[\cref{emb3}] Suppose $X$ is a negative definite $4$-manifold containing $\Sigma(2,3,5)$ as a submanifold. 
Then 
\[
\im\cs_{X,[-\Sigma(2,3,5)]} = \left\{ \frac{1}{120} , \frac{49}{120} ,1\right\} \subset (0,1].
\]
In particular, there exist at least four irreducible $SU(2)$-representations of $\pi_1(X)$. 
\end{thm}

To prove \cref{emb3}, we will show the following theorem: 
\begin{thm} \label{Image1}
Let $X$ be a negative definite $4$-manifold. 
Suppose that an oriented homology 3-sphere $Y$ is embedded in $X$ and there exists $k_1, \cdots k_m \in \Z_{>0}$ and $s_1, \cdots s_n \in [-\infty ,0]$ such that
\[
\bigcup_{k_1, \cdots k_m, s_1, \cdots s_n}  (r_s(Y) \cup \Gamma_{-Y}(k) ) = \Lambda_Y \cap (0,1)\]
 and $l^{s_i}_Y= l^{l_i}_Y=1$ for any $i$. Then, 
 \[
 \im\cs_{X,[Y]}^j = \Lambda_Y \cap (0,1)
 \]
 for any $j \in \Z_{>0}$.  In particular, for any other embedding $Y' \subset X$ and $r \in  \Lambda_Y \cap (0,1)$ with $[Y']=[Y]$, there exists an $SU(2)$-representation $\rho_r$ on $Y'$ such that 
 $\cs_{Y'} ( \rho_r ) =r$.
\end{thm}
\begin{proof}
By \cref{emb1} and \cref{emb2}, 
\[
 \Lambda_Y \cap (0,1]= \bigcup_{k_1, \cdots k_m, s_1, \cdots s_n}  (r_s(Y) \cup \Gamma_{-Y}(k) )  \subset \im \cs^1_{ X,[Y]} \subset  \Lambda_Y \cap (0,1].
\]
This implies that $\im \cs^1_{ X,[Y]} =  \Lambda_Y \cap (0,1]$. We use \cref{fund prop} and obtain 
\[
  \Lambda_Y \cap (0,1] = \im \cs^1_{ X,[Y]} \subset  \im \cs^j_{ X,[Y]} \subset  \Lambda_Y \cap (0,1].
  \]
  This completes the proof. \qed
\end{proof}
Here we give a proof of \cref{emb3}.

{\em Proof of \cref{emb3}. }
We check that $-\Sigma(2,3,5)$ satisfies the assumption of \cref{Image1}. The proof is written in \cref{-235}. 
\qed

\begin{thm}[\cref{emb4}]Let $Y$ be a Seifert homology 3-sphere of type $\Sigma(a_1, \cdots , a_n)$ with 
\[
\Lambda^*_{\Sigma(a_1, \cdots , a_n)} \cap \Z = \emptyset. \ \  \footnote{This condition can be seen as a combinatorial condition via  \eqref{cs_Yof Seifert} }
\] Suppose the Fr\o yshov invariant $h(Y)$ of $Y$ is non-zero. 
Then $Y$ cannot be embedded in any negative definite $4$-manifold $X$ such that the $SU(2)$-representation space $R(X_{j,c})$ of $X_{j,c}$ is connected for any $j$. 
\end{thm}
{\em Proof of \cref{emb4}.} It is shown that $\Gamma_{\Sigma(a_1, \cdots , a_n)} (1)<\infty$ if $h(\Sigma(a_1, \cdots, a_n )) >0$ in \cite{D18}. Since the $SU(2)$-Chern-Simons functional of $Y=\Sigma(a_1, \cdots , a_n)$ is Morse-Bott, $l_Y < \infty$. By \cref{emb1}, there exists $l \leq l_Y$ such that
 \[
r_{-\infty} (Y)- \lfloor r_{-\infty} (Y) \rfloor= \Gamma_{Y}(1)- \lfloor  \Gamma_{Y}(1) \rfloor \in \im \cs^l_{X, [\Sigma(a_1, \cdots , a_n)]}.
\]
Here, $\Gamma_Y(1) \in \Lambda^*_Y$ and the formula \eqref{cs_Yof Seifert} implies that $1 \neq \Gamma_Y(1) - \lfloor  \Gamma_{Y}(1) \rfloor \in (0,1]   $. This then implies that 
\[
 \im \cs^l_{X, [\Sigma(a_1, \cdots , a_n)]} \cap (0,1) \neq \emptyset. 
 \]
 Suppose that $R(X_{j,c})$ is connected.  Then, since $\cs_{X, c}^j$ is locally constant, 
 we see that $\im \cs_{X,c}^j = \{1\}$.
\qed

As a corollary of \cref{perfect}, we can show that several $SU(2)$-representations on some classes of Seifert homology $3$-spheres are extendable. 
\begin{cor}Let $(X, Y)$ be a closed connected negative definite $4$-manifold $X$ and an oriented codimension-$1$-submanifold $Y$ of $X$ with $H_*(Y; \Z) \cong H_*(S^3; \Z)$ and $\rho : \pi_1(Y)  \to SU(2)$ be an $SU(2)$-representation of $\pi_1(Y)$. 
\begin{itemize}
\item Suppose $k$ is odd and $n$ is a positive integer.
If $Y=  -\Sigma(2,3,6k-1)$ and $\rho$ is a representation corresponding to $\rho_{(1,2,5k-1)}$ in \eqref{236k-1}, then $\rho$ is extendable for $(X,Y)$.
\item Suppose $k$ is even $n$ is a positive integer. 
If $Y= -\Sigma(2,3,6k-1)$ and $\rho$ is a representation corresponding to $\rho_{(1,2,k)}$ in \eqref{236k-1}, then $\rho$ is extendable for $(X,Y)$.
\item  
If $Y = -\Sigma(2,3,5,7)$ and $\rho$ is any representation corresponding to $(1,1,2,4)$, then $\rho$ is extendable for $(X,Y)$.
\end{itemize}
\end{cor}
\begin{proof}The following calculations have been checked in \cref{Examples}:
\begin{itemize}
\item If $k$ is odd, then $l^s_{ -\Sigma(2,3,6k-1)} =1$, $\cs_Y\left(\rho_{(1,2,5k-1)} \#\theta \# \cdots \# \theta \right) = \frac{1}{24(6k-1)} = r_s( -\Sigma(2,3,6k-1))$ and $\cs^{-1}\left(\frac{1}{24(6k-1)} \right)$ is the one point. 
\item If $k$ is even, then $l^s_{ -\Sigma(2,3,6k-1)} =1$, $\cs_Y\left(\rho_{(1,2, k)} \# \theta \cdots \# \theta \right) = \frac{1}{24(6k-1)} = r_s(-\Sigma(2,3,6k-1))$ and $\cs^{-1}\left(\frac{1}{24(6k-1)} \right)$ is the one point. 
\item  $l^s_{ -\Sigma(2,3,5,7)} =1$, $\cs_Y\left(\rho_{(1,1, 2,4)} \# \theta \# \cdots \# \theta\right) = \frac{1}{840} = r_s(- \Sigma(2,3,5,7))$ and $\cs^{-1}_Y\left(\frac{1}{840} \right) \cap \wt{R}^*(Y) \cong S^2$. Off course, $S^2$ has a perfect Morse function. 
\end{itemize}
Therefore, in each case, we can apply \cref{perfect}. \qed
\end{proof}
In the case of $Y = -\Sigma(2,3,5,7)$, the $SU(2)$-representations corresponding to $(1,1,2,4)$ are parametrized by $S^2$. Therefore, one can prove \cref{emb5}.

\begin{thm} [\cref{emb5}]Let $X$ be a closed definite $4$-manifold $X$ containing $\dis \Sigma(2,3,5,7)$ as a submanifold. 
Then there is an $S^2$-component $C$ in $R^*(\dis \Sigma(2,3,5,7))$ such that all elements in $C$ can be extended to $X$. In particular, there exists an uncountable family of irreducible $SU(2)$-representations of $\pi_1(X)$. 
\end{thm}

On the other hand, there are several homology $3$-spheres $Y$ such that no irreducible representation of $Y$ is extendable for $(S^4, Y)$.
\begin{ex}
Let $Y$ be an oriented homology $3$-sphere embedded into $S^4$. 
Then no irreducible representation is extendable for $(S^4, Y)$. For example, if $p$ is odd and $k$ is a positive integer, it is shown in \cite{CH81} that $\Sigma(p, pk+1, pk+2)$ can be embedded in to $S^4$. Thus every irreducible $SU(2)$-representation of $\Sigma(p, pk+1, pk+2)$ is not extendable for $(S^4, Y)$.
\end{ex}

 \subsubsection{Fixed point theorems}
 We first prove \cref{emb6}.
\begin{thm}[\cref{emb6}]
Let $Y$ be an oriented homology $3$-sphere and let $h$ be an orientation preserving self-diffemorphism of $Y$.
\begin{enumerate}
\item 
 If $r_s(Y)<\infty$ and $l^s_Y<\infty$ for some $s \in [-\infty, 0]$, then there exists a positive number $l \leq l^s_Y$ such that 
\[
(h^* )^l \colon R^*(Y) \to R^*(Y)
\]
 has a fixed point. 
 \item  If $\Gamma_{-Y}(k)<\infty$ and $l^k_Y<\infty$ for some $k \in \Z_{>0}$, then there exists a positive number $l \leq l^k_Y$ such that 
\[
(h^* )^l \colon R^*(Y) \to R^*(Y)
\]
 has a fixed point. 
 \end{enumerate}
\end{thm}

{\em Proof of \cref{emb6}. } We only prove the first one since the second one is similar. We apply \cref{emb1} for the mapping torus $X_h(Y)$ of $h\colon Y \to Y$. Then we obtain the inclusion 
\[
r_s(Y) \in \bigcup_{1\leq j \leq l_Y^s} \im\cs_{X_h(Y),[Y]}^j .
\]
Therefore, there exists some $l \in \{1, \cdots , l^s_Y\}$ such that $R^*(X_h(Y)_{l ,[Y]} )$ is non-empty. Note that $ X_h(Y)_{l ,[Y]}$ is diffeomorphic to $X_{h^l}(Y)$ so we can identify $R^*(X_h(Y)_{l ,[Y]} )$ with $R^*(X_{h^l}(Y))$. Since  $R^*(X_h(Y)_{l ,[Y]} )$ is non-empty, we obtain an element of $R^*(X_{h^l}(Y))$. This gives a fixed point of $(h^l)^*$.
 \qed
 \begin{cor}
 Fix a Seifert homology $3$-sphere $\Sigma(a_1, \cdots, a_n)$ with $R(a_1, \cdots , a_n)>0$ and $\Lambda^*_{-\Sigma(a_1, \cdots, a_n)} \cap \Z = \emptyset$. Let $Y$ be an oriented homology $3$-sphere such that
  \begin{align}\label{00}
 \min \set { a+b | a\in \Lambda^*_{-\Sigma(a_1, \cdots, a_n)}  \cup \Z  , b \in \Lambda^*_Y, a+b \geq 0}  > \frac{1}{4a_1\cdots  a_n}
 \end{align} 
 and 
 \begin{align}\label{-1}
  \Lambda^*_Y \cap \Z = \emptyset. 
 \end{align}
 Then, for any orientation preserving diffeomorphism $h$ on $-\Sigma(a_1, \cdots, a_n)\# Y$, there exists some $l\in \Z_{>0}$ such that \[
  (h^l)^* \: R^*(-\Sigma(a_1, \cdots, a_n)\# Y) \to R^*(-\Sigma(a_1, \cdots, a_n)\# Y)
\]
   has a fixed point. 
 \end{cor}
 \begin{proof}It is proved in \cite[Corollary 1.4]{NST19} that for a Seifert homology 3-sphere $\Sigma(a_1, \cdots, a_n)$ satisfying $R(a_1, \cdots , a_n)>0$, it is proved that $r_0(-\Sigma(a_1, \cdots, a_n)) = \frac{1}{4 a_1\cdots  a_n}$. 
 In particular, $\frac{1}{4 a_1\cdots  a_n} \in \Lambda^*_{-\Sigma(a_1, \cdots, a_n)}$.
 Set $M:=-\Sigma(a_1, \cdots, a_n)\# Y$.
 First, we use the connected sum formula for $r_0$:
 \begin{align}\label{11}
 \frac{1}{4a_1\cdots  a_n}=r_0(-\Sigma(a_1, \cdots, a_n) ) \geq \min \{ r_0(M) , r_0(-Y)\}.
 \end{align}
 Since $r_0$ is contained in the set of critical values of irreducible $SU(2)$-flat connections, 
 \begin{align}\label{22}
 r_0(M)  \in  \set { a+b | a\in \Lambda_{-\Sigma(a_1, \cdots, a_n)}   , b \in \Lambda_Y, a+b > 0   } . 
 \end{align}
 Here we used $\cs_{Y_1\# Y_2} ( \rho_1 \# \rho_2 ) = \cs_{Y_1} ( \rho_1 )+ \cs_{Y_2} ( \rho_2 )$ for oriented 3-manifolds $Y_1$ and $Y_2$ and $SU(2)$-representations $\rho_1$ and $\rho_2$. 
 On the other hand,  by the formula \eqref{cs_Yof Seifert} for critical values of $\cs_{-\Sigma(a_1, \cdots, a_n)}$ and $\Lambda^*_{-\Sigma(a_1, \cdots, a_n)} \cap \Z = \emptyset$, we see
 \begin{align}\label{33}
 \min \set { a | a\in \Lambda^*_{-\Sigma(a_1, \cdots, a_n)} \cap \R_{\geq 0} } =  \frac{1}{4a_1\cdots  a_n}. 
 \end{align}
 If $r_0(M)= a+b \leq  \frac{1}{4a_1\cdots  a_n} $ (we used \eqref{11} and \eqref{22}) for $a\in \Lambda^*_{-\Sigma(a_1, \cdots, a_n)}  \cup \Z$ , $b \in \Lambda^*_Y \cup \Z$, then $b=0$ by our assumption \eqref{00}. 
Then \eqref{33} implies $a=   \frac{1}{4a_1\cdots  a_n}$. 
  Moreover, by the condition \eqref{-1}, we have
 \[
 \cs_{M}^{-1} ( \frac{1}{4a_1\cdots  a_n}) =  \{ \rho \# \theta_{-Y} \in \wt{R}(M) | \cs_{-\Sigma(a_1, \cdots, a_n)} ( \rho) =  \frac{1}{4a_1\cdots  a_n}, \rho \in \wt{R}(-\Sigma(a_1, \cdots, a_n))  \}, 
 \]
 where $\theta_{-Y}$ is the product connection on $-Y$.
 Thus the Chern-Simons functional of $M$ is Morse-Bott at the level $r_0(M)=  \frac{1}{4a_1\cdots  a_n}$. By using \cref{finiteness}, we conclude $l^0_M$ is finite. One can then apply \cref{emb6} to complete the proof. 
 \qed
 \end{proof}
 At the end of this section, we prove \cref{emb7}. 
 \begin{thm}[\cref{emb7}]
 For any orientation preserving diffeomorphism $h$ on $\Sigma (2,3,5,7)$, the fixed point set of 
\[
h^* : R^*( \Sigma (2,3,5,7)) \to R^*(  \Sigma (2,3,5,7))
\]
is uncountable. 
 \end{thm}
 {\em Proof of \cref{emb7}.}
 \cref{emb5} implies that, for any orientation preserving diffeomorphism $h$, $X_h(Y)$ has an uncountable family of irreducible $SU(2)$-representations. Thus \eqref{two to one} implies the conclusion.  \qed

\bibliography{tex}
\bibliographystyle{hplain}

\noindent\textsc{
RIKEN, the Institute of Physical and Chemical Research, 
\\
2-1 Hirosawa Wako,
Saitama 351-0198, 
Japan}
\\ \\
\noindent{E-mail:\ masaki.taniguchi@riken.jp}

\end{document}